\documentclass[bibother]{asl}
\usepackage[active]{srcltx}
\usepackage{atveryend}
\makeatletter
\let\origcitation\citation
\AtEndDocument{\def\mycites{}
 \def\citation#1{\g@addto@macro\mycites{#1^^J}\origcitation{#1}}}
\AtVeryEndDocument{\newwrite\citeout\immediate\openout\citeout=\jobname.cit
 \immediate\write\citeout{\mycites}\immediate\closeout\citeout}
\makeatother
\usepackage{enumerate}
\usepackage{cancel}
\usepackage{mathdots}
\usepackage{epsfig}
\usepackage{hyperref}
\usepackage{amssymb}
\usepackage{graphics} 
\usepackage{verbatim}
\usepackage{fullpage}
\usepackage{yfonts}
\usepackage[usenames,dvipsnames]{color}

\newcommand{\mfM}{\mathfrak{M}}
\renewcommand{\k}{\boldsymbol{k}}

\newcommand{\fm}{\mathfrak{m}}
\newcommand{\fn}{\mathfrak{n}}
\newcommand{\C}{{\mathbb C}}
\newcommand{\D}{{\mathbb D}}
\newcommand{\N}{{\mathbb N}}
\newcommand{\Q}{{\mathbb Q}}
\newcommand{\R}{{\mathbb R}}
\newcommand{\T}{{\mathbb T}}
\newcommand{\Z}{{\mathbb Z}}
\newcommand{\cC}{{\mathcal C}}
\newcommand{\cF}{{\mathcal F}}
\newcommand{\cL}{{\mathcal L}}
\newcommand{\inv}{^{-1}}
\newcommand{\df}{^{\mathrm{df}}}
\newcommand{\an}{_\mathrm{an}}
\newcommand{\anexp}{_\mathrm{an,exp}}
\newcommand{\No}{\mathbf{No}}
\newcommand{\On}{\mathbf{On}}
\newcommand{\Oz}{\mathbf{Oz}}
\newcommand{\PI}{\operatorname{PI}}
\newcommand{\rc}{\operatorname{rc}}
\newcommand{\pr}{\operatorname{pr}}

\newcommand{\supp}{\operatorname{supp}}
\newcommand{\uf}{_\mathcal{F}}
\newcommand{\ufexp}{_{\mathcal{F},\operatorname{exp}}}

\renewcommand{\leq}{\leqslant}
\renewcommand{\geq}{\geqslant}
\renewcommand{\preceq}{\preccurlyeq}
\renewcommand{\succeq}{\succcurlyeq}
\renewcommand{\epsilon}{\varepsilon}

\newtheorem{corollary}{Corollary}
\newtheorem{fact}{Fact}
\newtheorem{proposition}{Proposition}
\newtheorem{theorem}{Theorem}
\newtheorem*{theorem*}{Theorem}
\newtheorem{lemma}{Lemma}
\newtheorem*{open question*}{Open Question}

\numberwithin{question}{section}
\numberwithin{conjecture}{section}
\numberwithin{corollary}{section}
\numberwithin{fact}{section}
\numberwithin{proposition}{section}
\numberwithin{theorem}{section}
\numberwithin{lemma}{section}
\newtheorem*{notational conventions*}{Notational Conventions}

\theoremstyle{definition}
\newtheorem{definition}{Definition}
\numberwithin{definition}{section}
\newtheorem*{example*}{Example}
\newtheorem{example}{Example}
\numberwithin{example}{section}
\newtheorem{remark}{Remark}
\newtheorem*{remark*}{Remark}
\numberwithin{remark}{section}

\newcommand{\thistheoremname}{}
\newtheorem*{genericthm*}{\thistheoremname}
\newenvironment{namedthm*}[1]
{\renewcommand{\thistheoremname}{#1}%
\begin{genericthm*}}
{\end{genericthm*}}

\title[Surreal Ordered Exponential Fields]{Surreal Ordered Exponential Fields}


\author{Philip Ehrlich}
\revauthor{Ehrlich, Philip}

\address{Department of Philosophy\\
Ohio University\\
Athens, OH 45701, USA}

\email{ehrlich@ohio.edu}

\author{Elliot Kaplan}
\revauthor{Kaplan, Elliot}

\address{Department of Mathematics\\
University of Illinois at Urbana-Champaign\\
Urbana, IL 61801, USA}

\email{eakapla2@illinois.edu}
\begin{document}

\begin{abstract} 
In~\cite{EH5}, the algebraico-tree-theoretic simplicity hierarchical structure of J. H. Conway's ordered field $\No$ of surreal numbers was brought to the fore and employed to provide necessary and sufficient conditions for an ordered field (ordered $K$-vector space) to be isomorphic to an initial subfield ($K$-subspace) of $\No$, i.e. a subfield ($K$-subspace) of $\No$ that is an initial subtree of $\No$. In this sequel to~\cite{EH5}, piggybacking on the just-said results, analogous results are established for \emph{ordered exponential fields}, making use of a slight generalization of Schmeling's conception of a \emph{transseries field}. It is further shown that a wide range of ordered exponential fields are isomorphic to initial exponential subfields of $(\No, \exp)$. These include all models of $T(\R_W, e^x)$, where $\R_W$ is the reals expanded by a \emph{convergent Weierstrass system} $W$. Of these, those we call \emph{trigonometric-exponential fields} are given particular attention. It is shown that the exponential functions on the initial trigonometric-exponential subfields of $\No$, which includes $\No$ itself, extend to \emph{canonical} exponential functions on their \emph{surcomplex} counterparts. The image of the canonical map of the ordered exponential field $\T^{LE}$ of \emph{logarithmic-exponential} transseries into $\No$ is shown to be initial, as are the ordered exponential fields $\R((\omega))^{EL}$ and $\R\langle\langle\omega\rangle \rangle$.
\end{abstract} 

\maketitle
\setcounter{tocdepth}{1}
\tableofcontents

%------------------------------------------------%
\section{Introduction} \label{sec:In}
%------------------------------------------------%
In his monograph \emph{On Numbers and Games}~\cite{CO}, J. H. Conway introduced a real closed field $\No$ of \emph{surreal numbers} containing the reals and the ordinals as well as a great many less familiar numbers, including $-\omega$, $\omega/2$, $1/\omega$, and $\sqrt{\omega}$, to name only a few. Indeed, $\No$ is so remarkably inclusive that, subject to the proviso that numbers---construed here as members of ordered fields---be individually definable in terms of sets of NBG (von Neumann-Bernays-G\"odel set theory with Global Choice), it may be said to contain ``All Numbers Great and Small''~\cite{EH1, EH2, EH5, EH8}.

$\No$ also has a rich algebraico-tree-theoretic structure which was brought to the fore by Ehrlich~\cite{EH4, EH5} and further developed and explored in~\cite{vdDE, EH6, EH7, EH8, EK, F1, F2}. This \emph{simplicity hierarchical} (or \emph{$s$-hierarchical}) \emph{structure} depends upon $\No$'s structure as a lexicographically ordered full binary tree and arises from the fact that the sums and products of any two members of the tree are the simplest possible elements of the tree consistent with $\No$'s structure as an ordered group and an ordered field, respectively, it being understood that $x$ is \emph{simpler than} $y$ just in case $x$ is a predecessor of $y$ in the tree. 

Among the remarkable $s$-hierarchical features of $\No$ is that much as the surreal numbers emerge from the empty set of surreal numbers by means of a transfinite recursion that provides an unfolding of the entire spectrum of numbers great and small (modulo the aforementioned provisos), the recursive process of defining $\No$'s arithmetic in turn provides an unfolding of the entire spectrum of ordered fields (ordered $K$-vector spaces; ordered abelian groups) in such a way that an isomorphic copy of every such system either emerges as an initial substructure of $\No$--a substructure of $\No$ that is an initial subtree of $\No$--or is contained in a theoretically distinguished instance of such a system that does. More specifically, in~\cite{EH5} Ehrlich showed that: 

\begin{proposition}
\label{Prop. 2}
Every ordered vector space over an Archimedean ordered field is isomorphic to an initial subspace of $\No$; in particular, every divisible ordered abelian group is isomorphic to an initial subgroup of $\No$. 
\end{proposition}

\begin{proposition}
\label{Prop. 1}
Every real closed ordered field is isomorphic to an initial subfield of $\No$. 
\end{proposition}

These results were obtained with the aid of the following more general results from~\cite{EH5} that provide necessary and sufficient conditions for an ordered $K$-vector space (ordered field) to be isomorphic to an initial $K$-subspace (subfield) of $\No$.

\begin{proposition}
\label{KV}
An ordered $K$-vector space is isomorphic to an initial subspace of $\No$ if and only if $K$ is isomorphic to an initial subfield of $\No$.
\end{proposition}

\begin{proposition}\label{IF}
An ordered field $K$ is isomorphic to an initial subfield of $\No$ if and only if $K$ is isomorphic to a truncation closed, cross sectional subfield of a power series field $\R((t^\Gamma))_{\On}$ where $\Gamma$ is isomorphic to an initial subgroup of $\No$.
\end{proposition}

In~\cite{vdDE}, van den Dries and Ehrlich subsequently established:

\begin{proposition} 
\label{DE2}
The exponential field $(\No, \exp)$ of surreal numbers is
an elementary extension of the exponential field $(\R, e^{x})$ of real numbers, where $\exp$ is the recursively defined exponential function on $\No$ developed by Kruskal~\cite[page 38]{CO} and Gonshor~\cite[Ch. 10]{GO}. 
\end{proposition}

This result is obtained as a corollary of:

\begin{proposition} 
\label{DE1}
The field of surreal numbers equipped with 
restricted analytic functions (defined via Taylor series expansion) and with $\exp$ is an elementary
extension of the field of real numbers with restricted analytic
functions and real exponentiation. 
\end{proposition} 

Like a recent related work by the current authors~\cite{EK}, this is a sequel to~\cite{EH5}. Following some preliminary material in \S\ref{sec:P1}--\S\ref{sec:P3}, and piggybacking on Propositions~\ref{Prop. 2}--\ref{IF}, necessary and sufficient conditions are established for an \emph{ordered exponential field} to be isomorphic to an initial exponential subfield of $\No$. These conditions make use of the conception of a \emph{transserial Hahn field}, which is an adaptation for sets and proper classes of Schmeling's~\cite{SC} conception of a \emph{transseries field} (see Remark~\ref{rem-T4.2}). It is shown that an ordered exponential field $K$ is isomorphic to an initial exponential subfield of $\No$ if and only if $K$ is isomorphic to a truncation closed, cross sectional exponential subfield of a transserial Hahn field $\R((t^\Gamma))_{\On}$. As a byproduct of this result, we show that every transserial Hahn field admits a truncation closed logarithmic field embedding into $\No$, but that this embedding can not in general be taken to be initial. In light of the growing body of work relating transseries and surreal numbers~\cite{ADH2,Avv,BH,BH2,BHK,BHM,BM,BM2,CE,CEF}, this result on embeddings of transserial Hahn fields is of independent interest.

Using the characterization above, it is further shown that a wide range of ordered exponential fields are isomorphic to initial exponential subfields of $\No$. These include all models of the theory $T (\R\an,e^{x})$ of real numbers with restricted analytic functions and exponentiation~\cite{DMM}, a result previously established by Fornasiero~\cite{F2}, and, more generally, all models of the theory $T(\R_W, e^x)$ of real numbers with a \emph{convergent Weierstrass system} $W$~\cite{DL, vdD} and exponentiation. Of these, those we call \emph{trigonometric-exponential fields} are found to be of particular significance. More specifically, it is shown that the exponential functions on initial trigonometric-exponential subfields of $\No$, which includes $\No$ itself, extend to \emph{canonical} exponential functions on their \emph{surcomplex} counterparts. The proof of this uses the precursory result that trigonometric-exponential initial subfields of $\No$ and \emph{trigonometric ordered initial subfields} of $\No$, more generally, admit \emph{canonical} sine and cosine functions. This is shown to apply to the members of a distinguished family of initial exponential subfields of $\No$ isolated by van den Dries and Ehrlich (\cite{vdDE}, Corollary 5.5), to the image of the canonical map of the ordered exponential field $\T^{LE}$ of \emph{transseries} into $\No$~\cite{Avv}, which is shown to be initial, and to the ordered exponential fields $\R((\omega))^{EL}$ and $\R\langle\langle\omega\rangle \rangle$, considered by Berarducci and Mantova (\cite{BM2}; also see~\cite{K, KS, KT}), which are likewise shown to be initial.

Some of the methods employed in \S\ref{sec:TF}--\S\ref{sec:WS} are adaptations or expansions of methods developed by Ressayre \cite{R}, van den Dries, Macintyre and Marker~\cite{DMM, DMM1}, and D'Aquino, Knight, Kuhlmann and Lange~\cite{D} in their treatments of truncation closed embeddings into Hahn fields of ordered exponential fields and ordered fields with additional structure more generally. However, as is evident from Proposition~\ref{IF}, even in the case of an ordered field $K$ the existence of a truncation closed embedding into a Hahn field does not suffice to establish the existence of an initial embedding into $\No$. This inadequacy is even more pronounced in the case of an initial embedding of an ordered exponential field $K$ into $\No$, since besides the properties required for an initial embedding of $K$ as an ordered field, conditions must be placed on the logarithmic or exponential structure of $K$. As it turns out, the condition T4 from Schmeling's thesis~\cite{SC} on transseries plays an essential role. We thank the anonymous referee for pointing us towards this connection (see Remark~\ref{rem-thanks}), and for their many other helpful suggestions which have greatly improved the paper.

Throughout the paper the underlying set theory is NBG, which is a conservative extension of ZFC in which all proper classes are in bijective correspondence with the class of all ordinals (cf.~\cite{ME}). By ``set-model'' (``class-model'') we mean a model whose universe is a set (a proper class). The theories in the languages treated in \S\ref{sec:WS} and \S\ref{sec:TE} admit quantifier elimination, with the consequence that the results in those sections regarding class-models of such theories and elementary embeddings of models into class-models of such theories are provable in NBG. For this and details on formalizing the theory of surreal numbers in NBG more generally, see~\cite{EH2}.

%------------------------------------------------%
\section{Preliminaries I: surreal numbers}\label{sec:P1}
%------------------------------------------------%
By a \textbf{tree} $(A, < _s)$ we mean a partially ordered class such that for each $x \in A$, the class $\{ y \in A:y < _s x\}$ of \textbf{predecessors} of $x$, written `$\pr_A (x)$', is a set well-ordered by $ < _s $. A maximal subclass of $A$ well-ordered by $ < _s $ is called a \textbf{branch} of the tree. Two elements $x$ and $y$ of $A$ are said to be \textbf{incomparable with respect to $<_s$} if $x \ne y$, $x\not < _s y$ and $y\not < _s x$. An \textbf{initial subtree} of $(A, < _s)$ is a subclass $A'$ of $A$ with the induced order such that for each $x \in A'$, $\pr_{A'} (x) = \pr_A (x)$. The \textbf{tree-rank} of $x \in A$, written `$\rho _A (x)$', is the ordinal corresponding to the well-ordered set $(\pr_A (x), <_s)$; the $\alpha $th \textbf{level} of $A$ is $\big\{ x \in A:\rho _A (x) = \alpha \big\}$; and a \textbf{root} of $A$ is a member of the zeroth level. If $x, y \in A$, then $y$ is said to be an \textbf{immediate successor} of $x$ if $x < _s y$ and $\rho _A (y) = \rho _A (x) + 1$; and if $(x_\alpha)_{\alpha < \beta }$ is a chain in $A$ (i.e., a subclass of $A$ totally ordered by $ < _s $), then $y$ is said to be an \textbf{immediate successor of the chain} if $x_\alpha < _s y$ for all $\alpha < \beta$ and $\rho _A (y)$ is the least ordinal greater than the tree-ranks of the members of the chain. The \textbf{length} of a chain $(x_\alpha)_{\alpha < \beta }$ in $A$ is the ordinal $\beta $.

A tree $(A, < _s)$ is said to be \textbf{binary} if each member of $A$ has at most two immediate successors and every chain in $A$ of limit length has at most one immediate successor. If every member of $A$ has two immediate successors and every chain in $A$ of limit length (including the empty chain) has an immediate successor, then the binary tree is said to be \textbf{full}. Since a full binary tree has a level for each ordinal, the universe of a full binary tree is a proper class.

Following~\cite[Definition 1]{EH5}, a binary tree $(A, < _s)$ together with a total ordering $ < $ defined on $A$ will be said to be \textbf{lexicographically ordered} if for all $x, y \in A$, $x$ is incomparable with $y$ with respect to $<_s$ if and only if $x$ and $y$ have a common predecessor lying between them (i.e. there is a $z \in A$ such that $z<_s x$, $z<_s y$ and either $x<z<y$ or $y<z<x$). The appellation ``lexicographically ordered'' is motivated by the fact that: $(A, <, < _s)$ is a lexicographically ordered binary tree if and only if $(A, <, < _s)$ is isomorphic to an initial ordered subtree of the \textbf{lexicographically ordered canonical full binary tree} $(B, < _{lex(B)}, < _B)$, where $B$ is the class of all sequences of $ - $'s and $ + $'s indexed over some ordinal, $x < _B y$ signifies that $x$ is a proper initial subsequence of $y$, and $(x_\alpha)_{\alpha < \mu } < _{lex(B)} (y_\alpha)_{\alpha < \sigma }$ if and only if $x_\beta = y_\beta $ for all $\beta < $ some $\delta $, but $x_\delta < y_\delta$, it being understood that $ - < \textit{undefined} < + $~\cite[Theorem 1]{EH5}. Note that every binary tree admits a lexicographic ordering.

Let $(A, <, < _s)$ be a lexicographically ordered binary tree. If $(L, R)$ is a pair of subclasses of $A$ for which every member of $L$ precedes every member of $R$, then we will write `$L < R$'. Also, if $x$ and $y$ are members of $A$, then `$x < _s y$' will be read ``$x$ \textbf{is simpler than} $y$''; and when $x \in I = \big\{ y \in A:L < \{ y \} < R\big\}$ is such that $x < _s y$ for all $y \in I \setminus \{ x \}$, then we will denote this \textbf{simplest member of $A$ lying between the members of $L$ and the members of $R$} by `$\{ L\, | \, R\}$'. Following Conway's game-theoretic terminology, the members of $L$ and $R$ are called the \textbf{left options} of $x$ and the \textbf{right options} of $x$ respectively. For all $x \in A$, by `$L_{s(x)} $' we mean $\{ a \in A:a < _s x\text{ and }a < x \}$ and by `$R_{s(x)} $' we mean $\{ a \in A:a < _s x\text{ and }x < a \}$. 

The following proposition collects together a number of properties of, or results about, lexicographically ordered binary trees that will be appealed to in subsequent portions of the paper.

\begin{proposition}[\cite{EH5}, Theorem 2; \cite{EH7}, Proposition 2.3]\label{Tree 1}
Let $(A, <, < _s) $ be a lexicographically ordered binary tree.
\begin{enumerate}[(i)]
\item For all $x \in A$, $x = \{ L_{s(x)} \, | \, R_{s(x)} \}$;
\item for all $x, y \in A$, $x < _s y$ if and only if $L_{s(x)} < \{ y \} < R_{s(x)} $ and $y \ne x$; 
\item for all $x \in A$ and all $L, R \subseteq A$ with $L < \{x\} < R$, $x = \{ L \, | \, R \}$ if and only if $L$ is cofinal with $L_{s(x)} $ and $R$ is coinitial with $R_{s(x)}$ if and only if $\{ y \in A: L < \{y\} < R\} \subseteq \{ y \in A: L_{s(x)} < \{y\} < R_{s(x)}\}$\footnote{Due to an inadvertent omission, the necessary assumption that $L < \{x\} < R$ is missing in from~\cite[Proposition 2.3]{EH7} and \cite[Proposition 2.2]{EK}}.
\end{enumerate}
\end{proposition}

Let $(\No, <, < _s)$ be the \textbf{lexicographically ordered binary tree of surreal numbers} constructed in any of the manners found in the literature (cf.~\cite{EH4, EH5, EH6, EH8}), including simply letting $(\No, <, < _s) =(B, < _{lex(B)}, < _B)$. Central to the development of the $s$-hierarchical theory of surreal numbers is the following result where a lexicographically ordered binary tree $(A, <, < _s)$ is said to be \textbf{complete}~\cite[Definition 6]{EH5} if whenever $L$ and $R$ are sub\emph{sets} of $A$ for which $L < R$, there is an $x \in A$ such that $x = \{ L\, | \, R \}$. 

\begin{proposition}[\cite{EH5}, Theorem 4 and Proposition 2]\label{Tree 2}
A lexicographically ordered binary tree is complete if and only if it is full if and only if it is isomorphic to $(\No, <, < _s)$.
\end{proposition}

$\No$'s canonical class $\On$ of ordinals consists of the members of the ``rightmost'' branch of $(\No, <, < _s)$, i.e. the unique branch of $(\No, <, < _s)$ whose members satisfy the condition: $x<y$ if and only if $x<_sy$. 

By a \textbf{cut} in an ordered class $(A, <)$ we mean a pair $(X, Y)$ of subclasses of $A$ where $X<Y$ and $X \cup Y=A$. If $A \subsetneq A'$, $(X, Y)$ is a cut in $A$, and $X<\{z\}<Y$ where $z \in A' \setminus A$, we say $z$ \textbf{realizes the cut} $(X, Y)$. If $z$ realizes a cut in $A$, on occasion we denote the cut by $(A^{<z}, A^{>z})$ where $A^{<z}= \{a\in A:a<z\}$ and $A^{>z}= \{a\in A:a>z\}$.
 For $A\subsetneq \No$ and for $x \in \No$, we say that \textbf{$x$ is the simplest element realizing a cut in $A$} if $x\not\in A$ and if $x = \{ A^{<x}\, |\, A^{>x} \}$. Note that if $A$ is an initial subclass of $\No$ and if $x$ is the simplest element realizing a cut in $A$, then $A \cup \{x\}$ is also initial.

%-----------------------%
\subsection{The ordered field ${\No}$}\label{sec:surreal field}
%-----------------------%
With the lexicographically ordered full binary tree $(\No, <, < _s)$ of surreal numbers now at hand, the $s$-hierarchical ordered field of surreal numbers may be introduced as follows, where in the definitions of $+$,
$-$ and $\cdot$, the set-theoretic brackets that enclose the 
sets of ``left-sided options'' and ``right-sided options'' are omitted in 
accordance with custom. 
\begin{theorem}[Conway 1976; Ehrlich 2001]\label{theor12}
$(\No,+,\cdot,<,<_{s})$ is an ordered field when $+$, $-$ and
$\cdot$ are defined by recursion as follows where $x^L$, $x^R$, $y^L$ and
$y^R$ are
understood to range over the members of $L_{s(x)}$, $R_{s(x)}$, $L_{s(y)}$
and $R_{s(y)}$, respectively.
\end{theorem}
\begin{namedthm*}{Definition of $x+y$}
\[
x+y\ =\ \{x^L+y,\, x+y^L\, |\, x^R+y,\, x+y^R\}.
\]
\end{namedthm*}
\begin{namedthm*}{Definition of $-x$}
\[
-x\ =\ \{-x^R\, |\, -x^L\}.
\]
\end{namedthm*}
\begin{namedthm*}{Definition of $xy$}
\[
xy\ =\ \{x^Ly +xy^L-x^Ly^L,\, x^Ry+xy^R-x^Ry^R\, |\, x^Ly+xy^R-x^Ly^R,\, x^Ry+xy^L-x^Ry^L\}.
\]
\end{namedthm*}

%-----------------------%
\subsection{Conway names}\label{sec:theorems}
%-----------------------%
Let $\D$ be the set of all surreal numbers having finite tree-rank, and
\[
\R\ =\ \D\cup \big \{ \{ L\, | \, R \}: (L, R)\text{ is a Dedekind gap in }\D \big \},
\]
whereby a Dedekind gap $(L, R)$ in $\mathbb{D}$ we mean a pair of nonempty subsets $L, R$ of $\D$ where $L<R$, $L \cup R =\D$, $L$ has no greatest member and $R$ has no least member.

The following result regarding the structure of $\R$ is essentially due to Conway~\cite[pages 12, 23-25]{CO}.

\begin{proposition}\label{P:Cb}
$\R$ (with $ +, -, \cdot $ and $ < $ defined \`{a} la ${\No}$) is isomorphic to the ordered field of real numbers defined in any of the more familiar ways, $\D$ being ${\No}$'s ring of dyadic rationals (i.e., rationals of the form $m/2^n $ where $m$ and $n$ are integers); 
\[
n\ =\ \{ 0,\dots,n - 1 \, | \, \varnothing \}\ \text{ and }\ - n\ =\ \{ \varnothing \, |\, - (n - 1),\dots,0\}
\]
for each positive integer $n$, $0 = \{ \varnothing \, | \, \varnothing \}$, and the remainder of the dyadics are the arithmetic means of their left and right predecessors of greatest tree-rank; e.g., $1/2 = \{ 0 \, | \, 1 \}$. 
\end{proposition}

$\R$ is the unique Dedekind complete initial subfield of $\No$. Henceforth, all references to the reals are understood to be references to $\R$. 

A striking $s$-hierarchical feature of $\No$ is that every surreal number can be assigned a canonical ``proper name'' that is a reflection of its characteristic $s$-hierarchical properties. These \textbf{Conway names} or \textbf{normal forms} are expressed as formal sums of the form 
\[
\sum_{\alpha < \beta } r_\alpha\omega ^{y_\alpha } 
\]
where $\beta $ is an ordinal, $(y_\alpha)_{\alpha < \beta } $ is a strictly decreasing sequence of surreals, and $(r_\alpha)_{\alpha < \beta } $ is a sequence of nonzero real numbers, the Conway name of an ordinal being just its Cantor normal form, it being understood that $0$ is the empty sum indexed over $\alpha<\beta=0$~\cite[pages 31-33]{CO} and~\cite[\S3.1 and \S5]{EH5}. 

Every nonzero surreal $x$ is the sum of three components, each of which can be succinctly characterized in terms of the Conway name of $x$: the \textbf{purely infinite} component of $x$, whose terms solely have positive exponents; the \textbf{real} component of $x$, whose sole term (if it is not the empty sum) has exponent $0$; and the \textbf{infinitesimal} component of $x$, whose terms solely have negative exponents. Notice that $0$, being the empty sum, may be regarded as purely infinite.

The surreal numbers having Conway names of the form $\omega^y$ are called \textbf{leaders} since they denote the simplest positive members of the various Archimedean classes of $\No$. More formally, they may be inductively defined by the formula
\[
\omega^{y}\ =\ \Big\{ 0,\, n\omega^{y^L}\, \Big|\, \frac{1}{2^n}\omega^{y^R} \Big\},
\]
where $n$ ranges over the positive integers and $y^L$ and $y^R$ range over the elements of $L_{s(y)}$ and $R_{s(y)}$, respectively. For use in the proof of Lemma~\ref{deltaatomicz}, we need the following lemma:

\begin{lemma}\label{lem-conwaytrunc}
Let $x = \sum_{\alpha < \beta } r_\alpha\omega ^{y_\alpha } \in \No$, let $\sigma<\beta$, let $z:=\sum_{\alpha < \sigma } r_\alpha\omega ^{y_\alpha }$, and let $\epsilon \in \{\pm1\}$ be the sign of $r_\sigma$. Then $z+\epsilon\omega^{y_\sigma} \leq_s x$.
\end{lemma}
\begin{proof}
By replacing $x$ with $-x$ if need be, we may assume that $r_\sigma$ is positive, so $\epsilon = 1$. It suffices to find a representation $\{L\, |\,R\}$ of $z+\omega^{y_\sigma}$ with $L<\{x\}<R$, for the result then follows by Proposition~\ref{Tree 1}. As above, we have
\[
\omega^{y_\sigma}\ =\ \Big\{ 0,\, n\omega^{y_\sigma^L}\, \Big|\, \frac{1}{2^n}\omega^{y_\sigma^R} \Big\},
\]
where $n$ ranges over the positive integers and $y_\sigma^L$ and $y_\sigma^R$ range over the elements of $L_{s(y_\sigma)}$ and $R_{s(y_\sigma)}$, respectively. We also have
\[
z\ =\ \Big\{ \sum_{\alpha < \sigma_0 } r_\alpha\omega ^{y_\alpha }+ q_L \omega^{y_{\sigma_0}} \, \Big|\, \sum_{\alpha < \sigma_0 } r_\alpha\omega ^{y_\alpha }+ q_R \omega^{y_{\sigma_0}} \Big\},
\]
where $\sigma_0$ ranges over ordinals less than $\sigma$ and $q_L$ and $q_R$ range over rational numbers with $q_L<r_{\sigma_0}<q_R$~\cite[Theorem 15]{EH5}. The formula for addition yields
\[
z+\omega^{y_\sigma}\ =\ \Big\{\sum_{\alpha < \sigma_0 } r_\alpha\omega ^{y_\alpha }+ q_L \omega^{y_{\sigma_0}}+\omega^{y_\sigma},\, z+n\omega^{y_\sigma^L}\, \Big|\, \sum_{\alpha < \sigma_0 } r_\alpha\omega ^{y_\alpha }+ q_R \omega^{y_{\sigma_0}}+\omega^{y_\sigma},\, z+\frac{1}{2^n}\omega^{y_\sigma^R}\Big\},
\]
where $n$, $y_\sigma^L$, $y_\sigma^R$, $\sigma_0$, $q_L$, and $q_R$ are as above. It is straightforward to verify that $x$ lies between the left and right options in the above representation, so $z+\omega^{y_\sigma}\leq_s x$.
\end{proof}

We note that Lemma~\ref{lem-conwaytrunc} is a consequence of~\cite[Theorem 5.12]{GO}, though translating Gonshor's statement for our purposes requires a bit of background.

%-----------------------%
\subsection{Infinite sums}\label{sec:convergence}
%-----------------------%
There is a notion of convergence in $\No$ for sequences and series of surreals that can be conveniently expressed using normal forms written as above with dummy terms. Let $x\in\No$ and for each $y\in\No$, let $r_{y}(x)$ be the coefficient of $\omega^y$ in the normal form of $x$, it being understood that $r_{y}(x)=0$, if $\omega^y$ does not occur. Also let $(x_n)_{n \in \N}$ be a sequence of surreals written in normal form. Following Conway (and adopting notation of Siegel~\cite[page 432]{SI}), we say that $(x_n)_{n\in \N}$ \textbf{converges} to $x$ if
\[
r_y (x)\ =\ \lim _{n\rightarrow \infty }r_y (x_{n}),\ \text{for all $y\in\No$},
\]
and we also write
\[
x\ =\ \sum_{n=0}^{\infty }x_n
\]
to mean \textbf{the partial sums of the series converge to $x$}.

Among the convergent sequences and series of surreals are those whose mode of convergence is quite distinctive: for each $y\in\No$, there is an $m\in\N$ such that $r_{y}(x_n)=r_{y}(x_m)$ for all $n\geq m$. Thus, for each $y\in\bf {No}$,
 \[
 r_y (x)\ =\ \lim _{n\rightarrow \infty }r_y (x_{n})\ =\ r_{y}(x_m),
 \]
where $m$ depends on $y$.

Conway~\cite[page 40]{CO}, calls this mode of convergence \textbf{absolute convergence}, and he (also see~\cite[pages 432-434]{SI}) proves the following result:

\begin{proposition}\label{sur5}
Let $f$ be a formal power series with real coefficients, i.e. let
\[
f (x) \ =\ \sum_{n=0}^{\infty }r_nx^n.
\]
Then $f (\epsilon)$ is absolutely convergent for all infinitesimals $\epsilon$ in $\No$. 
\end{proposition}

Conway proof make use of a classical combinatorial result of Neumann (\cite[pages 206-209]{N},\cite[Lemma 3.2]{SI},~\cite[pages 260-266]{AL}); and, indeed, as Conway indicates, his proof is a straightforward adaptation to $\No$ of a classical proof of Neumann~\cite[page 210]{N},~\cite[page 267]{AL} about Hahn fields (see below) and division rings of formal power series more generally. 

%-----------------------%
\subsection{Distinguished ordered binary subtrees of $\No$}
%-----------------------%
Henceforth, the class of $\No$'s leaders will be denoted `$Lead_{\No}$' or `$\omega^{\No}$' and the class of $\No$'s purely infinite numbers will be denoted `$\No_{\PI}$'. $\Oz$ is the canonical integer part of $\No$ consisting of the surreals whose Conway names have no negative exponents and whose coefficient for any term whose exponent is $0$ is an integer~\cite[p. 45]{CO}. $Lead_{\No}$, $\No_{\PI}$ and $\Oz$ all have ordered tree structures inherited from $(\No, <, < _s)$. In the subsequent discussion we will appeal to the following results about these substructures of $(\No, <, < _s)$, that come from~\cite[page 3: Note 2]{EH7},~\cite[page 1245: Theorem 11]{EH5} and~\cite{BH}, respectively.

\begin{lemma}
\label{lemma 0}
$(\Oz, <\restriction_{ \Oz}, < _{s}\restriction_{\Oz})$ is an initial subtree of $(\No, <, < _s)$.
\end{lemma}

\begin{lemma}
\label{lemma 1}
$(Lead_{\No}, <\restriction_{ Lead_{\No}}, < _{s}\restriction_{ Lead_{\No} })$ is a lexicographically ordered full binary tree.
\end{lemma}

\begin{lemma}
\label{lemma 2}
$(\No_{\PI}, <\restriction_{ \No_{\PI}}, < _{s}\restriction_{ \No_{\PI}})$ is a lexicographically ordered full binary tree.
\end{lemma}

\begin{corollary}
\label{corollary 1}
The following are lexicographically ordered full binary trees: 
\begin{enumerate}[(i)]
\item $(Lead_\No^{>1}, <\restriction_{ Lead_\No^{>1}}, < _{s}\restriction_{ Lead_\No^{>1} })$,
\item $(\No_{\PI}^{>0}, <\restriction_{ \No_{\PI}^{>0}}, < _{s}\restriction_{\No_{\PI}^{>0} })$, 
\item $(Lead_\No^{<1}, <\restriction_{ Lead_\No^{<1}}, < _{s}\restriction_{ Lead_\No^{<1} })$, and 
\item$(\No_{\PI}^{<0}, <\restriction_{ \No_{\PI}^{<0}}, < _{s}\restriction_{\No_{\PI}^{<0}})$.
\end{enumerate}
\end{corollary}
\begin{proof}
This follows from Lemmas~\ref{lemma 1} and~\ref{lemma 2} and the simple fact that deleting the root of a lexicographically ordered full binary tree, in this case $1$ and $0$ respectively, results in two such trees. 
\end{proof}

%------------------------------------------------%
\section{Preliminaries II: surreal exponentiation}\label{sec:SE}
%------------------------------------------------%
The Kruskal-Gonshor surreal exponential function $\exp$~\cite[Chapter 10]{GO} is defined by recursion as follows:
\[
\exp (x)\ =\ \Big\{ 0,\, (\exp x^L) [x-x^L]_{n},\, (\exp x^R) [x-x^R]_{2n+1}\, \Big|\, \frac{\exp x^L}{[x^L-x]_{2n+1}},\, \frac{\exp x^R}{[x^R-x]_{n}}\Big\},
\]
where $x^L$ and $x^R$ range over $L_{s(x)}$ and $R_{s(x)}$ respectively, $n$ ranges over $\N$, and $[y]_{n}$ denotes $1+y+\frac{y^2}{2!}+...+\frac{y^n}{n!}$ for all surreal $y$, it being understood that expressions containing terms of the form $[y]_{2n+1}$ are to be included only when $[y]_{2n+1} >0$. Recall that, being an elementary extension of $(\R,e^x)$ (Proposition~\ref{DE2}), $(\No,\exp)$ satisfies the condition: $\exp(x) > x^n$ for each positive infinite surreal number $x$ and each natural number $n$.

While the definition of $\exp$ is quite complicated for the general surreal case, the following result of Gonshor~\cite[pages 149-157]{GO} shows it reduces to more revealing and manageable forms for the three theoretically significant cases. 

\begin{proposition}
\label{Gonshor 1986}
Let $\exp$ be the Kruskal-Gonshor exponential on $\No$.
\begin{enumerate}[(i)]
\item $\exp (x)=e^x$ for all $x \in \R$; 
\item $\exp (x)=\sum_{\ n =0 }^{\infty }x^{n}/{n!}$ for all infinitesimal $x$;
\item if $x$ is purely infinite, then
\[
\exp (x)\ =\ \Big \{ 0,\, (\exp x^L)(x-x^L)^{n}\, \Big|\, \frac{\exp x^R}{(x^R-x)^n}\Big \},
\]
where $x^L$ and $x^R$ now range over all purely infinite predecessors of $x$ with $x^L<x<x^R$.
\end{enumerate}
\end{proposition}

The significance of cases (i)--(iii) accrues from the fact that for an arbitrary surreal number $x$, 
\[
\exp (x)\ =\ \exp(x_P)\cdot \exp(x_R)\cdot \exp(x_I)
\]
where $x_P$, $x_R$ and $x_I$ are the purely infinite, real and infinitesimal components of $x$, respectively. 

From an algebraic point of view, it is already clear from Proposition~\ref{Gonshor 1986} (i)--(ii) what $\exp(x)$ is for real and infinitesimal values of $x$. To shed further algebraic light on $\exp(x)$ when $x$ is purely infinite we turn to the following additional result of Gonshor~\cite[Theorem 10.9]{GO}

\begin{proposition}
\label{Gonshor 1986.1}
The restriction of $\exp$ to the class of purely infinite surreal numbers is an isomorphism of ordered groups onto $Lead_\No$.
\end{proposition}

In addition to its inductively defined exponential function $\exp$, Norton and Kruskal independently provided inductive definitions of the inverse function $\log$, but thus far only an inductive definition of $\log$ for surreals of the form $\omega^\gamma$ has appeared in print. Nevertheless, since each positive surreal $x$, written in normal form, has a unique decomposition of the form 
\[
x\ =\ \omega^\gamma r(1+\epsilon),
\]
where $\omega^\gamma$ is a leader, $r$ is a positive member of $\R$ and $\epsilon$ is an infinitesimal, $\log(x)$ may be obtained for an arbitrary positive surreal $x$ from the equation 
\[
\log(x)\ =\ \log(\omega^\gamma)+\ln(r)+ \sum_{k=1}^{\infty } \frac{ (-1)^{k-1} \epsilon^{k}}{k},
\]
where $\log(\omega^\gamma)$ is inductively defined by 
\[
\big \{ \log(\omega^{\gamma^{L}})+n,\, \log(\omega ^{\gamma^R})-\omega ^{\frac{\gamma^R-\gamma}{n}}\, \big|\, \log(\omega ^{\gamma^R})-n,\, \log(\omega ^{\gamma^L})+\omega ^{\frac{\gamma- \gamma^L}{n}} \big\},
\]
where $\gamma^L$ and $\gamma^R$ range over the left and right predecessors of $\gamma$, respectively, and $n$ ranges over the positive integers.

 The logarithm of a leader can be described more explicitly using Gonshor's map $h:\No \to \No^{>0}$, defined as follows~\cite[page 172]{GO}:
\[
h(s) \ =\ \Big\{0,\, h(s^L)\, \Big|\, h(s^R),\, \frac{1}{k}\omega^{s}\Big\}
\]
where $s^L$ ranges over $L_{s(x)}$, where $s^R$ ranges over $R_{s(x)}$, and where $k$ ranges over the positive integers. Note that $h$ is strictly increasing and that $h(s) \prec \omega^s$ for each $s$. The map $h$ is related to logarithms by the following result of Gonshor~\cite{GO}:

\begin{proposition}
\label{Gonshor 1986.2}
Let $\gamma = \sum_{\alpha<\beta} r_\alpha \omega^{s_\alpha}\in \No$. Then
\[
\log \omega^{\gamma}\ =\ \sum_{\alpha<\beta} r_\alpha \omega^{h(s_\alpha)}.
\]
In particular, we have $\log \omega^{\omega^s}= \omega^{h(s)}$ for each $s \in \No$.
\end{proposition}

%-----------------------%
\subsection{$s$-hierarchical ordered exponential fields} \label{SH}
%-----------------------%
Following~\cite[Definition 3]{EH5}, $(A,+, \cdot, <,<_{s},0,1) $ is said to be an {\it $s$-hierarchical ordered field} if (i) $(A,+, \cdot, <,0, 1)$ is an ordered field; (ii) $(A,<,<_{s})$ is a lexicographically ordered binary tree; and (iii) for all $x,y\in A$ 
\[
x+y\ =\ \{ x^{L}+y,\, x+y^{L} \, | \, x^{R}+y,\, x+y^{R}\} 
\]
and
\[
xy\ =\ \{x^{L}y+xy^{L}-x^{L}y^{L},\, x^{R}y+xy^{R}-x^{R}y^{R} \,| \,x^{L}y+xy^{R}-x^{L}y^{R},\, x^{R}y+xy^{L}-x^{R}y^{L}\},
\]
where $x^L$, $x^R$, $y^L$ and $y^R$ are understood to range over the members of $L_{s(x)}$, $R_{s(x)}$, $L_{s(y)}$ and $R_{s(y)}$, respectively. 

Extending this idea, we will say $(A,+, \cdot, \exp, <,<_{s},0,1) $ is an \textbf{$s$-hierarchical ordered exponential field} if (i) $(A,+, \cdot, \exp, <,0,1) $ is an ordered exponential field; (ii) $(A,+, \cdot, <,<_{s},0,1) $ is an $s$-hierarchical ordered field; and (iii) for all $x \in A$
\[
\exp (x)\ =\ \Big\{ 0,\, (\exp x^L) [x-x^L]_{n}, \, (\exp x^R) [x-x^R]_{2n+1}\, \Big|\, \frac{\exp x^L}{[x^L-x]_{2n+1}},\, \frac{\exp x^R}{[x^R-x]_{n}}\Big\},
\]
where $x^L$ and $x^R$ range over $L_{s(x)}$ and $R_{s(x)}$ respectively, $n$ ranges over $\N$, and $[y]_{n} =:1+y+\frac{y^2}{2!}+...+\frac{y^n}{n!}$ for all surreal $y$, it being understood that expressions containing terms of the form $[y]_{2n+1}$ are to be included only when $[y]_{2n+1} >0$. 

In virtue of the definitions of Conway's field operations and the definition of the Kruskal-Gonshor surreal exponential function, {\bf No} is of course an $s$-hierarchical ordered exponential field. Moreover, extending the argument for initial subfields of {\bf {No}} from~\cite[page 1236]{EH5}, it is evident that 

\begin{proposition}
Every initial exponential subfield of {\bf {No}} is itself an $s$-hierarchical ordered exponential field.
\end{proposition}

As such, much as the recursive unfolding of the $s$-hierarchical ordered field of surreal numbers generates a recursive unfolding of all the initial subfields of the surreals, the recursive unfolding of the $s$-hierarchical ordered exponential field $({\bf No}, \exp)$ gives rise to a recursive unfolding of all the initial exponential subfields of $({\bf No}, \exp)$, the characterization of which is one of the central focuses of the paper. 

Following~\cite[Definition 7]{EH5}, a mapping $f: A \rightarrow A'$ between two lexicographically ordered binary trees is said to be \textbf{$s$-hierarchical} if for all $x = \{L\, |\, R\} \in A$, $f(x) = \{f (L)\, |\, f (R)\}$. If, in addition, $A$ and $A'$ are $s$-hierarchical structures (ordered fields; ordered groups; ordered vector spaces; etc.) and the mapping is an embedding (of ordered fields; ordered groups; ordered vector spaces; etc) it is said to be an \textbf{$s$-hierarchical embedding}. 

If $f: A \rightarrow A'$ is an $s$-hierarchical mapping, then $f$ is the unique initial ordered tree embedding of $A$ into $A'$~\cite[Lemma 1]{EH5}. Moreover, an inductive argument shows that an $s$-hierarchical mapping of $s$-hierarchical ordered fields is an $s$-hierarchical embedding~\cite[Lemma 2]{EH5}. By combining the inductive argument for $s$-hierarchical ordered fields with a similar inductive argument for $\exp x$, in which the induction hypothesis is that the desired property holds for all the $\exp x^L$'s and $\exp x^R$'s, one may establish for $s$-hierarchical ordered exponential fields the analog of the just-said result for $s$-hierarchical ordered fields. Collecting these results together, we have:

\begin{proposition}
\label{Prop. SH}
(i) Every $s$-hierarchical mapping between $s$-hierarchical ordered exponential fields
is an $s$-hierarchical embedding; (ii) if $f: A \rightarrow A'$ and $g: A \rightarrow A'$ are $s$-hierarchical mappings,
then $f = g$ and $f(A)$ is initial.
\end{proposition}

Let $A$ be an $s$-hierarchical ordered exponential field. Extending the concepts of universality and maximality for $s$-hierarchical ordered fields~\cite[page 1239]{EH5} to $s$-hierarchical ordered exponential fields, $A$ will be said to be \textbf{universal} if for each $s$-hierarchical ordered exponential field $B$, there is an $s$-hierarchical embedding $f: B \rightarrow A$, and $A$ will be said to be \textbf{maximal} if there is no $s$-hierarchical ordered exponential field that properly contains $A$ as an initial exponential subfield.

In virtue of Proposition~\ref{Prop. SH} and the fact that there is an $s$-hierarchial mapping of every lexicographically ordered binary tree into $({\bf {No}},<,<_{s})$ (see the proof of~\cite[Theorem 5]{EH5}), we also have:

\begin{proposition}
{\bf {No}} is (up to isomorphism) the unique universal and the unique maximal $s$-hierarchical ordered exponential field.
\end{proposition}

%------------------------------------------------%
\section{Preliminaries III: ordered abelian groups}\label{sec:P2}
%------------------------------------------------%
Let $\Gamma$ be an ordered abelian group. For $x,y \in \Gamma$, we set
\begin{align*}
x \preceq y\ &:\Longleftrightarrow\ |x|<n|y|\text{ for some }n \in\N\\
x \prec y\ &:\Longleftrightarrow\ n|x|<|y|\text{ for all }n \in\N\\
x \asymp y\ &:\Longleftrightarrow\ x\preceq y\text{ and }y \preceq x\text{ (equivalently, if $x \preceq y$ and $x\not\prec y$).}
\end{align*}
Then $\asymp$ is an equivalence relation on $\Gamma\setminus \{0\}$ and the equivalence classes corresponding to $\asymp$ are called the \textbf{Archimedean classes of $\Gamma$}.
We say that $\Gamma$ is \textbf{Archimedean} if $\Gamma\setminus \{0\}$ consists of exactly one Archimedean class.

%-----------------------%
\subsection{Hahn groups}
%-----------------------%
Let $\R((t^S))_{\On}$ be the ordered group of power series (defined \'a la Hahn~\cite{H}) consisting of \emph{all} formal power series of the form $\sum_{\alpha < \beta} r_\alpha t^{s_\alpha}$ where $(s_\alpha)_{\alpha < \beta \in \On}$ is a possibly empty descending sequence of elements of an ordered class $S$ and $r_\alpha \in \R \setminus \{0\}$ for each $\alpha < \beta$. 
When $S$ is a set, then $\R((t^S))_{\On}$ is a set as well, and it is often simply written $\R((t^S))$. When $S$ is a proper class, then $\R((t^S))_{\On}$ is also a proper class. We call $S$ the \textbf{value class of $\R((t^S))_{\On}$}. In the literature, the appellation ``Hahn group'' is usually reserved for those structures $\R((t^S))_{\On}=\R((t^S))$, where $S$ is a set. However, we refer to $\R((t^S))_{\On}$ as a \textbf{Hahn group} whether $S$ is a set or a proper class.\footnote{By a classical result of Hahn~\cite{H}, when $S$ is a set, the Hahn groups coincide with the {\bf Archimedean complete} ordered abelian groups, i.e. the ordered abelian groups $G$ that admit no proper extension to an ordered abelian group $G'$ such that for each $x \in G' \setminus G$ there is a $y \in G$ where $x \asymp y$. When $S$ is a proper class, however, the Hahn group $\R((t^S))_{\On}$ (as defined above) is Archimedean complete if and only if $S$ has no descending subclass $(s_\alpha)_{\alpha < \On}$. Since $\No$ contains such subclasses, the surreals, whose ordered additive group is isomorphic to $\R((t^{\No}))_{\On}$ (see \S\ref{sec:SLHF}), is not Archimedean complete. In this respect, the ordered group of surreals does not have the full structure of a classical Hahn group.}

For a subclass $B \subseteq S$, we let $t^B$ denote the class $\{t^b:b \in B\}$. A subgroup $\Gamma$ of $\R((t^S))_{\On}$ is said to be \textbf{cross sectional} if $\{t^s : s \in S \} \subseteq \Gamma$. 

Let $y=\sum_{\alpha < \beta} r_\alpha t^{s_\alpha}$ be an element of $\R((t^S))_{\On}$. We let $\supp(y):=\{s_\alpha:\alpha<\beta\}\subseteq S$ denote the \textbf{support of $y$}, so $\supp(y)$ is empty if and only if $y = 0$. We also set $c(y):= s_0 = \max\supp(y)$, so $c(y)$ is the unique element of $S$ with $y \asymp t^{c(y)}$. An element $x \in \R((t^S))_{\On}$ is said to be a \textbf{truncation of $y$} if $x = \sum_{\alpha < \sigma} r_\alpha t^{s_\alpha}$ for some $\sigma \leq \beta$. Given $s \in S$, let $\sigma\leq \beta$ be the least ordinal such that $s_\sigma\leq s$, and let $y_{>s}$ be the truncation $\sum_{\alpha < \sigma} r_\alpha t^{s_\alpha}$. Then $y-y_{>s} \preceq t^s$, and $y - y_{>s} \asymp t^s$ if and only if $s \in \supp(y)$. A subgroup $\Gamma$ of $\R((t^S))_{\On}$ is said to be \textbf{truncation closed} if every truncation of every member of $\Gamma$ is itself a member of $\Gamma$. Let $\Gamma$ be a truncation closed subgroup of $\R((t^S))_{\On}$. We define the \textbf{$\Gamma$-truncation of $y$}, denoted $y_\Gamma$, as follows:
\begin{enumerate}
\item if $\sum_{\alpha \leq \sigma} r_\alpha t^{s_\alpha}$ belongs to $\Gamma$ for each $\sigma < \beta$, then we let $y_\Gamma := y$;
\item otherwise, let $\beta_0<\beta$ be least such that $\sum_{\alpha \leq \beta_0} r_\alpha t^{s_\alpha}$ does not belong to $\Gamma$. Then we let $y_\Gamma:= \sum_{\alpha < \beta_0} r_\alpha t^{s_\alpha}$.
\end{enumerate}
Note that $y_\Gamma$ may or may not belong to $\Gamma$. Of course, $y_\Gamma \in \Gamma$ if $y \in \Gamma$. If the second case holds and $\beta_0$ is a successor ordinal, then $y_\Gamma \in \Gamma$. If $\Gamma = \R((t^{S'}))_{\On}$ for some subclass $S' \subseteq S$, then $y_\Gamma$ always belongs to $\Gamma$. 

\begin{lemma}\label{samecutsametrunc}
Let $\Gamma$ be a truncation closed divisible subgroup of $\R((t^S))_{\On}$ and let $y,y' \in \R((t^S))_{\On}$. If $y$ and $y'$ realize the same cut over $\Gamma$, then $y_\Gamma = y'_\Gamma$.
\end{lemma}
\begin{proof}
Suppose that $y_\Gamma \neq y'_\Gamma$ and let $\sum_{\alpha<\beta} r_\alpha t^{s_\alpha}$ be the greatest common truncation of $y$ and $y'$ in $\Gamma$. Without loss of generality, we may assume that $\sum_{\alpha<\beta} r_\alpha t^{s_\alpha}$ is \emph{not} equal to $y_\Gamma$, so there is $r_\beta \in \R\setminus \{0\}$ and $s_\beta \in S$ such that $z:=\sum_{\alpha\leq\beta} r_\alpha t^{s_\alpha}\in \Gamma$ is a truncation of $y$ but not $y'$. Note that $r_\beta t^{s_\beta} = z-\sum_{\alpha<\beta} r_\alpha t^{s_\alpha}$ belongs to $\Gamma$. We have
\[
y- z\ \prec\ r_\beta t^{s_\beta},\qquad y'- z\ \succeq\ r_\beta t^{s_\beta}.
\]
If $y-z$ and $y'-z$ have opposite signs, then $z$ is between $y$ and $y'$. Suppose that $y-z$ and $y'-z$ are both positive. Take $n >0$ such that $y'-z> \frac{r_\beta}{n} t^{s_\beta}$. Then $y<z+\frac{r_\beta}{n} t^{s_\beta}<y'$. Since $z$ and $r_\beta t^{s_\beta}$ belong to $\Gamma$ and since $\Gamma$ is divisible, we see that $z+\frac{r_\beta}{n} t^{s_\beta}$ belongs to $\Gamma$, so $y$ and $y'$ do not realize the same cut over $\Gamma$. The case that $y-z$ and $y'-z$ are both negative is similar.
\end{proof}

%-----------------------%
\subsection{Hahn spaces}
%-----------------------%
Let $\k$ be an Archimedean ordered field, that is, an ordered field with an Archimedean underlying ordered additive group. We uniquely identify $\k$ with a subfield of $\R$. Via this identification, we view $\R((t^S))_{\On}$ as an ordered $\k$-vector space when convenient (such as in Lemma~\ref{notsoclose} below). 

\begin{definition}
A \textbf{Hahn space over $\k$} is an ordered $\k$-vector space $\Gamma$ such that for all $a,b \in \Gamma \setminus \{0\}$ with $a \asymp b$, there is $r \in \k$ with $a-rb \prec b$. 
\end{definition}

See~\cite[Section 2.3]{ADH} for more information on Hahn spaces. Any ordered $\R$-vector space, including the Hahn group $\R((t^S))_{\On}$, is a Hahn space over $\R$.

\begin{lemma}\label{notsoclose}
Let $\Delta \subseteq \Gamma\subseteq \R((t^S))_{\On}$ be $\k$-vector spaces and suppose that $\Gamma$ is a Hahn space over $\k$ and $\Delta$ is truncation closed. Let $y \in \Gamma\setminus \Delta$, and let $z \in \Delta$ be a truncation of $y$. Then $z = y_\Delta$ if and only if $y-z \not\asymp a$ for all $a \in \Delta$.
\end{lemma}
\begin{proof}
Write $y = \sum_{\alpha < \beta} r_\alpha t^{s_\alpha} \in \Gamma$ and take $\beta_0 < \beta$ with $z = \sum_{\alpha < \beta_0} r_\alpha t^{s_\alpha} \in \Delta$. If $z \neq y_\Delta$, then $z + r_{\beta_0} t^{s_{\beta_0}}= \sum_{\alpha \leq \beta_0} r_\alpha t^{s_\alpha}$ belongs to $\Delta$ as well, so $ y-z \asymp r_{\beta_0} t^{s_{\beta_0}} \in \Delta$. Conversely, suppose that $y - z \asymp a$ for some $a \in \Delta$. Since $\Gamma$ is a Hahn space over $\k$, there is some $r \in \k$ with $y-z-ra \prec a$. Thus, $ra = r_{\beta_0} t^{s_{\beta_0}}+ \epsilon$, where $\epsilon \prec t^{s_{\beta_0}}$. Since $\Delta$ is truncation closed, we see that $r_{\beta_0} t^{s_{\beta_0}} \in \Delta$, so $z + r_{\beta_0} t^{s_{\beta_0}}$ is a truncation of $y$ contained in $\Delta$ which strictly extends $z$. Thus, $z \neq y_{\Delta}$.
\end{proof}

%-----------------------%
\subsection{Initial subgroups and subspaces of $\No$}
%-----------------------%
For $y \in \No$ with $y \neq 0$, we let $c(y)$ denote the unique element of $\No$ such that $y \asymp \omega^{c(y)}$. This agrees with our Hahn group definition above, where we identify $\No$ with the Hahn group $\R((\omega^\No))_{\On}$. This also agrees with the map $c$ defined by Gonshor~\cite{GO}. Let $\Gamma$ be an initial divisible subgroup of $\No$. The \textbf{value class of $\Gamma$} is the class
\[
S\ :=\ \{s \in \No: \omega^s \in \Gamma\}.
\]
Then $S$ is an initial subclass of $\No$ and $\Gamma$ is a truncation closed, cross sectional subgroup of $\R((\omega^S))_{\On}$~\cite[Theorem 5.1]{EK}.
Thus, each $\gamma \in \No$ has a $\Gamma$-truncation $\gamma_\Gamma$.

Let $\k$ be an Archimedean ordered field. Our identification of $\k$ with a subfield of $\R$ allows us to view $\k$ as an \emph{initial} subfield of $\No$. In fact, this identification $\k \hookrightarrow\R\subseteq \No$ is the \emph{unique} initial ordered field embedding $\k\to \No$~\cite[Theorem 8]{EH5}. We view $\No$ as a $\k$-vector space via this embedding. By Proposition~\ref{KV}, we get that every ordered $\k$-vector space $\Gamma$ admits an initial $\k$-linear embedding into $\No$. For future use, we record the induction step in the proof of this result~\cite[pages 1241-1242]{EH5}:

\begin{lemma}\label{newbasiselem}
Let $\Gamma$ be an initial ordered $\k$-vector subspace of $\No$ and let $x \in \No\setminus \Gamma$. Suppose that $x$ is the simplest element realizing a cut over $\Gamma$. Then the $\k$-vector subspace $\Gamma+\k x \subseteq\No$ is initial. In particular, if $S$ is the value class of $\Gamma$ and if $s \in \No\setminus S$ is the simplest element realizing a cut over $S$, then $\Gamma+\k \omega^s$ is initial.
\end{lemma}

%------------------------------------------------%
\section{Preliminaries IV: Hahn fields}\label{sec:P3}
%------------------------------------------------%
Let $K$ be an ordered field. The Archimedean classes of $K$ as well as the relations $\preceq$, $\prec$, and $\asymp$ are defined with respect to the underlying ordered additive group of $K$. We say that an element $x \in K$ is \textbf{infinite} if $x \succ 1$ and \textbf{infinitesimal} if $x \prec 1$. A \textbf{cross section} for $K$ is a multiplicative subgroup $\mfM \subseteq K^{>0}$ such that for each $x \in K\setminus \{0\}$ there is exactly one $\fm \in \mfM$ with $x \asymp \fm$.

Let $\Gamma$ be an ordered abelian group. Then the Hahn group $\R((t^\Gamma))_{\On}$ is in fact a \textbf{Hahn field}~\cite{H} when multiplication is defined \`a la polynomials with $t^{\gamma_1} t^{\gamma_2} = t^{\gamma_1 + \gamma_2}$ for all $\gamma_1, \gamma_2 \in \Gamma$. We call $\Gamma$ the \textbf{value group of $\R((t^\Gamma))_{\On}$}. Like the nonzero surreals written in normal form, each element $x$ of a Hahn field is the sum of three easily recognizable components: the \textbf{purely infinite} component of $x$, whose terms solely have positive exponents; the \textbf{real} component of $x$, whose sole term (if it is not the empty sum) has exponent $0$; and the \textbf{infinitesimal} component of $x$, whose terms solely have negative exponents. Again, as in the surreal case, $0$, being the empty sum, may be regarded as purely infinite. We remark that the purely infinite component of $x$ coincides with the truncation $x_{>0}$, where $0$ is the identity element of the value group $\Gamma$, and that $x$ is itself purely infinite just in case $x = x_{>0}$. As we noted above, following Proposition~\ref{sur5} in the surreal case, infinite sums of the form
\[
\sum_{n=0}^{\infty }r_n\epsilon^n
\]
are well-defined in Hahn fields for infinitesimal elements $\epsilon \in \R((t^\Gamma))_{\On}$, in virtue of a classical result of Neumann~\cite{N}.

Let $K$ be a truncation closed, cross sectional subfield of $\R((t^\Gamma))_{\On}$. The set $\R_K= \{r\ \in \R : rt^0 \in K \}$ is an Archimedean ordered field, which we will call the \textbf{coefficient field} of $K$. Note that $\R_K$ is isomorphic to the residue class field of $K$ with respect to the Archimedean valuation.\footnote{The traditional definition of the residue class field does not work in NBG if $K$ is a proper class. For a suitable modification applicable to sets and proper classes, see~\cite[page 1253]{EH5}.} The multiplicative subgroup $t^\Gamma\subseteq K^{>0}$ is a canonical cross section for $K$.

The following result, which is employed in the proof of Proposition~\ref{IF}, is critical in the proof of the main theorem.

\begin{proposition}[\cite{EH5}, Theorem 18;~\cite{EH8}, Theorem 14]
\label{Ehrlich initial 2}
If $K$ is a truncation closed, cross sectional subfield of a Hahn field $\R((t^\Gamma))$ and $\imath:\Gamma\rightarrow \No$ is an initial group embedding, then the mapping that sends 
\[
\sum_{\alpha < \beta} r_\alpha t^{\gamma_\alpha}\ \in\ K
\]
to the surreal number having Conway name
\[
\sum_{\alpha < \beta} r_\alpha \omega^{\imath(\gamma_\alpha)}
\]
is an initial embedding of $K$ into $\No$.
\end{proposition}

If $K$ is an initial subfield of $\No$, then by~\cite[Theorem 18]{EH5}, the class
\[
\Gamma\ :=\ \{\gamma \in \No:\omega^\gamma \in K\}
\]
is an initial subgroup of $\No$, which we call the \textbf{value group of $K$}, and $K$ is a truncation closed, cross sectional subgroup of $\R((\omega^\Gamma))_{\On}$.

%------------------------------------------------%
\section{Transserial Hahn fields}\label{sec:TF}
%------------------------------------------------%
A \textbf{logarithm} on an ordered field $K$ is an ordered group embedding $\log:K^{>0}\to K$ which satisfies $\log y\leq y-1$ for $y \in K^{>0}$. An \textbf{ordered logarithmic field} is an ordered field equipped with a logarithm. We denote the functional inverse of $\log$ by $\exp$, where it is defined. If $K= (K,\log)$ is an ordered logarithmic field and $\log$ is surjective onto $K$, then $K$ is called an \textbf{ordered exponential field}. 

\begin{definition}
A \textbf{logarithmic Hahn field} is an ordered Hahn field $\R((t^\Gamma))_{\On}$ equipped with a logarithm $\log$ which satisfies the following axioms:
\begin{enumerate}
\item $\log$ extends the natural logarithm on $\R^{>0}$;
\item $\log t^\gamma$ is purely infinite for each $\gamma \in \Gamma$;
\item $\log(1+\epsilon) = \sum_{n=1}^\infty (-1)^{n-1}\epsilon^n/n$ for all infinitesimal $\epsilon \in \R((t^\Gamma))_{\On}$.
\end{enumerate}
\end{definition}

\begin{remark}\label{rem-extendlog}
Let $\R((t^\Gamma))_{\On}$ be an ordered Hahn field, and let $\log : t^\Gamma\to \R((t^\Gamma))$ be an ordered group embedding. If $\log t^\gamma$ is purely infinite and $\log t^\gamma \prec t^\gamma$ for each $\gamma \in \Gamma^{>0}$, then there is a unique logarithm on $\R((t^\Gamma))_{\On}$ which extends $\log$ and makes $\R((t^\Gamma))_{\On}$ into a logarithmic Hahn field. This logarithm (which we also denote by $\log$) is given by setting
\[
\log\big(rt^\gamma(1+\epsilon)\big)\ :=\ \ln r + \log t^\gamma+ \sum_{n=1}^{\infty } (-1)^{n-1} \epsilon^n/n
\]
for $r\in \R^{>0}$, $\gamma \in \Gamma$, and $\epsilon \prec 1$.
\end{remark}

\begin{definition}[Schmeling, Definition 2.2.1~\cite{SC}]\label{transseriesdef}
A \textbf{transserial Hahn field} is a logarithmic Hahn field $\R((t^\Gamma))_{\On}$ which satisfies the additional axiom:
\begin{enumerate}
\item[(T4)] For every sequence $(\gamma_n)_{n \in\N}$ in $\Gamma$ such that $\gamma_{n+1}\in \supp(\log t^{\gamma_n})$ for each $n$, there is an $n_0\in \N$ such that for all $n\geq n_0$, we have
\[
\big|\log t^{\gamma_n} - (\log t^{\gamma_n})_{>\gamma_{n+1}}\big|\ =\ t^{\gamma_{n+1}}.
\]
\end{enumerate}
\end{definition}

\begin{remark}\label{rem-T4}
Schmeling introduced axiom (T4) so that certain \emph{compositions} of transseries are well-defined; see~\cite[Remark 2.2.3]{SC}. For us, however, the utility of axiom (T4) comes from its connection to the existence of \emph{$\Delta$-atomic elements}, introduced in \S\ref{subsec-deltapaths}. Lemmas~\ref{existsatomic} and~\ref{lem-existsatomic2} clarify the connection between axiom (T4) and the existence of these elements.
\end{remark}

\begin{remark}\label{rem-T4.2}
Schmeling uses the terminology \emph{transseries field} instead of \emph{transserial Hahn field}. While Schmeling's transseries fields are always sets, we allow for transserial Hahn fields which are proper classes. Our terminology is chosen partially for this additional generality and partially to avoid any confusion that may arise by employing both ``transseries field'' and ``field of transseries'', which have distinct meanings.
\end{remark}

For the remainder of this section, we let $\T= \R((t^\Gamma))_{\On}$ be a logarithmic Hahn field. We use the notation $\T$ to bring ``transseries'' to mind, but much of this section does not require axiom (T4), so we do not assume this axiom until it is needed.

\begin{lemma}\label{lem-posinflogs}
Let $a \in\T$ be positive and infinite. Then $\log a$ is positive and infinite and $\log a \prec a$.
\end{lemma}
\begin{proof}
Let $n \in \N$ be given. We need to show that $n<\log a$ and that $n \log a < a$. To see that $n < \log a$, use that $a > \exp n\in \R$ and that $\log$ is strictly increasing, so $\log a > \log (\exp n) = n$. To see that $n \log a < a$, we first note that $\log a\leq a-1< a$, and we apply this to $a/(n+1)$ in place of $a$ to get
\[
(n+1)\big(\log a - \log(n+1)\big)\ =\ (n+1)\log\big(a/(n+1)\big)\ <\ (n+1)\big(a/(n+1)\big)\ = \ a.
\]
Since $ \log(n+1)\prec \log a$, we have $(n+1)\big(\log a - \log(n+1)\big)> n\log a$, so $n\log a <a$.
\end{proof}

\begin{lemma}\label{lem-boundedtobounded}
Let $u \in \T^{>0}$ with $u \asymp 1$. Then $\log u \preceq 1$. Moreover, for each $z \in \T$ with $z\preceq 1$, $\exp z$ is defined and $\exp z \asymp 1$.
\end{lemma}
\begin{proof}
Take $r \in \R^{>0}$ and $\epsilon \prec 1$ with $y = r(1+\epsilon)$. Then
\[
\log y \ =\ \ln r + \log(1+\epsilon).
\]
Since $\ln r \in \R$ and $\log(1+\epsilon)\prec 1$, we see that $\log y \preceq 1$. Now let $z \in \T$ with $z\preceq 1$. Take $r' \in \R$ and $\epsilon' \prec 1$ with $z = r'+\epsilon'$. Then $\exp r' = e^{r'} \in \R^{>0}$ and $\exp \epsilon' = \ \sum_{n=0}^\infty (\epsilon')^n/n!\asymp 1$, so $\exp z = e^{r'}\exp \epsilon' \asymp 1$.
\end{proof}

\begin{lemma}\label{lem-monomialinfinite}
Let $y \in \T^{>0}$ and suppose that $\log y$ is purely infinite. Then $y \in t^\Gamma$. 
\end{lemma}
\begin{proof}
Take $\gamma \in \Gamma$ and $u \in \T^{>0}$ with $u \asymp 1$ and $y = ut^\gamma$. Then 
\[
\log y\ =\ \log t^\gamma+ \log u.
\]
Since $\log y$ and $\log t^\gamma$ are purely infinite and $\log u \preceq 1$ by Lemma~\ref{lem-boundedtobounded}, we must have $\log u = 0$, so $u=1$ and $y = t^\gamma$.
\end{proof}

\begin{lemma}\label{lem-monomtoall}
Suppose that $\log(t^\Gamma)$ is a truncation closed subgroup of $\T$. Then $\log(\T^{>0})$ is also truncation closed.
\end{lemma}
\begin{proof}
Let $y \in \T^{>0}$ and take $\gamma \in \Gamma$ and $u \in \T^{>0}$ with $u \asymp 1$ and $y = ut^\gamma$. Then 
\[
\log y\ =\ \log t^\gamma+ \log u.
\]
Let $x$ be a truncation of $\log y$. Since $\log t^\gamma$ is purely infinite and $\log u \preceq 1$ by Lemma~\ref{lem-boundedtobounded}, we either have that $x$ is a truncation of $\log t^\gamma$ or that $x = \log t^\gamma+ z$, where $z \preceq 1$ is a truncation of $\log u$. In the first case, we have $x \in \log(t^\Gamma)$ by our assumption that $\log(t^\Gamma)$ is truncation closed. In the second case, Lemma~\ref{lem-boundedtobounded} gives that $z \in \log(\T^{>0})$, so $\log t^\gamma + z$ belongs to $\log(\T^{>0})$ as well. 
\end{proof}

%-----------------------%
\subsection{Subfields parametrized by subspaces of $\Gamma$}
%-----------------------%
Let $\k$ be an Archimedean ordered field. If $\log(t^\Gamma)$ is a $\k$-subspace of $\T$, then we view $\Gamma$ as an ordered $\k$-vector space, where $r\gamma$ is the unique element of $\Gamma$ with $\log t^{r\gamma} = r\log t^\gamma$ for $r \in \k$ and $\gamma \in \Gamma$. We will say that \textbf{$\Gamma$ is an ordered $\k$-vector space} to mean that $\log(t^\Gamma)$ is a $\k$-subspace of $\T$, where the operation of scalar multiplication on $\Gamma$ is assumed to be the operation described above. For the remainder of this section, we assume that $\Gamma$ is an ordered $\k$-vector space. It follows easily from Lemma~\ref{lem-boundedtobounded} that then $\log(\T^{>0})$ is a $\k$-subspace of $\T$. For $r \in \k$ and $y \in \T^{>0}$, we let $y^r$ be the unique element of $\T^{>0}$ with $\log y^r = r\log y$. Let $\Delta$ be a $\k$-subspace of $\Gamma$ and set $\T_\Delta:=\R((t^\Delta))_{\On} \subseteq \T$. Note that the $\T_{\Delta}$-truncation of $y$ lies in $\T_{\Delta}$ for each $y \in \T$. 

\begin{definition}
We say that $\Delta$ is a \textbf{logarithmic $\k$-subspace of $\Gamma$} if $ \log(t^\Delta)\subseteq \T_\Delta$. We say that $\Delta$ is an \textbf{exponential $\k$-subspace of $\Gamma$} if $\log(t^\Delta) =\log(t^\Gamma) \cap \T_\Delta$.
\end{definition}

\begin{lemma}\label{logsubspaceislogsubfield}
If $\Delta$ is a logarithmic $\k$-subspace of $\Gamma$, then $\log(\T_\Delta^{>0})\subseteq \T_\Delta$. If $\Delta$ is an exponential $\k$-subspace of $\Gamma$, then $\log(\T_\Delta^{>0})=\log(\T^{>0})\cap \T_\Delta$.
\end{lemma}
\begin{proof}
Suppose that $\Delta$ is a logarithmic $\k$-subspace of $\Gamma$, let $x \in \T_\Delta^{>0}$ and take $\delta \in \Delta$, $r \in \R^{>0}$, and $\epsilon \in \T_\Delta$ with $\epsilon \prec 1$ such that $x = rt^\delta(1+\epsilon)$. Then
\[
\log x \ = \ \log t^\delta + \ln r + \log(1+\epsilon).
\]
We have $\log t^\delta \in \T_\Delta$ by assumption and $\ln r \in \R\subseteq \T_\Delta$. We also have 
\[
\log(1+\epsilon) \ =\ \sum_{n=1}^{\infty } (-1)^{n-1} \epsilon^n/n\ \in\ \T_\Delta,
\]
since this sum exists in $\T$ and consists of elements from $\T_{\Delta}$.

Now suppose that $\Delta$ is an exponential $\k$-subspace of $\Gamma$ and let $y \in \log(\T^{>0})\cap \T_\Delta$. We need to show that $y\in \log(\T_{\Delta}^{>0})$. Take $r \in \R$ and $\epsilon \prec 1$ with $y = y_{>0}+r+\epsilon$. We have $e^r \in \R^{>0} \subseteq\T_{\Delta}^{>0}$, and we have 
\[
\exp \epsilon \ =\ \sum_{n =0 }^{\infty }\epsilon^n/n!\ \in\ \T_\Delta^{>0},
\]
since this sum exists in $\T^{>0}$ and consists of elements from $\T_{\Delta}$. Thus, $r$ and $\epsilon$ belong to $\log(\T_\Delta^{>0})$. Since $\log(\T^{>0})\cap \T_\Delta$ is a subgroup of $\T_{\Delta}$ containing $y$, $r$, and $\epsilon$, it follows that $y_{>0} \in \log(\T^{>0})\cap \T_\Delta$. Since $y_{>0}$ is purely infinite, Lemma~\ref{lem-monomialinfinite} gives that $y_{>0} \in \log(t^\Gamma) \cap \T_{\Delta}$. Thus, $y_{>0} \in \log(t^\Delta) \subseteq \log(\T_{\Delta}^{>0})$, as $\Delta$ is an exponential $\k$-subspace of $\Gamma$. We conclude that $y \in \log(\T_{\Delta}^{>0})$.
\end{proof}

By Lemma~\ref{logsubspaceislogsubfield}, the subfield $\T_\Delta$ of $\T$ is a logarithmic Hahn field whenever $\Delta$ is a logarithmic $\k$-subspace of $\Gamma$. If $\T$ is a transserial Hahn field, then so is $\T_\Delta$. 

\begin{lemma}\label{deltacuts2}
Let $\Delta$ be an exponential $\k$-subspace of $\Gamma$ and let $y \in \T^{>0}$ with $c(y) \not \in \Delta$. Then $\log y$ and $\log t^{c(y)}$ realize the same cut over $\T_\Delta$.
\end{lemma}
\begin{proof}
We may, of course, assume that $y \neq t^{c(y)}$. Take $u \in \T^{>0}$ with $u \asymp 1$ and $uy = t^{c(y)}$. Then $\log u + \log y = \log t^{c(y)}$. Suppose that $\log y < \log t^{c(y)}$ (the case when $\log y >\log t^{c(y)}$ is similar). Suppose towards contradiction that $\log y < a< \log t^{c(y)}$ for some $a \in \T_\Delta$. Then $0 < a- \log y<\log u \preceq 1$, so by Lemma~\ref{lem-boundedtobounded}, we may take $z \in \T^{>0}$ with $z \asymp 1$ and $\log z = a-\log y$. Then $\log(yz) = a \in \T_\Delta$, so $yz \in \T_\Delta^{>0}$ since $\Delta$ is an exponential $\k$-subspace of $\Gamma$. Since $z \asymp 1$, this gives $c(y) \in \Delta$, a contradiction. 
\end{proof}

\begin{lemma}\label{deltacuts}
Let $\Delta$ be a logarithmic $\k$-subspace of $\Gamma$ with $t^{\Delta^{>0}} \subseteq \log(t^\Delta)$ and let $y \in \T$ with $c(y) \not \in \Delta$. Then the cut that $y$ realizes over $\T_\Delta$ is completely determined by the cut that $y$ realizes over $\log(\T_\Delta^{>0})$.
\end{lemma}
\begin{proof}
We may assume that $y>0$. Let $a \in \T_\Delta$ with $a < y$. We need to find an element of $\log(\T_\Delta^{>0})$ between $a$ and $y$. If $a\leq 0$, then $a\leq \log 1<y$, so we may assume that $a>0$. Take $\delta \in \Delta$ with $a \asymp t^\delta$ and take $r \in \k^{>0}$ with $a<r t^\delta$. Since $c(y) \not \in \Delta$, we have $t^\delta \prec y$, so $rt^\delta < y$ and it remains to show that $rt^\delta \in \log(\T_\Delta^{>0})$. If $\delta \leq 0$, then $rt^\delta \preceq 1$, so $r t^\delta \in \log(\T_\Delta^{>0})$ by Lemma~\ref{lem-boundedtobounded}. If $\delta>0$, then using our assumption that $t^{\Delta^{>0}} \subseteq \log(t^\Delta)$, we have $t^\delta = \log t^\gamma$ for some $\gamma \in \Delta$. Since $\Delta$ is a $\k$-subspace of $\Gamma$, we see that $rt^\delta = \log t^{r\gamma} \in \log(t^\Delta)$. A similar argument shows that for $b \in \T$ with $b>y$, we can an element of $\log(\T_\Delta^{>0})$ between $b$ and $y$. 
\end{proof}

As the reader will recall, a family of structures $(A_\alpha)_{\alpha<\beta}$ where $\beta\leq \On$ is said to be a \textbf{continuous chain} if $A_\sigma\subseteq A_\alpha$ for all $\sigma<\alpha<\beta$ and if $A_\alpha = \bigcup_{\sigma<\alpha}A_\sigma$ for each infinite limit ordinal $\alpha<\beta$. 

\begin{lemma}\label{defofDeltaE}
If $\Delta$ is a logarithmic $\k$-subspace of $\Gamma$ then there is a smallest exponential $\k$-subspace of $\Gamma$ containing $\Delta$, which we denote by $\Delta^E$. If $\T_\Delta$ is a transserial Hahn field, then so is $\T_{\Delta^E}$.
\end{lemma}
\begin{proof}
We construct a continuous chain $(\Delta_\alpha)_{\alpha< \On}$ of logarithmic $\k$-subspaces of $\Gamma$ by setting 
\[
\Delta_0\ :=\ \Delta,\qquad \Delta_{\alpha+1}\ :=\ \big\{\gamma \in \Gamma: \log t^\gamma \in \T_{\Delta_\alpha}\big\},\qquad \Delta_\alpha\ :=\ \bigcup_{\beta<\alpha}\Delta_\beta \text{ when $\alpha>0$ is a limit ordinal.}
\]
Note that by induction on $\alpha$, $\Delta_{\alpha+1}$ is indeed a logarithmic $\k$-subspace of $\Gamma$ containing $\Delta_{\alpha}$ for each ordinal $\alpha$. Set $\Delta^E:= \bigcup_{\alpha<\On}\Delta_\alpha$. We claim that $\Delta^E$ is an exponential $\k$-subspace of $\Gamma$. To see this, let $\gamma \in \Gamma$ with $\log t^\gamma \in \T_{\Delta^E}$. Since $\supp(\log t^\gamma)$ is a set, it must be contained in some $\Delta_\alpha$, so $\log t^\gamma \in \T_{\Delta_\alpha}$ and $\gamma \in \Delta_{\alpha+1}\subseteq \Delta^E$. To see that $\Delta^E$ is minimal, note that for each $\alpha$, any exponential $\k$-subspace of $\Gamma$ containing $\Delta_\beta$ for each $\beta<\alpha$ must contain $\Delta_\alpha$. By induction, we see that any exponential $\k$-subspace of $\Gamma$ containing $\Delta_0= \Delta$ must contain $\Delta^E$.

Now suppose that $\T_\Delta$ satisfies axiom (T4). We will prove by induction on $\alpha<\On$ that $\T_{\Delta_\alpha}$ also satisfies (T4). This is clear if $\alpha$ is a limit ordinal. Suppose that $\T_{\Delta_\alpha}$ satisfies (T4) and let $\gamma_0,\gamma_1,\ldots \in \Delta_{\alpha+1}$ be as in the statement of (T4). Then $\log t^{\gamma_0} \in \T_{\Delta_\alpha}$ by definition, so $\gamma_1 \in \Delta_\alpha$. This gives $\gamma_n \in \Delta_\alpha$ for all $n>0$. Since $\T_{\Delta_\alpha}$ satisfies (T4), we have
\[
\big|\log t^{\gamma_n} - (\log t^{\gamma_n})_{>\gamma_{n+1}}\big|\ =\ t^{\gamma_{n+1}}
\]
for all $n$ greater than some $n_0$. Thus, $\T_{\Delta_{\alpha+1}}$ satisfies (T4) as well.
\end{proof}

At various points, we will need to impose additional assumptions on $\Gamma$, namely that $\log(t^\Gamma)$ is truncation closed and that $t^{\Gamma^{>0}} \subseteq \log(t^\Gamma)$. We already made use of the latter property (with $\Delta$ in place of $\Gamma$) in Lemma~\ref{deltacuts}. The following lemma shows that if $\Gamma$ enjoys these properties, then we can find a logarithmic $\k$-subspace $\Delta \subseteq \Gamma$ which enjoys the same properties while still being a set. It is worth mentioning that any \emph{exponential} $\k$-subspace $\Delta\subseteq \Gamma$ inherits these properties from $\Gamma$, but that if $\Gamma$ is a proper class, then it may not be possible to find a nontrivial exponential $\k$-subspace of $\Gamma$ whose universe is a set. This is a consequence of negative results of van der Hoeven~\cite{vdH} and Kuhlmann-Kuhlmann-Shelah~\cite{KKS} regarding exponentials on Hahn fields.

\begin{lemma}\label{settosubspace}
Let $A$ be a subset of $\Gamma$. Then $A$ is contained in a logarithmic $\k$-subspace $\Delta\subseteq \Gamma$ whose universe is a set. Moreover, if $\log(t^\Gamma)$ is truncation closed and $t^{\Gamma^{>0}} \subseteq \log(t^\Gamma)$, then $\Delta$ may be chosen so that $\log(t^\Delta)$ is truncation closed and $t^{\Delta^{>0}} \subseteq \log(t^\Delta)$.
\end{lemma}
\begin{proof}
Let $\Delta_0$ be the $\k$-subspace of $\Gamma$ generated by $A$ and for each $n$, let $\Delta_{n+1}$ be the $\k$-subspace of $\Gamma$ generated by $\Delta_n$ and $\bigcup_{\delta \in \Delta_n} \supp(\log t^\delta)$. Let $\Delta_\omega:=\bigcup_n \Delta_n$. We note that $\Delta_\omega$ is a $\k$-subspace containing $A$ whose universe is a set, and we claim that $\Delta_\omega$ is a logarithmic $\k$-subspace of $\Gamma$. Indeed, let $\delta \in \Delta_\omega$ and take $n$ with $\delta \in \Delta_n$. Then $\log t^\delta \in \T_{\Delta_{n+1}}\subseteq \T_{\Delta_\omega}$. If we only require $\Delta$ to be a logarithmic $\k$-subspace of $\Gamma$, then we may take $\Delta = \Delta_\omega$, but for the ``moreover'' part of the lemma, we have to do a bit more work.

Suppose $\log(t^\Gamma)$ is truncation closed and $t^{\Gamma^{>0}} \subseteq \log(t^\Gamma)$. Given $\Delta_\omega$ as above, we define an increasing chain $(\Delta_{\omega,n})_{n \in \N}$ of $\k$-subspaces of $\Gamma$ as follows:
\begin{enumerate}
\item Let $\Delta_{\omega,0}$ be the $\k$-subspace of $\Gamma$ generated by the set
\[
B_0\ :=\ \{\gamma \in \Gamma: \log t^\gamma\text{ is a truncation of }\log t^\delta \text{ for some }\delta \in \Delta_\omega\}.
\]
\item For each $n$, let $\Delta_{\omega,n+1}$ be the $\k$-subspace of $\Gamma$ generated by the set
\[
B_{n+1}\ :=\ \Delta_{\omega,n}\cup \{\gamma \in \Gamma: \log t^\gamma \in t^{\Delta_{\omega,n}^{>0}}\}.
\]
\end{enumerate}
Note that $\Delta_\omega$ is contained in $B_0 \subseteq \Delta_{\omega,0}$, since $\log t^\delta$ is a truncation of itself for $\delta \in \Delta_\omega$. We claim that each $\Delta_{\omega,n}$ is a logarithmic $\k$-subspace of $\Gamma$. Since $\log(t^{\Delta_{\omega,n}})$ is equal to the $\k$-subspace of $\T$ generated by $\log(t^{B_n})$, it suffices to show that $\log(t^{B_n}) \subseteq \T_{\Delta_{\omega,n}}$ for each $n$. This holds for $B_0$ since each element in $\log(t^{B_0})$ is a truncation of some element in $\log(t^{\Delta_\omega})\subseteq\T_{\Delta_\omega}$, and so $\log(t^{B_0}) \subseteq \T_{\Delta_\omega}$. Note that $t^{\Delta_{\omega,n}^{>0}} \subseteq \log(t^{B_{n+1}})$ by our assumption that $t^{\Gamma^{>0}} \subseteq \log(t^\Gamma)$. Thus, we have the equality 
\[
\log(t^{B_{n+1}}) \ =\ t^{\Delta_{\omega,n}^{>0}}\cup \log(t^{\Delta_{\omega,n}}).
\]
Assuming that $\Delta_{\omega,n}$ is a logarithmic $\k$-subspace of $\Gamma$, this gives $\log(t^{B_{n+1}}) \subseteq\T_{\Delta_{\omega,n}}$, so $\Delta_{\omega,n+1}$ is a logarithmic $\k$-subspace of $\Gamma$ as well. Next, we claim that $\log(t^{\Delta_{\omega,n}})$ is truncation closed for each $n$. First, let $\gamma,\delta \in \Delta_{\omega,0}$ be given. We will show that $(\log t^\gamma)_{>\delta} \in \log(t^{\Delta_{\omega,0}})$. Take $\gamma_1,\ldots,\gamma_m \in B_0$ and $r_1,\ldots,r_m \in \k$ with $\gamma = r_1\gamma_1+\cdots+r_m\gamma_m$, so 
\[
(\log t^\gamma)_{>\delta}\ =\ \big(r_1\log t^{\gamma_1}+\cdots+r_m \log t^{\gamma_m}\big)_{>\delta} \ =\ r_1(\log t^{\gamma_1})_{>\delta}+\cdots+r_m (\log t^{\gamma_m})_{>\delta}.
\]
Since $\log(t^{B_0})$ is truncation closed by construction, each $(\log t^{\gamma_i})_{>\delta}$ is in $\log(t^{B_0})$. Thus, $(\log t^\gamma)_{>\delta} \in \log(t^{\Delta_{\omega,0}})$, so $\log(t^{\Delta_{\omega,0}})$ is truncation closed. Now, let $n$ be given and assume that $\log(t^{\Delta_{\omega,n}})$ is truncation closed. We will show that $\log(t^{\Delta_{\omega,n+1}})$ is truncation closed. As with $\Delta_{\omega,0}$, it suffices to show that $\log(t^{B_{n+1}})$ is truncation closed, and this holds since $\log(t^{B_{n+1}}) = t^{\Delta_{\omega,n}^{>0}}\cup \log(t^{\Delta_{\omega,n}})$ is the union of truncation closed sets. Now we let $\Delta := \bigcup_n \Delta_{\omega,n}$, so $\Delta$ is a logarithmic $\k$-subspace of $\Gamma$ and $\log(t^\Delta)$ is truncation closed. To see that $t^{\Delta^{>0}} \subseteq \log(t^\Delta)$, let $\delta \in \Delta^{>0}$ and take $n$ with $\delta \in \Delta_{\omega,n}^{>0}$. Then $t^\delta \in \log(t^{\Delta_{\omega,n+1}}) \subseteq \log(t^\Delta)$. It remains to note that $\Delta$ is a set. 
\end{proof}

%-----------------------%
\subsection{$\Delta$-paths} \label{subsec-deltapaths}
%-----------------------%
We continue to assume that $\Gamma$ is an ordered $\k$-vector space. We fix a proper logarithmic $\k$-subspace $\Delta\subseteq \Gamma$. A \textbf{$\Delta$-path in $\T$} is a sequence $(y_n)_{n\in \N}$ in $\T$ such that:
\begin{enumerate}
\item each $y_n$ is a positive infinite element;
\item $c(y_n) \not\in \Delta$ for each $n$;
\item $y_{n+1} = \big|\log y_n - (\log y_n)_{\T_\Delta}\big|$ for each $n$.
\end{enumerate}
Note that if $(y_n)_{n\in \N}$ is a $\Delta$-path, then so is any of its \textbf{tails} $(y_n)_{n\geq n_0}$. The definition of a $\Delta$-path is motivated by the Main Lemma in~\cite[page 286]{R}. 

\begin{lemma}\label{independence}
Let $(y_n)_{n\in \N}$ be a $\Delta$-path in $\T$. Then $c(y_0),c(y_1),\ldots \in \Gamma$ are $\k$-linearly independent over $\Delta$.
\end{lemma}
\begin{proof}
Take $n\geq 0$, $\delta \in \Delta$, and $r_0,\ldots,r_n\in \k$ with 
\[
r_0c(y_0) + \cdots+r_nc(y_n) + \delta \ =\ 0.
\]
Then $y_0^{r_0}\cdots y_n^{r_n}t^\delta =u$ for some $u \in K^{>0}$ with $u \asymp 1$. We have
\[
r_0\log y_0 + \cdots+r_n \log y_n+ \log t^\delta \ =\ \log u.
\]
For each $m$, we have $\log y_m = (\log y_m)_{\T_\Delta}+\epsilon_m y_{m+1}$ where $\epsilon_m = \pm1$. Set
\[
a\ :=\ r_0 (\log y_0)_{\T_\Delta}+ \cdots+ r_n(\log y_n)_{\T_\Delta}+\log t^\delta\ \in\ \T_\Delta, 
\]
so $r_0\log y_0 + \cdots+r_n \log y_n+ \log t^\delta=r_0\epsilon_0y_1+\cdots + r_n \epsilon_n y_{n+1} +a$. Since $u \asymp 1$, we have $\log u \preceq 1$, so 
\[
r_0\epsilon_0y_1+\cdots + r_n \epsilon_n y_{n+1} +a \ =\ \log u\ \preceq\ 1.
\]
Since $y_m \succ \log y_m \succeq y_{m+1} \succ 1$ by Lemma~\ref{lem-posinflogs} and since $y_m \not\asymp a$ for each $m$, this is only possible if each $r_m = 0$.
\end{proof}

A $\Delta$-path $(y_n)_{n\in \N}$ is said to be \textbf{$\Delta$-atomic} if $y_n = t^{c(y_n)}$ for each $n$. Finding a $\Delta$-atomic $\Delta$-path seems difficult in general, but if $\T$ is a transserial Hahn field, then any $\Delta$-path has a $\Delta$-atomic tail. 

\begin{lemma}\label{existsatomic}
Let $(y_n)_{n\in \N}$ be a $\Delta$-path in $\T$. If $\T$ is a transserial Hahn field, then there is an $n_0 \in \N$ such that $(y_n)_{n\geq n_0}$ is $\Delta$-atomic.
\end{lemma}
\begin{proof}
Since
\[
t^{c(y_{n+1})}\ \asymp\ y_{n+1}\ =\ \big|\log y_n- (\log y_n)_{\T_\Delta}\big|,
\]
we have $c(y_{n+1}) \in \supp(\log y_n)$ for each $n$, so (T4) gives $n_0 \in \N$ such that
\[
t^{c(y_{n+1})}\ =\ \big|\log y_n- (\log y_n)_{\T_\Delta}\big| \ =\ y_{n+1}
\]
for all $n\geq n_0$. Thus, $(y_n)_{n\geq n_0}$ is $\Delta$-atomic.
\end{proof}

Given a $\Delta$-path $(y_n)_{n\in \N}$, we let $\Delta_{(y_n)}$ denote the $\k$-linear span of $\Delta\cup\big\{c(y_0),c(y_1),\ldots\big\}$ in $\Gamma$. Here is a sort of converse to Lemma~\ref{existsatomic}:

\begin{lemma}\label{lem-existsatomic2}
Let $(y_n)_{n\in \N}$ be a $\Delta$-atomic $\Delta$-path in $\T$. Then $\Delta_{(y_n)}$ is a logarithmic $\k$-subspace of $\Gamma$. Moreover, if $\T_\Delta$ is a transserial Hahn field, then so is $\T_{\Delta_{(y_n)}}$.
\end{lemma}
\begin{proof}
First, let $\gamma \in \Delta_{(y_n)}$. We need to show that $\log t^\gamma \in \T_{\Delta_{(y_n)}}$. Take $\delta \in \Delta$ and $r_0,\ldots,r_n \in \k$ with $\gamma = \delta + r_0c(y_0)+\cdots + r_n c(y_n)$. Then
\[
\log t^\gamma \ =\ \log t^\delta + r_0 \log t^{c(y_0)} + \cdots + r_n \log t^{c(y_n)}\ =\ \log t^\delta + r_0 \log y_0 + \cdots + r_n \log y_n,
\]
so it suffices to show that $\log y_n \in \T_{\Delta_{(y_n)}}$ for each $n$. We have
\[
\log y_n\ =\ (\log y_n)_{\T_\Delta}\pm y_{n+1}\ = \ (\log y_n)_{\T_\Delta}\pm t^{c(y_{n+1})} \ \in \ \T_{\Delta_{(y_n)}}
\]
for each $n$, as desired.

Now, let $\gamma_0,\gamma_1,\ldots \in \Delta_{(y_n)}$ be as in the statement of (T4), and take $\delta \in \Delta$ and $r_0,\ldots,r_k \in \k$ with $\gamma_0 = \delta+r_0c(y_0) + \cdots+r_kc(y_k)$. Then

\[
\log t^{\gamma_0} \ =\ \log t^\delta+r_0\log y_0+\cdots+r_k \log y_k\ =\ \log t^\delta+r_0\big((\log y_0)_{\T_\Delta}\pm y_1\big) + \cdots+r_k\big((\log y_k)_{\T_\Delta}\pm y_{k+1}\big).
\]
Thus, $\gamma_1$ is in the support of $\log t^\delta$ or some $(\log y_m)_{\T_\Delta}$ or some $y_{m+1}$, where $m \leq k$. In the first two cases, we would have $\gamma_n \in \Delta$ for all $n>0$, and we could conclude that
\[
\big|\log t^{\gamma_n} - (\log t^{\gamma_n})_{>\gamma_{n+1}}\big|\ =\ t^{\gamma_{n+1}}
\]
for all $n$ greater than some $n_0$, since $\T_\Delta$ satisfies (T4). Thus, we may assume that $\gamma_1 \in \supp y_{m+1}$, so $t^{\gamma_1} = y_{m+1}$, since $(y_n)_{n\in \N}$ is $\Delta$-atomic. Thus, $\gamma_2$ is in the support of 
\[
\log y_{m+1} \ =\ (\log y_{m+1})_{\T_\Delta}\pm y_{m+2}. 
\]
Again, if $\gamma_2$ is in the support of $(\log y_{m+1})_{\T_\Delta}$, then we are done, so we may assume that $t^{\gamma_2} = y_{m+2}$. Continuing in this way, we see that either $\gamma_n \in \Delta$ eventually (in which case we are done), or $t^{\gamma_n} = y_{m+n}$ for all $n>0$. In this second case, we note that since $\log y_{m+n}=(\log y_{m+n})_{\T_\Delta}\pm y_{m+n+1} =(\log y_{m+n})_{\T_\Delta}\pm t^{\gamma_{n+1}}$, we have $(\log y_{m+n})_{>\gamma_{n+1}} = (\log y_{m+n})_{\T_\Delta}$. Thus
\[
\big|\log t^{\gamma_n} - (\log t^{\gamma_n})_{>\gamma_{n+1}}\big|\ =\ \big|\log y_{m+n} - (\log y_{m+n})_{>\gamma_{n+1}}\big|\ =\ y_{m+n+1} \ =\ t^{\gamma_{n+1}}
\]
for all $n>0$. We conclude that $\T_{\Delta_{(y_n)}}$ satisfies (T4).
\end{proof}

With some additional assumptions on $\Delta$, we can find $\Delta$-paths easily:

\begin{lemma}\label{expsubspacedcut}
Suppose that $\Delta$ is an exponential $\k$-subspace of $\Gamma$ and that $\log(t^\Gamma)$ is truncation closed. Let $y \in \T$ be positive and infinite with $c(y)\not\in \Delta$. Define the sequence $(y_n)_{n\in \N}$ by
\[
y_0\ :=\ y,\qquad y_{n+1}\ :=\ \big|\log y_n - (\log y_n)_{\T_\Delta}\big|.
\]
Then $(y_n)_{n\in \N}$ is well-defined (that is, each $y_n$ is nonzero) and $(y_n)_{n\in \N}$ is a $\Delta$-path.
\end{lemma}
\begin{proof}
Let $n \geq 0$ be given and suppose that $y_n$ is a positive infinite element and that $c(y_n) \not\in \Delta$. We need to show that $y_{n+1}$ is positive and infinite and that $c(y_{n+1}) \not\in \Delta$. First, we claim that $\log y_n\not\in \T_\Delta$. Since $\Delta$ is an exponential $\k$-subspace of $\Gamma$, if $\log y_n$ were in $\T_\Delta$, then $\log y_n$ would be in $\log(\T_\Delta^{>0})$, giving $y_n \in \T_\Delta^{>0}$, a contradiction. Since $(\log y_n)_{\T_\Delta}$ is in $\T_{\Delta}$, it follows that $y_{n+1}$ is nonzero. Since $\T_\Delta$ is a Hahn space over $\R$, Lemma~\ref{notsoclose} gives that 
\[
y_{n+1} \ =\ \big|\log y_n - (\log y_n)_{\T_\Delta}\big| \ \not\asymp\ t^\delta
\]
for any $\delta \in \Delta$, which gives $c(y_{n+1}) \not\in \Delta$. To see that $y_{n+1}$ is infinite, suppose towards contradiction that 
\[
y_{n+1} \ =\ \big|\log y_n - (\log y_n)_{\T_\Delta}\big| \ \preceq\ 1.
\]
Since $\log(t^\Gamma)$ is truncation closed, $\log(\T^{>0})$ is as well by Lemma~\ref{lem-monomtoall}, so $(\log y_n)_{\T_\Delta} \in \log(\T^{>0})$. Since $(\log y_n)_{\T_\Delta} \in \T_\Delta$ and $\Delta$ is an exponential $\k$-subspace of $\Gamma$, we have $(\log y_n)_{\T_\Delta} \in \log(\T_\Delta^{>0})$. Take $z \in \T_\Delta^{>0}$ with $(\log y_n)_{\T_\Delta} = \log z$. Then 
\[
\log y_n - \log z\ =\ \log(y_n/z) \ \preceq\ 1.
\]
Lemma~\ref{lem-boundedtobounded} gives $y_n/ z \asymp 1$, so $y_n \asymp z \in \T_\Delta$, a contradiction.
\end{proof}

%------------------------------------------------%
\section{The surreals as a logarithmic Hahn field} \label{sec:SLHF}
%------------------------------------------------%
Making use of Conway names, Conway~\cite[page 33; also see~\cite{EH5, EH1, AL}]{CO} showed that $\No$ has the structure of a Hahn field $ \R((\omega^\No))_{\On}$. Moreover, by item (i) of Proposition~\ref{Gonshor 1986}, $\log$ extends the natural logarithm on $\R^{>0}$. Furthermore, since $\log$ is analytic, an inductive definition of $\log(1+\epsilon) = \sum_{n=1}^\infty (-1)^{n-1}\epsilon^n/n$ for all infinitesimal $\epsilon \in \No$ can be provided in the manner discussed in~\cite{F}. Accordingly, since $\log(\omega^\gamma)$ is purely infinite by Proposition~\ref{Gonshor 1986.1}, we have the following:

\begin{proposition}\label{corsurrealsarelog}
The surreal numbers are a logarithmic Hahn field.
\end{proposition}

We will see in Proposition~\ref{surrealsaretransseries} below that the surreal numbers are even a transserial Hahn field.

%-----------------------%
\subsection{Extending logarithmic subspaces of $\No$}\label{subsec-exsub}
%-----------------------%
In this subsection, let $\k$ be an Archimedean ordered field. Note that $\log \omega^{r\delta} =r\log \omega^{\delta}$ for each $r \in \k$ and each $\delta \in \Delta$ (this is an easy consequence of Proposition~\ref{Gonshor 1986.2}).
Let $\Delta$ be an initial logarithmic $\k$-subspace of $\No$. Then $\No_\Delta = \R((\omega^\Delta))_{\On}$ is an initial logarithmic Hahn subfield of $\No$. 

\begin{lemma}\label{expincut0}
Let $z \in \omega^{\Delta^{>0}}\setminus \log(\omega^\Delta)$ and suppose that $y \in \log(\omega^\Delta)$ for all $y \in \omega^{\Delta^{>0}}$ with $y <_s z$. Let $\gamma\in \No$ with $\omega^\gamma = \exp z$. Then $\gamma$ is the simplest element realizing a cut over $\Delta$ and $\Delta+\k \gamma$ is an initial logarithmic $\k$-subspace of $\No$.
\end{lemma}
\begin{proof}
By~\cite[Theorem 3.8]{BM}, we have
\[
\omega^\gamma \ =\ \exp z\ =\ \{z^k,\, \exp z_L\, |\, \exp z_R\}
\]
where $k$ ranges over the positive integers and $z_L$, $z_R$ range over the predecessors of $z$ in $\omega^{\Delta^{>0}}$ with $z_L<z<z_R$. By assumption, $\exp z_L$ and $\exp z_R$ are in $\omega^\Delta$ for all such $z_L$ and $z_R$. Of course, $z^k\in \omega^\Delta$, so 
\[
\omega^\gamma \ =\ \{0,\, \omega^{\gamma_L}\, |\, \omega^{\gamma_R}\},
\]
where $\gamma_L$ and $\gamma_R$ range over elements of $\Delta$ with $\gamma_L<\gamma<\gamma_R$. Since the map $x \mapsto \omega^x:\No\to \No$ preserves simplicity by~\cite[Theorem 12]{EH5}, it follows that $\gamma$ is the simplest element realizing a cut over $\Delta$. Thus, $\Delta+\k \gamma$ is initial by Lemma~\ref{newbasiselem}. To see that $\Delta+\k\gamma$ is a logarithmic $\k$-subspace of $\No$, let $\delta \in \Delta$ and let $r \in \k$. Since $\Delta$ is a logarithmic $\k$-subspace of $\No$ and $\log \omega^\gamma = z\in \No_\Delta$, we see that $\log \omega^{\delta+r\gamma} = \log \omega^\delta + r\log \omega^\gamma \in \No_\Delta$. 
\end{proof}

\begin{lemma}\label{expincut}
Suppose that $\omega^{\Delta^{>0}} \subseteq \log(\omega^\Delta)$. Let $z \in \No_\Delta$ be purely infinite with $z \not\in \log(\omega^\Delta)$ and suppose that $y \in\log(\omega^\Delta)$ for every proper truncation $y$ of $z$. Let $\gamma\in \No$ with $\omega^\gamma = \exp z$. Then $\gamma$ is the simplest element realizing a cut over $\Delta$ and $\Delta+\k \gamma$ is an initial logarithmic $\k$-subspace of $\No$.
\end{lemma}
\begin{proof}
By~\cite[Theorem 3.8]{BM}, we have
\[
\omega^\gamma \ =\ \exp z\ =\ \big\{0,\, \exp (z_{>\delta}+q_L\omega^\delta)\, \big|\, \exp(z_{>\delta}+q_{ R}\omega^\delta)\big\},
\]
where $\delta$ ranges over $\supp(z)$ and where $q_L$ and $q_R$ range over rational numbers with $q_L\omega^\delta< z-z_{>\delta}<q_R\omega^\delta$. Let $\delta \in \supp(z)$, so $z_{>\delta} \in \log(\omega^\Delta)$ by assumption. Since $z$ is purely infinite, $\delta$ is positive, so $\omega^\delta \in \log(\omega^\Delta)$ by assumption. As $\log(\omega^\Delta)$ is a $\k$-subspace of $\No$, this gives $z_{>\delta}+q\omega^\delta \in\log(\omega^\Delta)$ for each $q \in \Q$. As in the previous lemma, we have
\[
\omega^\gamma \ =\ \{0,\, \omega^{\gamma_L}\, |\, \omega^{\gamma_R}\},
\]
where $\gamma_L$ and $\gamma_R$ range over elements of $\Delta$ with $\gamma_L<\gamma<\gamma_R$. Again, this shows that $\gamma$ is the simplest element realizing a cut over $\Delta$, so $\Delta+\k \gamma$ is initial by Lemma~\ref{newbasiselem}. As in the previous lemma, $\Delta+\k\gamma$ is a logarithmic $\k$-subspace of $\No$ since $\log \omega^\gamma = z \in \No_\Delta$.
\end{proof}

\begin{corollary}\label{corEisinitial}
$\Delta^E$ is an initial exponential $\k$-subspace of $\No$.
\end{corollary}
\begin{proof}
Note that $\log(\omega^{\No})$ consists of all purely infinite members of $\No$, so $\Delta$ is itself an exponential $\k$-subspace of $\No$ if and only if each purely infinite element in $\T_\Delta$ is also in $\log(\omega^\Delta)$.
Suppose that $\Delta \neq \Delta^E$. We will find an initial logarithmic $\k$-subspace $\Gamma$ of $\No$ which is contained in $\Delta^E$ and which properly contains $\Delta$. First, if $\omega^{\Delta^{>0}}\not\subseteq \log(\omega^\Delta)$, then choose $z \in \omega^{\Delta^{>0}}\setminus \log(\omega^\Delta)$ such that $y \in \log(\omega^\Delta)$ for all $y \in \omega^{\Delta^{>0}}$ with $y <_s z$. Let $\gamma\in \No$ with $\omega^\gamma = \exp z$. Then $\gamma \in \Delta^E$, and $\Gamma:=\Delta+\k \gamma\subseteq \Delta^E$ is an initial logarithmic $\k$-subspace of $\No$ by Lemma~\ref{expincut0}. 

Now, suppose that $\omega^{\Delta^{>0}}\subseteq \log(\omega^\Delta)$ and let $z \in \No_\Delta \setminus \log(\omega^\Delta)$ be purely infinite. We may assume that $y \in\log(\omega^\Delta)$ for every proper truncation $y$ of $z$. Let $\gamma\in \No$ with $\omega^\gamma = \exp z$. Then $\gamma \in \Delta^E$, and $\Gamma:=\Delta+\k \gamma$ is an initial logarithmic $\k$-subspace of $\No$ by Lemma~\ref{expincut}.
\end{proof}

Before proceeding, we need the following corollary of Proposition~\ref{Gonshor 1986.2}.

\begin{corollary}\label{cortruncoflog}
Let $\gamma \in \No\setminus \Delta$. Then $\log \omega^{\gamma_\Delta}$ is a truncation of $\log \omega^\gamma$ belonging to $\No_\Delta$.
\end{corollary}
\begin{proof}
Let $\gamma = \sum_{\alpha<\beta} r_\alpha \omega^{s_\alpha}$ and take $\beta_0 \leq \beta$ with $\gamma_\Delta = \sum_{\alpha<\beta_0} r_\alpha \omega^{s_\alpha}$. Proposition~\ref{Gonshor 1986.2} gives
\[
\log \omega^{\gamma}\ =\ \sum_{\alpha<\beta} r_\alpha \omega^{h(s_\alpha)},\qquad \log \omega^{\gamma_\Delta}\ =\ \sum_{\alpha<\beta_0} r_\alpha \omega^{h(s_\alpha)}.
\]
Since $h$ is strictly increasing, we see that $\log \omega^{\gamma_\Delta}$ is a truncation of $\log \omega^\gamma$. To see that $\log \omega^{\gamma_\Delta}$ belongs to $\No_\Delta$, let $S$ be the value class of $\Delta$. Then $\Delta$ is a truncation closed, cross sectional $\k$-subspace of $\R((\omega^S))_{\On}$ by~\cite[Theorem 5.1]{EK}, so $\omega^{s_\alpha}$ belongs to $\Delta$ for each $\alpha<\beta_0$. Since $\Delta$ is a logarithmic $\k$-subspace of $\No$, we have $\log \omega^{\omega^{s_\alpha}}= \omega^{h(s_\alpha)} \in \No_\Delta$ for each $\alpha<\beta_0$. Thus, the sum $\sum_{\alpha<\beta_0} r_\alpha \omega^{h(s_\alpha)}$ belongs to $\No_\Delta$ as well. 
\end{proof}

\begin{lemma}\label{deltaatomicz}
Suppose that $\omega^{\Delta^{>0}} \subseteq \log(\omega^\Delta)$ and let $(z_n)_{n \in \N}$ be a $\Delta$-path in $\No$. Suppose also that:
\begin{enumerate}
\item $z_0$ is the simplest element realizing a cut over $\No_\Delta$;
\item $c(z_n)_\Delta \in \Delta$ for each $n$;
\item $\Delta_{(z_n)}$ is a Hahn space over $\k$.
\end{enumerate}
Then $(z_n)_{n\in \N}$ is $\Delta$-atomic and $\Delta_{(z_n)}$ is an initial logarithmic $\k$-subspace of $\No$.
\end{lemma}
\begin{proof}
We set $\Delta_0:= \Delta$ and for each $n$, we set 
\[
\gamma_n\ :=\ c(z_n),\qquad \Delta_{n+1}\ :=\ \Delta_n+\k \gamma_n.
\]
We will show by induction on $n$ that $z_n = \omega^{\gamma_n}$ and that $\gamma_n$ is the simplest element realizing a cut over $\Delta_n$. If we can show this, then Lemmas~\ref{newbasiselem} and~\ref{lem-existsatomic2} give us that $\Delta _{(z_n)}$ is an initial logarithmic $\k$-subspace of $\No$. 

We begin with $n = 0$. We have $z_0>0$ and $\gamma_0 = c(z_0) \not\in \Delta$ since $(z_n)_{n\in \N}$ is a $\Delta$-path. Thus, $z_0$ and $\omega^{\gamma_0}$ realize the same cut over $\No_\Delta$, so $z_0 \leq_s \omega^{\gamma_0}$ by our simplicity assumption 1 on $z_0$. Since $\omega^{\gamma_0}$ is a leader, $z_0$ is positive, and $\omega^{\gamma_0}$ and $z_0$ are Archimedean equivalent (i.e. $\omega^{\gamma_0} \asymp z_0$), this gives $z_0=\omega^{\gamma_0}$. Our simplicity assumption also gives that $z_0=\omega^{\gamma_0}$ is the simplest element realizing a cut over $\omega^\Delta$, so $\gamma_0$ is the simplest element realizing a cut over $\Delta$ by~\cite[Theorem 12]{EH5}.

Now let $n\geq 0$. We make the inductive assumption that the following hold:
\begin{enumerate}[(i)]
\item $z_m = \omega^{\gamma_m}$ for $m\leq n$;
\item $\gamma_m$ is the simplest element realizing a cut over $\Delta_m$ for $m \leq n$;
\item There are $s_0>s_1>\cdots>s_{n-1} \in \No$ with $\gamma_m = (\gamma_m)_\Delta\pm \omega^{s_m}$ and $\gamma_{m+1} = h(s_m)$ for $m< n$.
\end{enumerate}
Assumption (ii) and Lemma~\ref{newbasiselem} give us that $\Delta_{n+1}$ is initial. Assumption (iii) holds trivially for $n = 0$. 

For $m \leq n$, let $S_m \subseteq \No$ be the value class of $\Delta_m$. Since $\Delta_m$ is initial for $m \leq n$ by (ii) and Lemma~\ref{newbasiselem}, we see that each $S_m$ is initial and that $\Delta_m$ is a truncation closed, cross sectional subspace of $\R((\omega^{S_m}))_{\On}$ for $m \leq n$; see~\cite[Theorem 5.1]{EK}. Moreover, (iii) gives that $S_m = S_0\cup\{s_0,s_1,\ldots,s_{m-1}\}$ for $m \leq n$. 

Take $s \in \No$ with $\gamma_n- (\gamma_n)_\Delta\asymp \omega^s$. We will show that $s<s_{n-1}$, that $\gamma_n = (\gamma_n)_\Delta\pm \omega^s$, and that $\gamma_{n+1} = h(s)$, thereby proving (iii) with $s_n:= s$. We will use this to prove (i) and (ii). First note that $\gamma_n- (\gamma_n)_\Delta\not\asymp \delta$ for any $\delta \in \Delta$ by Lemma~\ref{notsoclose}, so $s\not\in S_0$. If $n = 0$, then the relation $s<s_{n-1}$ holds trivially. If $n>0$, then $\gamma_n=h(s_{n-1})$, so 
\[
\omega^s\ \asymp\ \gamma_n- (\gamma_n)_\Delta\ \preceq\ \gamma_n\ =\ h(s_{n-1})\ \prec\ \omega^{s_{n-1}},
\]
where the last inequality follows from the definition of $h$. Thus, $s<s_{n-1}$ as desired. Since $s\not\in S_0$ and $s<s_{n-1}<s_{n-2}<\cdots < s_0$, it follows that $s \not\in S_n$. 

To see that $\gamma_n = (\gamma_n)_\Delta\pm \omega^s$, take $r\in \R\setminus \{0\}$ and $z \prec \omega^s$ with $\gamma_n = (\gamma_n)_\Delta+ r\omega^s+z$. Let $\epsilon \in \{\pm 1\}$ be the sign of $r$. We claim that $\gamma_n$ and $(\gamma_n)_\Delta+\epsilon\omega^s$ realize the same cut over $\Delta_n$. Suppose towards contradiction that there is $\delta \in \Delta_n$ lying between $\gamma_n$ and $(\gamma_n)_\Delta+\epsilon\omega^s$. Then $\delta - (\gamma_n)_\Delta$ lies between $r\omega^s+z$ and $\epsilon \omega^s$, so $\omega^s \asymp \delta - (\gamma_n)_\Delta\in \Delta_n$. This gives $s \in S_n$, a contradiction. It follows from our claim and Lemma~\ref{lem-conwaytrunc} that $(\gamma_n)_\Delta+\epsilon\omega^s \leq_s \gamma_n$, so $(\gamma_n)_\Delta+\epsilon\omega^s = \gamma_n$ since $\gamma_n$ is the simplest element realizing a cut over $\Delta_n$ by (ii). 

Now, we show that $\gamma_{n+1} = h(s)$. First, we claim that $\log \omega^{(\gamma_n)_\Delta} = (\log \omega^{\gamma_n})_{\No_\Delta}$. Since $\log \omega^{(\gamma_n)_\Delta}\in \No_\Delta$ is a truncation of $\log \omega^{\gamma_n}$ by Corollary~\ref{cortruncoflog}, it suffices to show that $\log \omega^{\gamma_n}-\log \omega^{(\gamma_n)_\Delta}\not\asymp \omega^\delta$ for any $\delta \in \Delta$; see Lemma~\ref{notsoclose}. Since $\gamma_n = (\gamma_n)_\Delta\pm\omega^s$, we have
\[
\log \omega^{\gamma_n}-\log \omega^{(\gamma_n)_\Delta} \ =\ \log \omega^{\pm\omega^s}\ =\ \pm \omega^{h(s)}.
\]
Thus, it suffices to show that $h(s) \not\in \Delta$. If $h(s)$ were in $\Delta$, then since $h(s)>0$, we would have $\omega^{h(s)} =\log \omega^{\omega^s}\in \log(\omega^\Delta)$ by assumption, so $\omega^s \in \Delta$. This would give $s \in S_0$, since $\Delta$ is a cross sectional subspace of $\R((\omega^{S_0}))_{\On}$, a contradiction. Thus, $\log \omega^{(\gamma_n)_\Delta} = (\log \omega^{\gamma_n})_{\No_\Delta}$ as claimed. With this claim taken care of, we may apply our inductive assumption $z_n = \omega^{\gamma_n}$ to get 
\[
z_{n+1}\ =\ \big|\log z_n- (\log z_n)_{\No_\Delta}\big|\ =\ \big|\log \omega^{\gamma_n}- (\log \omega^{\gamma_n})_{\No_\Delta}\big|\ =\ \big|\log \omega^{\gamma_n}- \log \omega^{(\gamma_n)_\Delta}\big|\
=\ \omega^{h(s)}. 
\]
This establishes both that $\gamma_{n+1} = h(s)$ and that $z_{n+1} = \omega^{\gamma_{n+1}}$, taking care of (i) and (ii). It remains to show that $\gamma_{n+1} = h(s)$ is the simplest element realizing a cut over $\Delta_{n+1}$. Since $\Delta_{n+1}$ is initial, we know that its value class $S_n\cup \{s\}$ is initial as well. Therefore, all predecessors of $s$ lie in $S_n$. Since $\omega^{S_0} \subseteq \Delta$ and $\Delta$ is a logarithmic $\k$-subspace of $\No$, we have $\omega^{h(S_0)} = \log(\omega^{\omega^{S_0}}) \subseteq \omega^\Delta$, so $h(S_0) \subseteq \Delta$. Since
\[
h(S_n) \ =\ h(S_0) \cup \big\{h(s_0),\ldots,h(s_{n-1})\big\}\ =\ h(S_0) \cup \{\gamma_1,\ldots,\gamma_n\},
\]
we have $h(S_n) \subseteq \Delta_{n+1}$. Additionally, we have
\[
\omega^s\ =\ \big|\gamma_n- (\gamma_n)_\Delta\big|\ \in\ \Delta_{n+1}
\]
and so $\frac{1}{k}\omega^s \in \Delta_{n+1}$ for each $k$. Since all left and right options in the definition of $h(s)$ lie in $\Delta_{n+1}$, we see that $\gamma_{n+1} = h(s)$ is the simplest element realizing a cut over $\Delta_{n+1}$.
\end{proof}

%------------------------------------------------%
\subsection{Surreal numbers as transseries}\label{subsec-surrealseries}
%------------------------------------------------%
In this subsection, we use the extension lemmas established in the previous subsection with $\R$ in place of $\k$ to prove the following:

\begin{proposition}\label{surrealsaretransseries}
The surreal numbers are a transserial Hahn field.
\end{proposition}

Proposition~\ref{surrealsaretransseries} was first proven by Berarducci and Mantova, in the course of establishing a derivation on the surreal numbers~\cite[Theorem 8.10]{BM}. While their proof uses their \emph{nested truncation rank}, our proof relies instead on logarithmic and exponential $\R$-subspaces and $\Delta$-paths.

\begin{proof}[Proof of Proposition~\ref{surrealsaretransseries}]
The surreal numbers are a logarithmic Hahn field by Corollary~\ref{corsurrealsarelog}. Thus, $\No_\Delta$ is also a logarithmic Hahn field whenever $\Delta$ is a logarithmic $\R$-subspace of $\No$. Our task is to show that $\No$ satisfies axiom (T4). Let $\Delta \subseteq \No$ be an initial logarithmic $\R$-subspace and suppose that $\No_\Delta$ satisfies (T4). Such a subspace exists, for example $\Delta = \{0\}$. We will find an initial logarithmic $\R$-subspace $\Gamma\subseteq \No$ properly containing $\Delta$ such that $\No_\Gamma$ satisfies (T4).

First, if $\Delta$ is not an exponential $\R$-subspace of $\No$, then we let $\Gamma := \Delta^E$. Then $\Gamma$ is initial by Corollary~\ref{corEisinitial} and $\No_\Gamma$ satisfies (T4) by Lemma~\ref{defofDeltaE}. Now, suppose that $\Delta$ is an exponential $\R$-subspace of $\No$. Let $\gamma \in \No^{>0} \setminus \Delta^{>0}$ be the simplest element realizing a cut over $\Delta$. Then $z_0:= \omega^\gamma$ is the simplest element realizing a cut over $\No_\Delta$. For each $n$, let 
\[
z_{n+1}\ :=\ \big|\log z_n - (\log z_n)_{\No_\Delta}\big|.
\]
Then $(z_n)_{n\in \N}$ is a $\Delta$-path in $\No$ by Lemma~\ref{expsubspacedcut}. To see that $c(z_n)_\Delta \in \Delta$ for each $n$, let $n$ be given and note that $\log \omega^{c(z_n)_\Delta } \in \No_\Delta$ by Corollary~\ref{cortruncoflog}. Since $\Delta$ is an exponential $\R$-subspace of $\No$, this gives $\log \omega^{c(z_n)_\Delta } \in \log(\omega^\Delta)$, so $c(z_n)_\Delta \in \Delta$. Thus, we may apply Lemma~\ref{deltaatomicz} to see that $(z_n)_{n\in \N}$ is $\Delta$-atomic and that $\Gamma:=\Delta_{(z_n)}$ is an initial logarithmic $\R$-subspace of $\No$. Lemma~\ref{lem-existsatomic2} tells us that $\No_\Gamma$ satisfies (T4).
\end{proof}

%------------------------------------------------%
\section{Transserial embeddings}\label{sec:IT}
%------------------------------------------------%
Let $\T=\R((t^\Gamma))_{\On}$ be a logarithmic Hahn field and let $\imath:\Gamma\to \No$ be an ordered group embedding. Then $\imath$ induces an ordered field embedding $\tilde{\imath}:\T \to\No$ given by
\[
\tilde{\imath}\Big(\sum_{\alpha<\beta}r_\alpha t^{\gamma_\alpha} \Big)\ =\ \sum_{\alpha<\beta}r_\alpha \omega^{\imath(\gamma_\alpha)}.
\]
The image of the map $\tilde{\imath}$ is $\No_{\imath(\Gamma)}$. If $\imath$ is an initial embedding (that is, an embedding with initial image), then so is $\tilde{\imath}$ by Proposition~\ref{Ehrlich initial 2}.

We say that $\imath$ is a \textbf{$\log$-embedding} if 
\[
\tilde{\imath}(\log t^\gamma)\ =\ \log\omega^{\imath(\gamma)}
\]
for each $\gamma \in \Gamma$. If $\imath:\Gamma\to \No$ is a $\log$-embedding, then $\tilde{\imath}(\log x)= \log \tilde{\imath}(x)$ for all $x \in \T^{>0}$. Indeed, let $x\in \T^{>0}$ be given and take $\gamma \in \Gamma$, $r \in \R^{>0}$, and $\epsilon \prec 1$ with $x = rt^\gamma(1+\epsilon)$. We have
\begin{align*}
\tilde{\imath}(\log x) \ &= \ \tilde{\imath}(\log t^\gamma) +\tilde{\imath}(\ln r)+ \tilde{\imath}\Big(\sum_{n=1}^{\infty } (-1)^{n-1} \epsilon^n/n\Big)\\
&= \ \log \omega^{\imath(\gamma)} + \ln r + \sum_{n=1}^{\infty } (-1)^{n-1}\tilde{\imath}(\epsilon)^n/n\ =\ \log \tilde{\imath}(x),\end{align*}
where the second equality uses that $\tilde{\imath}$ is a $\log$-embedding and that $\tilde{\imath}$ is $\R$-linear and respects infinite sums. Finally, we define a \textbf{transserial embedding} to be any embedding $\T\to \No$ of the form $\tilde{\imath}$ for some $\log$-embedding $\imath$. By Proposition~\ref{surrealsaretransseries}, the surreal numbers are a transserial Hahn field, so we have the following:

\begin{fact}\label{embeddingtotransseries}
If $\T$ admits a transserial embedding into $\No$, then $\T$ is a transserial Hahn field.
\end{fact}

In Proposition~\ref{transseriestoembedding} below, we prove a converse to Fact~\ref{embeddingtotransseries}: any transserial Hahn field admits a transserial embedding into $\No$. Before proving Proposition~\ref{transseriestoembedding}, we provide sufficient conditions for when a transserial Hahn field admits an \emph{initial} transserial embedding into $\No$ in Theorem~\ref{thmembedhahn}. A corollary of this theorem, Corollary~\ref{corembedhahn}, is used both in the proof of Proposition~\ref{transseriestoembedding} and in our main result: Theorem~\ref{T1}. We remark on the necessity of the conditions in Theorem~\ref{thmembedhahn} at the end of this section.

%-----------------------%
\subsection{Initial transserial embeddings}\label{subsec-transserialembeddings}
%-----------------------%
In this subsection, we prove the following:
 
\begin{theorem}\label{thmembedhahn}
Let $\k$ be an Archimedean ordered field, let $\T = \R((t^\Gamma))_{\On}$ be a transserial Hahn field, and suppose that $\Gamma$ is a Hahn space over $\k$, that $\log(t^\Gamma)$ is truncation closed, and that $t^{\Gamma^{>0}} \subseteq \log(t^\Gamma)$. Then $\T$ admits an initial transserial embedding into $\No$.
\end{theorem}

For the remainder of this subsection, we fix a transserial Hahn field $\T = \R((t^\Gamma))_{\On}$ and an Archimedean ordered field $\k$. We assume that $\Gamma$ is a Hahn space over $\k$, that $\log(t^\Gamma)$ is truncation closed, and that $t^{\Gamma^{>0}} \subseteq \log(t^\Gamma)$. The assumption that $\T$ satisfies (T4) will not be used until the proof of Proposition~\ref{propembedhahn}.

\begin{proposition}\label{extendtoE}
Let $\Delta$ be a logarithmic $\k$-subspace of $\Gamma$ and let $\imath:\Delta\to \No$ be an initial $\log$-embedding. Then $\imath$ extends uniquely to an initial $\log$-embedding $\jmath:\Delta^E \to \No$.
\end{proposition}
\begin{proof}
We first show that any such embedding is unique. Let $\jmath_1,\jmath_2:\Delta^E \to \No$ be any two $\log$-embeddings extending $\imath$ and let $(\Delta_\alpha)_{\alpha<\On}$ be the continuous chain of logarithmic $\k$-subspaces of $\Gamma$ in the proof of Lemma~\ref{defofDeltaE}. We know that $\jmath_1$ and $\jmath_2$ agree with each other (and with $\imath$) on $\Delta_0 = \Delta$. If $\jmath_1$ and $\jmath_2$ agree on $\Delta_\beta$ for all $\beta$ less than a limit ordinal $\alpha$, then they agree on $\Delta_\alpha$ as well. We will show that if they agree on $\Delta_\alpha$, then they must agree on $\Delta_{\alpha+1}$. Let $\gamma \in \Delta_{\alpha+1}$. 
Then 
\[
\log\omega^{\jmath_1(\gamma)}\ =\ \tilde{\jmath}_1(\log t^\gamma)\ =\ \tilde{\jmath}_2(\log t^\gamma)\ =\ \log\omega^{\jmath_2(\gamma)},
\]
where the middle equality follows from the fact that $\log t^\gamma \in \T_{\Delta_\alpha}$.
Injectivity of the logarithm yields $\jmath_1(\gamma)=\jmath_2(\gamma)$. Thus, $\jmath_1$ and $\jmath_2$ agree on $\Delta^E$.

We now show that such an embedding exists. Suppose that $\Delta \neq \Delta^E$, so $\log(t^\Delta) \neq \log(t^\Gamma)\cap \T_\Delta$. It is enough to show that $\imath$ can be extended to an initial $\log$-embedding $\imath^*:\Delta^*\to \No$ where $\Delta^*$ is a logarithmic $\k$-subspace of $\Delta^E$ properly containing $\Delta$. We choose an element $x \in \log(t^\Gamma)\cap \T_\Delta$ with $x \not\in \log(t^\Delta)$ as follows:
\begin{enumerate}
\item If $t^{\Delta^{>0}}\not\subseteq \log(t^\Delta)$, then we let $x$ be an element of $t^{\Delta^{>0}}\setminus\log(t^\Delta)$ which is simplest in the following sense: for each $y \in t^{\Delta^{>0}}$, if $\tilde{\imath}(y) <_s \tilde{\imath}(x)$, then $y \in \log(t^\Delta)$.
\item If $t^{\Delta^{>0}}\subseteq \log(t^\Delta)$, then we let $x$ be an element of $\log(t^\Gamma)\cap \T_\Delta$ with $x\not\in\log(t^\Delta)$ such that every proper truncation of $x$ is in $\log(t^\Delta)$.
\end{enumerate}
Since $t^{\Delta^{>0}}\subseteq t^{\Gamma^{>0}}\subseteq \log(t^\Gamma)$ and $\log(t^\Gamma)$ is truncation closed, we can always find an element satisfying one of these conditions. Let $\gamma \in \Gamma$ with $t^\gamma = \exp x$, so $\gamma \in \Delta^E \setminus \Delta$. Let $z:= \tilde{\imath}(x)$ and let $\gamma^* \in \No$ with $\omega^{\gamma^*}= \exp z$. We claim that $\gamma^*$ realizes the same cut over $\imath(\Delta)$ that $\gamma$ realizes over $\Delta$. To see this, take $\delta \in \Delta$ and note that
\begin{align*}
\delta < \gamma\ &\Longleftrightarrow\ t^\delta<t^\gamma\ \Longleftrightarrow\ \log t^\delta < x\ \Longleftrightarrow\ \tilde{\imath}(\log t^\delta)<z\\
&\Longleftrightarrow\ \log\omega^{\imath(\delta)}<z\ \Longleftrightarrow\ \omega^{\imath(\delta)}< \omega^{\gamma^*}\ \Longleftrightarrow\ \imath(\delta)< \gamma^*.
\end{align*}
We set $\Delta^* := \Delta+\k\gamma$, and we extend $\imath$ to an ordered $\k$-vector space embedding $ \imath^*: \Delta^*\to \No$ by setting $\imath^*(\gamma):= \gamma^*$. Then $\imath^*(\Delta^*) = \imath(\Delta) + \k\gamma^*$ is an initial logarithmic $\k$-subspace of $\No$; this follows from Lemma~\ref{expincut0} if $x$ is chosen as in (1) and from Lemma~\ref{expincut} if $x$ is chosen as in (2). It remains to show that $\Delta^*$ is a logarithmic $\k$-subspace of $\Gamma$ and that $\imath^*$ is a $\log$-embedding. Consider an element $\delta + r\gamma$ where $\delta \in \Delta$ and where $r \in \k$. We have
\[
\tilde{\imath}^*(\log t^{\delta + r\gamma})\ =\ \tilde{\imath}(\log t^\delta+ rx)\ =\ \log\omega^{\imath(\delta)}+rz\ =\ \log\omega^{\imath(\delta)}+ \log\omega^{r\gamma^*}\ =\ \log\omega^{\imath^*(\delta+r\gamma)}. 
\]
This completes the proof.
\end{proof}

\begin{proposition}\label{extendtoy}
Let $\Delta$ be an exponential $\k$-subspace of $\Gamma$ whose universe is a set, let $\imath:\Delta\to \No$ be an initial $\log$-embedding, and let $(y_n)_{n\in \N}$ be a $\Delta$-atomic $\Delta$-path in $\T$. Then $\imath$ extends to an initial $\log$-embedding $\jmath:\Delta_{(y_n)}\to \No$.
\end{proposition}
\begin{proof}
For each $n$, set $d_n:= (\log y_n)_{\T_\Delta}$ and let $d_n^*:= \tilde{\imath}(d_n) \in \No_{\imath(\Delta)}$. Let $z_0\in \No$ be an element realizing the same cut over $\tilde{\imath}(\T_\Delta) = \No_{\imath(\Delta)}$ that $y_0$ realizes over $\T_\Delta$ (we can find such an element because $\Delta$ is a set). We arrange that $z_0$ is the simplest element in this cut. We define a sequence $(z_n)_{n\in \N}$ in $\No$, starting with $z_0$, by setting
\[
z_{n+1}\ :=\ |\log z_n-d_n^*|\ \text{ for all }n.
\]
We claim that $z_n$ realizes the same cut over $\tilde{\imath}(\T_\Delta)$ that $y_n$ realizes over $\T_\Delta$ for each $n\in \N$. This holds by assumption for $n = 0$, so we assume that this holds for a given $n$, and we will show that this holds for $n+1$. Since $\Delta$ is an exponential $\k$-subspace of $\Gamma$, we have $t^{\Delta^{>0}} \subseteq\log(t^\Delta)$, so by Lemma~\ref{deltacuts}, it is enough to show that $z_{n+1}$ realizes the same cut over $\log\tilde{\imath}(\T_\Delta^{>0})$ that $y_{n+1}$ realizes over $\log(\T_\Delta^{>0})$. Moreover, since $y_{n+1} = |\log y_n - d_n|$ and $z_{n+1} = |\log z_n-d_n^*|$, this reduces to showing that $\log y_n - d_n$ and $\log z_n - d_n^*$ realize the same cut over $\log(\T_\Delta^{>0})$ and $\log\tilde{\imath}(\T_\Delta^{>0})$, respectively. Since $d_n$ is a truncation of $\log y_n$ and $\log(\T^{>0})$ is truncation closed, we have $d_n \in \log(\T^{>0})$. Thus, $d_n \in \log(\T_\Delta^{>0})$, since $\Delta$ is an exponential $\k$-subspace of $\Gamma$, so for $a \in \log(\T_\Delta^{>0})$, we have
\[
\log y_n-d_n<a\ \Longleftrightarrow\ y_n<\exp(a+d_n),
\]
where $\exp(a+d_n) \in \T_\Delta^{>0}$. Since $\imath$ is assumed to be a $\log$-embedding, we have $\tilde{\imath}\big(\exp(a+d_n)\big) = \exp\tilde{\imath}(a+d_n)$. In light of our inductive assumption on $y_n$, this gives
\[
\log y_n<a+d_n\ \Longleftrightarrow \ y_n< \exp(a+d_n)\ \Longleftrightarrow \ z_n< \exp \tilde{\imath}(a+d_n)\ \Longleftrightarrow \log z_n - d_n^*<\tilde{\imath}(a),
\]
as desired. 

Now that we know that $z_n$ realizes the same cut over $\tilde{\imath}(\T_\Delta)$ that $y_n$ realizes over $\T_\Delta$ for each $n$, we may deduce that each $z_n$ is positive and infinite with $c(z_n) \not\in \imath(\Delta)$. Moreover, Lemma~\ref{samecutsametrunc} tells us that $d_n^* = (\log z_n)_{\tilde{\imath}(\T_\Delta)}$ for each $n$. Thus, $(z_n)_{n\in \N}$ is an $\imath(\Delta)$-path in $\No$. We claim that $c(z_n)_{\imath(\Delta)} \in \imath(\Delta)$ for each $n$. By Corollary~\ref{cortruncoflog}, we see that $\log \omega^{c(z_n)_{\imath(\Delta)}}$ is a truncation of $\log \omega^{c(z_n)}$ belonging to $\No_{\imath(\Delta)}$, so $\log \omega^{c(z_n)_{\imath(\Delta)}}$ is even a truncation of $(\log \omega^{c(z_n)})_{\No_{\imath(\Delta)}}$. Since $\log \omega^{c(z_n)}$ and $\log z_n$ realize the same cut over $\No_{\imath(\Delta)}$ by Lemma~\ref{deltacuts2}, this tells us that $\log \omega^{c(z_n)_{\imath(\Delta)}}$ is a truncation of $(\log z_n)_{\No_{\imath(\Delta)}} = d_n^*$. Since $d_n^* = \tilde{\imath}(d_n)$ and $d_n$ belongs to the truncation closed set $\log(t^\Delta)$, we see that $\log \omega^{c(z_n)_{\imath(\Delta)}}$ belongs to $ \tilde{\imath}\big(\log(t^\Delta)\big)= \log(\omega^{\imath(\Delta)})$. Thus, $c(z_n)_{\imath(\Delta)}$ belongs to $\imath(\Delta)$. Using Lemma~\ref{deltaatomicz}, this allows us to conclude that $(z_n)_{n\in \N}$ is $\imath(\Delta)$-atomic and that $\imath(\Delta)_{(z_n)}$ is an initial logarithmic $\k$-subspace of $\No$. 

To finish the proof, we need to show that the map $\jmath:\Delta_{(y_n)}\to \No$ which extends $\imath$ and sends $c(y_n)$ to $c(z_n)$ for each $n$ is a $\log$-embedding. Since $y_n = t^{c(y_n)}$ and $z_n = \omega^{c(z_n)}$ for each $n$, the only part which is not routine is checking that this map is order-preserving. Let $m\leq n\in \N$, let $\delta \in \Delta$, and let $r_m,\ldots,r_n\in \k$ with $r_m \neq 0$. We need to show that $r_mc(y_m) + \cdots+r_nc(y_n) + \delta> 0 \Longleftrightarrow r_mc(z_m) + \cdots+r_nc(z_n) + \imath(\delta) > 0$. We have
\[
r_mc(y_m) + \cdots+r_nc(y_n) + \delta > 0\ \Longleftrightarrow \ y_m^{r_m}\cdots y_n^{r_n}t^\delta \succ 1\ \Longleftrightarrow\ r_m\log y_m + \cdots+r_n\log y_n+ \log t^\delta> 0.
\]
For each $k$, we have $\log y_k =d_k+\epsilon_k y_{k+1}$ where $\epsilon_k = \pm1$. Set $a:= r_md_m+ \cdots+ r_nd_n+\log t^\delta\in \T_\Delta$, so 
\[
r_m\log y_m + \cdots+r_n\log y_n+ \log t^\delta>0\ \Longleftrightarrow\ r_m\epsilon_my_{m+1}+\cdots + r_n \epsilon_n y_{n+1} +a> 0.
\]
Likewise,
\[
r_mc(z_m) + \cdots+r_nc(z_n) + \imath(\delta) > 0\ \Longleftrightarrow\ r_m\epsilon_mz_{m+1}+\cdots + r_n \epsilon_n z_{n+1} +\tilde{\imath}(a)> 0.
\]
Since $y_k \succ \log y_k \succeq y_{k+1}$ and $y_k \not\asymp a$ for each $k$, there are only two possibilities. Either $y_{m+1} \prec a$, in which case
\[
r_m\epsilon_my_{m+1}+\cdots + r_n \epsilon_n y_{n+1} +a> 0\ \Longleftrightarrow \ a>0,
\]
or $y_{m+1} \succ a$, in which case
\[
r_m\epsilon_my_{m+1}+\cdots + r_n \epsilon_n y_{n+1} +a> 0\ \Longleftrightarrow \ r_m\epsilon_my_{m+1}>0 \ \Longleftrightarrow \ r_m\epsilon_m>0.
\]
The same case distinction applies for the $z_i$'s. Since $z_{m+1}$ realizes the same cut over $\tilde{\imath}(\T_\Delta)$ that $y_{m+1}$ realizes over $\T_\Delta$, we see that $y_{m+1} \prec a\Longleftrightarrow z_{m+1} \prec \tilde{\imath}(a)$, so 
\[
r_mc(y_m) + \cdots+r_nc(y_n) + \delta > 0\ \Longleftrightarrow\ a>0\ \Longleftrightarrow\ r_mc(z_m) + \cdots+r_nc(z_n) + \imath(\delta) > 0
\]
 in the first case and
\[
r_mc(y_m) + \cdots+r_nc(y_n) + \delta > 0\ \Longleftrightarrow\ r_m\epsilon_m>0\ \Longleftrightarrow\ r_mc(z_m) + \cdots+r_nc(z_n) + \imath(\delta) > 0
\]
in the second. 
\end{proof}

\begin{proposition}\label{propembedhahn}
Let $\Delta$ be a logarithmic $\k$-subspace of $\Gamma$ whose universe is a set, and let $\imath:\Delta\to \No$ be an initial $\log$-embedding. Then $\imath$ extends to an initial $\log$-embedding $\jmath:\Gamma\to \No$.
\end{proposition}
\begin{proof}
It suffices to extend $\imath$ to an initial $\log$-embedding of some logarithmic $\k$-subspace $\Delta^*$ of $\Gamma$ where $\Delta^*$ is a set which properly contains $\Delta$. We first handle the case that $\Gamma$ is itself a set. If $\Delta \neq \Delta^E$, then we use Proposition~\ref{extendtoE} to extend $\imath$ to an initial $\log$-embedding $\jmath:\Delta^E \to \No$. Note that $\Delta^E$ is a set, since it is contained in $\Gamma$. If $\Delta = \Delta^E$, then take $\gamma \in \Gamma^{>0} \setminus \Delta^{>0}$ and let
\[
y_0\ :=\ t^\gamma,\qquad y_{n+1}\ :=\ \big|\log y_n - (\log y_n)_{\T_\Delta}\big|\ \text{ for each $n$}.
\]
Then $(y_n)_{n\in \N}$ is a $\Delta$-path by Lemma~\ref{expsubspacedcut}. Using Lemma~\ref{existsatomic}, we may replace $(y_n)_{n\in \N}$ with a tail and arrange that $(y_n)_{n\in \N}$ is $\Delta$-atomic. Since $\Delta$ is a set, Proposition~\ref{extendtoy} allows us to extend $\imath$ to an initial $\log$-embedding $\jmath:\Delta_{(y_n)} \to \No$.

Now we handle the case that $\Gamma$ is a proper class. Take $y \in \Gamma \setminus \Delta$ and apply Lemma~\ref{settosubspace} with $\Delta\cup\{y\}$ in place of $A$ to get a logarithmic $\k$-subspace $\Gamma_0 \subseteq \Gamma$ properly containing $\Delta$ such that $\Gamma_0$ is a set, $\log(t^{\Gamma_0})$ is truncation closed, and $t^{\Gamma_0^{>0}} \subseteq \log(t^{\Gamma_0})$. Applying the previously handled case with $\Gamma_0$ in place of $\Gamma$, we extend $\imath$ to an initial $\log$-embedding of some logarithmic $\k$-subspace $\Delta^*\subseteq \Gamma_0$ which properly contains $\Delta$.
\end{proof}

\begin{proof}[Proof of Theorem~\ref{thmembedhahn}]
Applying Proposition~\ref{propembedhahn} with $\{0\}$ in place of $\Delta$, noting that the map $0 \mapsto 0$ is an initial $\log$-embedding, gives an initial $\log$-embedding $\imath:\Gamma\to \No$. Then $\tilde{\imath}:\T\to \No$ is an initial transserial embedding.
\end{proof}

%-----------------------%
\subsection{General transserial embeddings}\label{subsecBM}
%-----------------------%
The following is an important corollary of Theorem~\ref{thmembedhahn}:

\begin{corollary}\label{corembedhahn}
Let $\T=\R((t^\Gamma))_{\On}$ be a transserial Hahn field and suppose that $\T$ contains a truncation closed, cross sectional, exponential subfield $K$. Then $\T$ admits an initial transserial embedding into $\No$.
\end{corollary}
\begin{proof}
We verify that $\T$ satisfies the conditions of Theorem~\ref{thmembedhahn}.
Let $\R_K$ be the coefficient field of $K$, that is, $\R_K= \{r\ \in \R : rt^0 \in K \}$. Since $K$ is cross sectional, we have $t^\Gamma \subseteq K$, so $\log(t^\Gamma) \subseteq K$. Since $K$ is an exponential field, each purely infinite element $a\in K$ belongs to $\log(t^\Gamma)$ by Lemma~\ref{lem-monomialinfinite}. Since $r \log t^\gamma$ is a purely infinite element of $K$ for each $r \in \R_K$ and each $\gamma \in \Gamma$, we see that $\log(t^\Gamma)$ is an $\R_K$-subspace of $K$. Since $t^\gamma$ is a purely infinite element of $K$ whenever $\gamma \in \Gamma^{>0}$, we have $t^{\Gamma^{>0}} \subseteq \log(t^\Gamma)$. If $a$ is a truncation of some element in $\log(t^\Gamma)$, then $a \in K$, since $K$ is truncation closed, and $a$ is purely infinite, so $a\in \log(t^\Gamma)$. Thus $\log(t^\Gamma)$ is truncation closed. It remains to show that $\Gamma$ is a Hahn space over $\R_K$. Let $\gamma,\delta \in \Gamma\setminus\{0\}$ with $\gamma \asymp \delta$. Take $r \in \R$ with $\gamma -r\delta \prec \delta$. We need to show that $r \in \R_K$. By replacing $\delta$ with $-\delta$ and $r$ with $-r$ if need be, we may assume $\delta>0$. Then $p\delta<\gamma$ for all $p\in \R_K^{<r}$ and $q\delta> \gamma$ for all $q\in \R_K^{>r}$. As $\log$ is strictly increasing, we have
\[
\{\log t^{p\delta}:p\in \R_K^{<r}\}\ <\ \log t^\gamma\ <\ \{\log t^{q\delta}:q\in \R_K^{>r}\}.
\]
As $\log t^{p\delta}/\log t^{\delta} = p$ for $p \in \R$, we may divide everything above by $\log t^\delta>0$ to get
\[
\R_K^{<r}\ <\ \frac{\log t^{\gamma }}{\log t^\delta}\ <\ \R_K^{>r}.
\]
Thus, the residue of $\frac{\log t^{\gamma }}{\log t^\delta}$ is $r$, so $r \in \R_K$.
\end{proof}

Using Corollary~\ref{corembedhahn}, we can establish a converse to Fact~\ref{embeddingtotransseries}:

\begin{proposition}\label{transseriestoembedding}
Every transserial Hahn field admits a transserial embedding into $\No$. 
\end{proposition}
\begin{proof}
The proof relies on the following construction of Schmeling~\cite[Subsection 2.3.2]{SC}: given a transserial Hahn field $\T=\R((t^\Gamma))_{\On}$, Schmeling defines an ordered group $\Gamma_{\exp}$ which extends $\Gamma$ and a logarithm on $\T_{\exp} := \R((t^{\Gamma_{\exp}}))_{\On}$ which extends the logarithm on $\T$ such that $\T$ is contained in $\log(\T_{\exp}^{>0})$. The precise definitions of $\Gamma_{\exp}$ and $\T_{\exp}$ are not needed for our purpose. Though Schmeling's transserial Hahn fields are always sets, his construction works just as well for proper classes. 

Let $\T = \R((t^\Gamma))_{\On}$ be a transserial Hahn field. We define an increasing chain of ordered abelian groups $(\Gamma_n)_{n\in \N}$ and a corresponding chain of transserial Hahn fields $(\T_n)_{n\in \N}$, where $\T_n = \R((t^{\Gamma_n}))_{\On}$, as follows:
\[
\Gamma_0\ :=\ \Gamma,\qquad \T_0\ :=\ \T,\qquad \Gamma_{n+1}\ :=\ (\Gamma_n)_{\exp},\qquad \T_{n+1}\ :=\ (\T_n)_{\exp},
\]
where $(\Gamma_n)_{\exp}$ and $(\T_n)_{\exp}$ are defined \`{a} la Schmeling.
We let $\T_\omega :=\bigcup_n \T_n$, so $\T_\omega$ is an exponential field, but not necessarily a transserial Hahn field (since an increasing union of Hahn fields is not a Hahn field in general). We also let $\Gamma_\omega:= \bigcup_n \Gamma_n$, so $\R((t^{\Gamma_\omega}))_{\On}$ is a transserial Hahn field, with the unique logarithm which extends the logarithm on each subfield $\T_n$. Note that $\T_\omega$ is a truncation closed, cross sectional subfield of $\R((t^{\Gamma_\omega}))_{\On}$, so by Corollary~\ref{corembedhahn}, $\R((t^{\Gamma_\omega}))_{\On}$ admits an initial transserial embedding into $\No$. The restriction of this embedding to $\T$ is itself a transserial embedding, as is easily verified.
\end{proof}

\begin{remark}
Let $K$ be a truncation closed, cross sectional logarithmic subfield of a transserial Hahn field $\T$ and let $\imath:\T\to\No$ be a transserial embedding into $\No$. Then $\imath(K)$ is a truncation closed logarithmic subfield of $\No$, so $K$ is a \emph{field of transseries} in the sense of Berarducci and Mantova~\cite[p. 3561]{BM2}. Schmeling refers to any logarithmic subfield of a transserial Hahn field as a \emph{field of transseries}, though most natural examples (increasing unions of transserial Hahn fields, grid-based transseries, etc.) are truncation closed, cross sectional logarithmic subfields of some appropriately chosen transserial Hahn field. Thus, Berarducci and Mantova's definition of a field of transseries essentially corresponds to Schmeling's (as Berarducci and Mantova conjecture in~\cite{BM2}).
\end{remark}

The proof of Proposition~\ref{transseriestoembedding} shows that every transserial Hahn field $\T$ is contained in a larger transserial Hahn field which admits an initial transserial embedding into $\No$. Though the restriction of this embedding to $\T$ is still transserial, it may not be initial. One may ask whether $\T$ itself admits an \emph{initial} transserial embedding into $\No$. The next example shows that this is not the case in general. 

\begin{example}\label{ex-notinitial}
Let $\Gamma$ be the ordered group
\[
\Gamma\ :=\ \bigoplus_{n \in \N}\Z \gamma_n,
\]
where $\gamma_0 \succ \gamma_1 \succ\cdots$. We define a logarithm on $\R((t^\Gamma))_{\On}$ by first setting
\[
\log t^{k_0\gamma_0+\cdots+k_n\gamma_n}\ :=\ k_0t^{\gamma_1}+\cdots+k_n t^{\gamma_{n+1}}
\]
for each element $k_0\gamma_0+\cdots+k_n\gamma_n \in \Gamma$ and by then extending to $\R((t^\Gamma))_{\On}$ using Remark~\ref{rem-extendlog}. Note that $\log t^{\gamma_n} = t^{\gamma_{n+1}}$ and that $\log t^\gamma$ is purely infinite for each $\gamma \in \Gamma$. A straightforward computation gives that $\log t^\gamma\prec t^\gamma$ for each $\gamma \in \Gamma^{>0}$, so $\R((t^\Gamma))_{\On}$ is a logarithmic Hahn field. We claim that $\R((t^\Gamma))_{\On}$ is even a transserial Hahn field. To see this, let $(\delta_n)_{n\in \N}$ be as in the statement of (T4), so $\delta_{n+1} \in \supp \log(t^{\delta_n})$ for each $n$. We see that $\delta_1 = \gamma_{m+1}$ for some $m$. Then the only element in $\supp(\log t^{\delta_1})$ is $\gamma_{m+2}$, so $\delta_2 = \gamma_{m+2}$. More generally, $\delta_n = \gamma_{m+n}$ for each $n>1$, and $\log t^{\delta_n} = t^{\delta_{n+1}}$ for all $n>1$. In particular, (T4) holds.

Now we claim that there is no initial transserial embedding $\R((t^\Gamma))_{\On}\to \No$. It is enough to show that there is no initial ordered group embedding $\Gamma \to \No$. Suppose towards contradiction that $\imath:\Gamma\to \No$ is an initial ordered group embedding. Then $1 \in \imath(\Gamma)$, since $\Gamma$ is nontrivial and $\imath$ is initial. Since $\Gamma$ has no least positive element, we see that $1/2^n \in \imath(\Gamma)$ for each $n \in \N$. Take $\gamma \in \Gamma$ with $\imath(\gamma) = 1$. Then $\gamma/2^n \in \Gamma$ for each $n \in \N$, but there is no element of $\Gamma$ which can be divided by 2 arbitrarily many times.
\end{example}

Given this negative result, one may ask for whether the conditions in Theorem~\ref{thmembedhahn} are necessary. That is, if $\T = \R((t^\Gamma))_{\On}$ is a transserial Hahn field which admits an initial transserial embedding into $\No$, then is $\Gamma$ necessarily a Hahn space over some Archimedean ordered field $\k$? Is $\log(t^\Gamma)$ necessarily truncation closed? Does $\log(t^\Gamma)$ necessarily contain $t^{\Gamma^{>0}}$? As it turns out, only the second condition is necessary.

\begin{remark}
Let $\T = \R((t^\Gamma))_{\On}$ be a transserial Hahn field and let $\imath:\Gamma\to \No$ be an initial $\log$-embedding. Then $\log(t^\Gamma)$ is truncation closed. To see this, let $\gamma \in \Gamma$ and write $\imath(\gamma) = \sum_{\alpha<\beta} r_\alpha \omega^{s_\alpha}\in \No$. Then 
\[
\log \omega^{\imath(\gamma)}\ =\ \sum_{\alpha<\beta} r_\alpha \omega^{h(s_\alpha)},
\]
by Proposition~\ref{Gonshor 1986.2}. If $a \in \T$ is a truncation of $\log t^\gamma$, then $\tilde{\imath}(a) = \sum_{\alpha<\beta_0} r_\alpha \omega^{h(s_\alpha)}$ for some $\beta_0\leq \beta$. Since $\imath(\Gamma)$ is truncation closed, there is a $\delta \in \Gamma$ with $\imath(\delta) = \sum_{\alpha<\beta_0} r_\alpha \omega^{s_\alpha}$. Then 
\[
\log \omega^{\imath(\delta)}\ =\ \sum_{\alpha<\beta_0} r_\alpha \omega^{h(s_\alpha)}\ =\ \tilde{\imath}(a),
\]
again by Proposition~\ref{Gonshor 1986.2}. Since $\imath$ is a $\log$-embedding, we have $\log t^\delta = a$.
\end{remark}

\begin{example}
Let $\Gamma$ be the ordered group
\[
\Gamma\ :=\ \R \gamma_0 \oplus \bigoplus_{n \in \N^{>0}}\Q \gamma_n,
\]
where $\gamma_0 \succ \gamma_1 \succ\cdots$. As in Example~\ref{ex-notinitial}, we may define a logarithm on $\R((t^\Gamma))$ which makes it into a transserial Hahn field by setting 
\[
\log t^{r\gamma_0+q_1\gamma_1 \cdots+q_n\gamma_n}\ :=\ rt^{\gamma_1}+q_1t^{\gamma_2}+\cdots+q_n t^{\gamma_{n+1}}
\]
for $r \in \R$ and $q_1,\ldots,q_n \in \Q$. Note that $\Gamma$ is not a Hahn space over $\Q$, since $\sqrt{2} \gamma_{0} \asymp \gamma_{0}$, but $\sqrt{2} \gamma_{0} -q \gamma_{0}\not\prec \gamma_{0}$ for any $q \in \Q$. Note also that $t^{\gamma_0} \not\in \log(t^\Gamma)$. Nevertheless, there is an initial $\log$-embedding $\imath:\Gamma\to \No$, given by 
\[
\imath(r\gamma_0 + q_1\gamma_1+ \cdots+q_n \gamma_n) \ :=\ r+q_1\omega^{-1}+ \cdots+q_n \omega^{-n}
\]
for $r \in \R$ and $q_1,\ldots,q_n \in \Q$. One may verify that $\imath(\Gamma)$ is initial using the characterization of initial subgroups of $\No$ in~\cite[Theorem 5.1]{EK}. To see that $\imath$ is a $\log$-embedding, one needs to use the identity
\[
\log \omega^{\omega^{-n}}\ =\ \omega^{\omega^{-n-1}}.
\]
This identity is well-known, and it can be proven by induction on $n$ using Gonshor's map $h$ and Proposition~\ref{Gonshor 1986.2}. We leave the full details to the reader.
\end{example}

\begin{open question*}
What is a set of conditions that are individually necessary and collectively sufficient for a transserial Hahn field to admit an initial transserial embedding into $\No$?
\end{open question*}

%------------------------------------------------%
\section{Initial embeddings of ordered exponential fields}\label{sec:IE}
%------------------------------------------------%
We now turn to our promised characterization of which ordered exponential fields are isomorphic to initial exponential subfields of $\No$.

\begin{theorem}\label{T1}
Let $A$ be an ordered exponential field. Then $A$ is isomorphic to an initial exponential subfield of $\No$ if and only if $A$ is isomorphic to a truncation closed, cross sectional exponential subfield $K$ of a transserial Hahn field $\T$.
\end{theorem}
\begin{proof}
Let $K$ be an initial exponential subfield of $\No$ and let $\Gamma$ be the value group of $K$. Then $K$ is a truncation closed, cross sectional subfield of the Hahn field $\R((\omega^\Gamma))_{\On}$ by~\ref{IF}. For each $\gamma \in \Gamma$, we have $\omega^\gamma \in K^{>0}$, so $\log \omega^\gamma \in K \subseteq \R((\omega^\Gamma))_{\On}$. Since $\No$ is a transserial Hahn field by Proposition~\ref{surrealsaretransseries}, we see that $\R((\omega^\Gamma))_{\On}$ is a transserial Hahn field as well.

For the converse, let $K$ be a truncation closed, cross sectional exponential subfield of a transserial Hahn field $\T$. By Corollary~\ref{corembedhahn}, $\T$ admits an initial transserial embedding into $\No$. The image of $K$ under this embedding is initial by Proposition~\ref{Ehrlich initial 2}.
\end{proof}

\begin{remark}\label{rem-thanks}
An earlier version of this work did not characterize the initial exponential subfields of $\No$ using transseries. We instead introduced a condition referred to as \emph{molecularity}, which essentially required that the field $K$ in the statement of Theorem~\ref{T1} be ``built out of'' exponential subspaces and atomic elements. We thank the anonymous referee for pointing out the connection to Schmeling's transseries, which greatly simplified both the statement and the proof of Theorem~\ref{T1}.
\end{remark}

Ressayre (\cite{R}; also see~\cite{D}) showed if $A$ is a \emph{real closed exponential field} with residue class field $\R_A$ and value group $\Gamma$, then $A$ is isomorphic to a truncation closed, cross sectional subfield $K$ of a Hahn field $\R((t^\Gamma))_{\On}$, where the logarithm of $t^\gamma \in K$ is purely infinite for each $\gamma \in \Gamma$. It is an open question whether every model of the theory $T (\R,e^{x})$ of real numbers with exponentiation is isomorphic to a truncation closed, cross sectional exponential subfield of a transserial Hahn field (or even of a \emph{logarithmic} Hahn field). As such, contrary to what is stated in~\cite{EH9}, the following remains an

\begin{open question*}
Is every model of $T (\R,e^{x})$ isomorphic to an initial exponential subfield of $\No$?	
\end{open question*}

However, while this question remains open, in the following section it is shown there are distinguished classes of models of $T (\R,e^{x})$ having additional structure that are isomorphic to initial exponential subfields of $\No$. We remark that the main difficulty seems to be showing that every $A \models T (\R,e^{x})$ admits an embedding into a Hahn field $\R((t^\Gamma))_{\On}$ which sends the exponential of each infinitesimal $\epsilon \in A$ to the series $\sum_{n=0}^\infty \epsilon^n/n!$\footnote{The second author would like to thank Lothar Sebastian Krapp for helpful discussions about this issue.}. If Ressayre's embedding theorem can be amended to also satisfy this condition, then the methods in the next section can likely be used to provide a positive answer to the above question.

%------------------------------------------------%
\section{Exponential fields which define convergent Weierstrass systems}\label{sec:WS}
%------------------------------------------------%
Let $I = [-1,1]$. Given $n$, an open neighborhood $U \supseteq I^n$, and a real analytic function $f:U\to \R$, we define a corresponding \textbf{restricted analytic function} $\bar{f}:U \to \R$ by
\[
\bar{f}(x)\ :=\ \left\{\begin{array}{cc}f(x)&\text{if }x\in I^n\\0&\text{otherwise.}\end{array}\right.
\]
Let $\cC^\omega_r$ denote the family of all restricted analytic functions (of any arity) and let $\cF\subseteq \cC^\omega_r$. 

\begin{definition}\
\begin{enumerate}[(i)]
\item Let $\cL\uf$ be the language $(+,\cdot,-,0,1,<,\bar{f}\in \cF)$ and let $\R\uf$ be the natural expansion of $\R$ to an $\cL\uf$-structure. Let $T\uf$ be the complete $\cL\uf$-theory of $\R\uf$.
\item
Let $\cL\ufexp$ be the language $\cL\uf\cup\{\exp\}$ and let $\R\ufexp$ be the expansion of $\R\uf$ by the total exponential function. Let $T\ufexp$ be the complete $\cL\ufexp$-theory of $\R\ufexp$.
\end{enumerate}
\end{definition}

\begin{lemma}\label{lem-Hahnexpansion}
Let $\Gamma$ be a divisible ordered abelian group (whose universe is a set or a proper class). Then the Hahn field $\R((t^\Gamma))_{\On}$ admits a natural expansion to a model of $T\uf$. In this expansion, the interpretation of any restricted analytic function agrees with its Taylor series expansion.
\end{lemma}
\begin{proof}
It suffices to show this in the case that $\cF = \cC^\omega_r$; the result follows for arbitrary families $\cF$ by taking a reduct. If $\Gamma$ is a set, this is a result of van den Dries, Macintyre and Marker~\cite{DMM}. If $\Gamma$ is a proper class, then we have $\Gamma= \bigcup_{\alpha< \On}\Gamma_\alpha$ where $(\Gamma_\alpha)_{\alpha<\On}$ is an increasing family of divisible ordered abelian subgroups of $\Gamma$ whose universes are sets. Then $\R((t^\Gamma))_{\On}= \bigcup_{\alpha<\On}\R((t^{\Gamma_\alpha}))$. Moreover, since $T_{\cC^\omega_r}$ is model complete~\cite{vdD86,Ga}, the union of a continuous chain of set-models of $T_{\cC^\omega_r}$ is a model of $T_{\cC^\omega_r}$ as well~\cite[page 41, (i)]{EH2}. 
\end{proof}

In particular, $\No$ admits a natural expansion to a model of $T\uf$, where the restricted analytic functions are interpreted via Taylor series expansion. This is also, of course, a consequence of Proposition~\ref{DE1}.

Let $\cF\df\subseteq \cC^\omega_r$ be the collection of all restricted analytic functions which are $0$-definable in the structure $\R\uf$. Let $K \models T\uf$ and let $A\subseteq K$. We say that $A$ is \textbf{$\cF$-closed} if $A$ is a real closed subfield of $K$ which is closed under all functions $\bar{f} \in \cF\df$. We let the \textbf{$\cF$-closure} of $A$ be the smallest $\cF$-closed subfield of $K$ containing $A$. Given an $\cF$-closed subfield $A \subseteq K$ and an element $y \in K$, we let $A\langle y\rangle$ be the $\cF$-closure of $A\cup\{y\}$. Note that $\cF\df$ is closed under taking partial derivatives, so the following is a consequence of~\cite[Lemma 3.5]{MR} and~\cite[Lemma 3.3]{DMM1}.

\begin{lemma}\label{tclosed}
Let $\Gamma$ be a divisible ordered abelian group whose universe is a set and let $A$ be a truncation closed subset of $\R((t^\Gamma))$. Then the $\cF$-closure of $A$ is truncation closed.
\end{lemma}

We say that $T\uf$ \textbf{defines a convergent Weierstrass system} if the family 
\[
\big\{f(a+x):\bar{f}\in \cF\df\text{ and } a\in I^n\big\}
\]
forms a convergent Weierstrass system, as defined in~\cite{vdD}. Note that if $\cF= \cC^\omega_r$, then $T\uf$ defines a convergent Weierstrass system. Indeed, the definitions and methods in~\cite{vdD} are based on techniques developed in~\cite{DD} in the case $\cF = \cC^\omega_r$. The notion of a general Weierstrass system was introduced in~\cite{DL}. Our main result in this section is the following:

\begin{theorem}\label{initialexponentialWeier}
Suppose that $T\uf$ defines a convergent Weierstrass system and that the restriction $\overline{\exp}$ is in $\cF\df$. Then any model $K\models T\ufexp$ admits an initial $\cL\ufexp$-elementary embedding into $\No$.
\end{theorem}

In preparation for the proof of Theorem~\ref{initialexponentialWeier}, we devote the next subsection to proving a simpler result:

\begin{theorem}\label{initialWeier}
Suppose that $T\uf$ defines a convergent Weierstrass system. Then any model $K\models T\uf$ admits an initial $\cL\uf$-elementary embedding into $\No$.
\end{theorem}

In~\cite{F2}, Fornasiero proves that every model of $T_{\cC^\omega_r,\exp}$ admits an initial $\cL_{\cC^\omega_r,\exp}$-elementary embedding into $\No$. Fornasiero's proof involves constructing direct embeddings into $\No$, whereas we invoke Corollary~\ref{corembedhahn} above after first finding an embedding into a suitable transserial Hahn field. Both proofs, of course, make essential use of o-minimality and quantifier elimination results. Fornasiero also proves that every model of $T_{\cC^\omega_r}$ admits an initial $\cL_{\cC^\omega_r}$-elementary embedding into $\No$, again using a different method from our own.

%-----------------------%
\subsection{Initial embeddings of models of $T\uf$}
%-----------------------%
In this subsection, we assume that $T\uf$ defines a convergent Weierstrass system, and we fix a model $K\models T\uf$ whose universe is a set or a proper class. As any convergent Weierstrass system contains all constant functions $x \mapsto r$, we may naturally identify $\R$ with a subfield of $K$. The following result,~\cite[Theorem 1.7]{vdD}, is key:

\begin{proposition}\label{closureiselem1}
$\R\uf$ admits quantifier elimination in the language $\cL_{\cF\df} \cup\{(-)\inv\}$, where $(-)\inv$ is interpreted as multiplicative inversion away from zero. 
\end{proposition}

\begin{corollary}\label{closureiselem}
Let $A$ be an $\cF$-closed subclass of $K$ and let $y \in K\setminus A$. Then $A$ is an $\cL\uf$-elementary substructure of $K$ and the complete $\cL\uf$-type of $y$ over $A$ is determined by the cut of $y$ in $A$.
\end{corollary}
\begin{proof}
By the same argument as in~\cite{DMM}, a subclass $A\subseteq K$ is an elementary substructure of $K$ if and only if $A$ is $\cF$-closed, so the $\cF$-closure of $A$ is the same as the $\cL\uf$-definable closure of $A$. Since $T\uf$ is o-minimal, the complete $\cL\uf$-type of $y$ over $A$ is determined by the cut of $y$ in $A$.
\end{proof}

The following fact is an immediate consequence of~\cite[Corollary 3.7]{DMM}:

\begin{fact}
\label{wilkieineq}
Let $A$ be an $\cF$-closed subclass of $K$, let $\mfM$ be a cross section for $A$, and let $\fm \in K$ be such that $\fm \not\asymp a$ for all $a \in A$. Then $\mfM \times \fm^\Q$ is a cross section for $A\langle\fm\rangle$. 
\end{fact}

Let $\R((t^\Gamma))_{\On}$ be a Hahn field, viewed as a model of $T\uf$ as described in Lemma~\ref{lem-Hahnexpansion}. Let $A$ be an $\cF$-closed subclass of $K$, let $\imath:A\to \R((t^\Gamma))_{\On}$ be an $\cL\uf$-elementary embedding, and suppose that $\imath(A)$ is a truncation closed, cross sectional $\cL\uf$-substructure of $\R((t^\Gamma))_{\On}$. Note that $\R$ is contained in $A$ and that $\imath$ is $\R$-linear. Let $y \in K \setminus A$. A \textbf{partial development of $y$ over $A$} is an element $\sum_{\alpha<\beta}r_{\alpha}t^{\gamma_\alpha} \in \R((t^\Gamma))_{\On}$ such that $\sum_{\alpha\leq \sigma}r_{\alpha}t^{\gamma_\alpha} \in A$ and
\[
y -\imath\inv\Big(\sum_{\alpha\leq \sigma}r_{\alpha}t^{\gamma_\alpha}\Big)\ \prec\ \imath\inv(t^{\gamma_\sigma}).
\]
for each $\sigma < \beta$. One easily verifies that any two partial developments of $y$ over $A$ are either equal or one is a truncation of the other, and that any truncation of a partial development is itself a partial development. Thus, there is a maximum partial development of $y$ over $A$, called the \textbf{development of $y$ over $A$} and denoted $D^\imath_A(y)$. These developments behave much like the truncations considered in \S\ref{sec:P2}. Indeed, if $K$ is itself a subfield of $\R((t^\Gamma))_{\On}$ and $\imath:A\to \R((t^\Gamma))_{\On}$ is the identity map, then the development of $y$ over $A$ is easily seen to coincide with the $A$-truncation of $y$.

\begin{lemma}\label{basicdevelopments}\
\begin{enumerate}
\item If $D^\imath_A(y)$ belongs to $\imath(A)$, then 
\[
y - \imath\inv\big(D^\imath_A(y)\big) \ \not\asymp \ \imath\inv(t^\gamma)
\]
for any $\gamma \in \Gamma$.
\item If $D^\imath_A(y)$ does not belong to $\imath(A)$, then $D^\imath_A(y)$ realizes the same cut over $\imath(A)$ that $y$ realizes over $A$.
\end{enumerate}
\end{lemma}
\begin{proof}
For (1), suppose towards contradiction that $D^\imath_A(y)$ belongs to $\imath(A)$ and that we can find $\gamma \in \Gamma$ with
\[
y - \imath\inv\big(D^\imath_A(y)\big) \ \asymp \ \imath\inv(t^\gamma).
\]
Take $r \in \R\setminus \{0\}$ such that
\[
y - \imath\inv\big(D^\imath_A(y)\big) -r \imath\inv(t^\gamma)\ \prec \ \imath\inv(t^\gamma).
\]
Then $D^\imath_A(y)+r t^\gamma\in \imath(A)$ is a partial development of $y$ over $A$ which strictly extends $D^\imath_A(y)$, a contradiction.

For (2), suppose towards contradiction that there is an $a \in A$ with $a<y$ and $\imath(a) > D^\imath_A(y)$ (the same argument will work in the case that $b>y$ and $\imath(b)< D^\imath_A(y)$ for some $b \in A$). Let $z =\sum_{\alpha< \sigma}r_{\alpha}t^{\gamma_\alpha}$ be the greatest common truncation of $\imath(a)$ and $D^\imath_A(y)$. Then $z \in \imath(A)$, since $z$ is a truncation of $\imath(a)$ and $\imath(A)$ is truncation closed. Since $D^\imath_A(y) \not\in \imath(A)$, we see that $z$ is a proper truncation of $D^\imath_A(y)$, so take $r \in \R\setminus \{0\}$ and $\gamma \in \Gamma$ such that $z+rt^\gamma$ is a truncation of $D^\imath_A(y)$ properly extending $z$. Then $z+rt^\gamma$ is not a truncation of $\imath(a)$ by the maximality of $z$, so 
\[
\imath(a)- (z +rt^\gamma) \ \succeq\ t^\gamma\ \succ \ D^\imath_A(y)- (z +rt^\gamma).
\]
Since $\imath(a)> D^\imath_A(y)$, this gives $\imath(a) > z +rt^\gamma$. Since $a<y$, we have
\[
y - \imath\inv(z+rt^\gamma)\ >\ a- \imath\inv(z+rt^\gamma) \ \succeq\ \imath\inv(t^\gamma).
\]
This contradicts that $z+rt^\gamma$, being a truncation of $D^\imath_A(y)$, is a partial development of $y$. 
\end{proof}

\begin{proposition}
\label{anembedding}
$K$ admits a truncation closed, cross sectional $\cL\uf$-elementary embedding into a Hahn field $\R((t^\Gamma))_{\On}$.
\end{proposition}
\begin{proof}
Fix a cross section $\mfM\subseteq K^{>0}$, let $\Gamma$ be an isomorphic copy of the ordered group $\mfM$, written additively, and let $\eta:\Gamma\to \mfM$ be an ordered group isomorphism witnessing this. Let us assume that we have an embedding $\imath:A \to \R((t^\Gamma))_{\On}$ where
\begin{enumerate}[(i)]
\item $A$ is an $\cF$-closed subclass of $K$, $\imath(A)$ is truncation closed, and $\imath$ is $\cL\uf$-elementary;
\item $\imath(\fm) = t^{\eta(\fm)}$ for each $\fm \in \mfM \cap A$.
\end{enumerate}
Such an embedding exists: since $T\uf$ defines a convergent Weierstrass system, we have $\R\subseteq K$ and so the identification of $\R$ with $\{rt^0:r\in \R\}$ is such an embedding (where $\Gamma = \{0\}$). If $A = K$ then we are done, so we assume that $A \neq K$ and show any such embedding $\imath$ can be properly extended to an embedding $\jmath:B\to \R((t^\Gamma))_{\On}$ with the same properties. We consider two possibilities:

\textbf{Case 1:} Suppose $\mfM \cap A \neq \mfM$. Take $\fm \in \mfM \setminus A$ and let $\gamma := \eta(\fm)$, so $t^\gamma$ realizes the same cut over $\imath(\mfM\cap A) = t^{\eta(\mfM\cap A)}$ that $\fm$ realizes over $\mfM\cap A$. One easily verifies that $t^\gamma$ actually realizes the same cut over $\imath(A)$ that $\fm$ realizes over $A$, so by o-minimality, $\imath$ extends to an $\cL\uf$-elementary embedding
\[
\jmath:A\langle \fm\rangle\to \R((t^{\Gamma}))_{\On}
\]
which sends $\fm$ to $t^\gamma = t^{\eta(\fm)}$. Since the class $\imath(A)\cup \{t^\gamma\}$ is truncation closed, Lemma~\ref{tclosed} tells us that $\jmath\big(A\langle y \rangle\big) = \imath(A)\langle t^\gamma \rangle$ is truncation closed as well. Using Fact~\ref{wilkieineq}, we see that $\mfM \cap A\langle \fm\rangle= (\mfM \cap A) \times \fm^\Q$, and that $\jmath(\fn\fm^q) = t^{\eta(\fn)}t^{q\gamma}= t^{\eta(\fn\fm^q)}$ for each $\fn \in \mfM \cap A$ and each $q\in \Q$. Thus, $\jmath$ also satisfies (i) and (ii). 

\textbf{Case 2:} Suppose $\mfM \cap A = \mfM$, so $\imath(A)$ is cross sectional. Let $y \in K\setminus A$. We claim that $D^\imath_A(y) \not\in \imath(A)$. If $D^\imath_A(y)$ were in $\imath(A)$, then we could take $\fm \in \mfM$ with $y- \imath\inv\big(D^\imath_A(y)\big) \asymp \fm$. Since each $\fm \in \mfM$ is equal to $\imath\inv(t^{\eta(\fm)})$, this would contradict the first part of Lemma~\ref{basicdevelopments}. Thus, $D^\imath_A(y) \not\in \imath(A)$ as claimed, and so $D^\imath_A(y)$ realizes the same cut over $\imath(A)$ that $y$ realizes over $A$ by the second part of Lemma~\ref{basicdevelopments}. By o-minimality, $\imath$ extends to an $\cL\uf$-elementary embedding
\[
\jmath:A\langle y \rangle\to \R((t^\Gamma))_{\On}
\]
which sends $y$ to $D^\imath_A(y)$. As all truncations of $D^\imath_A(y)$ are in $\imath(A)$, the class $\imath(A)\cup \big\{D^\imath_A(y)\big\}$ is truncation closed. Thus, $\jmath\big(A\langle y \rangle\big) = \imath(A)\big\langle D^\imath_A(y)\big\rangle$ is truncation closed as well by Lemma~\ref{tclosed}, so $\jmath$ satisfies condition (i). Condition (ii) holds for $\jmath$ since it holds for $\imath$.
\end{proof}

\begin{proof}[Proof of Theorem~\ref{initialWeier}]
By Proposition~\ref{anembedding}, we may identify $K$ with a truncation closed, cross sectional $\cL\uf$-elementary substructure of a Hahn field $\R((t^\Gamma))_{\On}$. As $\Gamma$ is divisible, we may fix an initial ordered group embedding $\imath:\Gamma\to \No$ by Proposition~\ref{Prop. 2}. Then the induced ordered field embedding
\[
\sum_{\alpha<\beta} r_\alpha t^{\gamma_\alpha} \mapsto \sum_{\alpha<\beta} r_\alpha \omega^{\imath(\gamma_\alpha)}:K\to \No
\]
is initial by Proposition~\ref{Ehrlich initial 2}. Since each restricted analytic function in $\cF\df$ agrees with its Taylor series expansion in both $\R((t^\Gamma))_{\On}$ and in $\No$, the embedding above is an $\cL_{\cF\df}$-embedding. By Proposition~\ref{closureiselem1}, the image of this embedding is $\cL\uf$-elementary.
\end{proof}

%-----------------------%
\subsection{Convergent Weierstrass systems with an entire exponential function}
%-----------------------%
In this subsection, we assume that $T\uf$ defines a convergent Weierstrass system and that $\overline{\exp}$ is in $\cF\df$. We fix a model $K \models T\ufexp$ whose universe is a set or a proper class. Note that $K$ is an ordered logarithmic field, as defined earlier. We note here some basic consequences of the theory $T\ufexp$:
\begin{enumerate}
\item If $a \in K^{>0}$ and $a \asymp 1$, then $\log a \preceq 1$. To see this, take $r \in \R$ with $r>1$ and $r\inv <a<r$. Then $-\ln r<\log a<\ln r$, so $\log a \preceq \ln r \asymp 1$. 
\item If $a \in K$ and $a \preceq 1$, then $\exp a \asymp 1$. This can be shown using an argument similar to (1).
\item If $a \in K$ is positive and infinite, then $\log a$ is also positive and infinite and $\log a \prec a$. The proof of this is identical to the proof of Lemma~\ref{lem-posinflogs}, using that $K$ contains the exponential field $\R$ as an elementary substructure. 
\end{enumerate}
The method used in~\cite{DMM} gives the following:

\begin{proposition}\label{qeexp}
$\R\ufexp$ admits quantifier elimination in the language $\cL_{\cF\df,\exp}\cup \{\log\}$.
\end{proposition}

\begin{definition}
A \textbf{development triple} $(A,\Delta,\imath)$ consists of an $\cF$-closed subclass $A\subseteq K$, an $\R$-vector subspace $\Delta\subseteq A$ with $\exp(\Delta) \subseteq A$, and an $\cL\uf$-elementary embedding $\imath:A\to\R((t^\Delta))_{\On}$ such that
\begin{enumerate}[(i)]
\item $\imath(A)$ is truncation closed,
\item $\imath(\exp \delta) = t^\delta$ for each $\delta \in \Delta$, and
\item $\imath(\delta)$ is purely infinite for each $\delta \in \Delta$.
\end{enumerate}
We say that a development triple $(A,\Delta,\imath)$ is an \textbf{$\exp$-development triple} if for each $a \in A$, there is a $\delta \in \Delta$ with $a-\delta \preceq 1$.
\end{definition}

The definition of a development triple is inspired by~\cite{D}. Our development triples are slightly different from the ones in~\cite{D}, as we insist that our substructures $A$ be $\cF$-closed. Let $(A,\Delta,\imath)$ be a development triple. Then condition (ii) tells us that $\{\exp \delta:\delta \in \Delta\}$ is a cross section for $A$ and that $\imath(A)$ is cross sectional. Since $\delta\prec\exp \delta$ for $\delta \in \Delta^{>0}$, it follows from (ii) that $\imath(\delta)\prec t^\delta$. Thus, by Remark~\ref{rem-extendlog}, there is a unique logarithm on $\R((t^\Delta))_{\On}$ which makes $\R((t^\Delta))_{\On}$ a logarithmic Hahn field and which satisfies $\log t^\delta = \imath(\delta)$. We call this logarithm the \textbf{induced logarithm on $\R((t^\Delta))_{\On}$}. We say that $(A,\Delta,\imath)$ is \textbf{transserial} if $\R((t^\Delta))_{\On}$, equipped with the induced logarithm, is a transserial Hahn field. 

\begin{lemma}\label{expdevelopmentclosed}
Let $(A,\Delta,\imath)$ be a development triple. Then $A$ is closed under $\log$ and $\imath:A \to \R((t^\Delta))_{\On}$ is a logarithmic field embedding with respect to the induced logarithm on $\R((t^\Delta))_{\On}$. Moreover, $(A,\Delta,\imath)$ is an $\exp$-development triple if and only if $A$ is closed under $\exp$.
\end{lemma}
\begin{proof}
First, if $a \asymp 1$, then take $r\in \R^{>0}$ and $\epsilon \prec 1$ with $a = r(1+\epsilon)$. We have $ \log(1+\epsilon) \in A$ and $\imath\big(\log(1+\epsilon)\big) = \log \imath(1+\epsilon)$, since $A$ is $\cF$-closed and $\imath$ is an $\cL\uf$-elementary embedding. Since $\ln r\in \R$ and $\imath$ is the identity on $\R$, it follows that $ \log a \in A$ and $\imath(\log a) = \log \imath(a)$.

Now let $a$ be arbitrary and take $\delta \in \Delta$ with $\imath(a) \asymp t^\delta$. Since $t^\delta = \imath(\exp \delta)$, we have $a\asymp \exp \delta$, so take $u \in A^{>0}$ with $u \asymp 1$ and $a = u\exp \delta$. Then $\log u \in A$ by the previous case, so $\log a = \log u + \delta$ is in $A$ as well. Since $\imath(a) = \imath(u)t^\delta$, we have
\[
\log \imath(a)\ =\ \log \imath(u)+ \log t^\delta \ =\ \imath(\log u) + \imath(\delta) \ =\ \imath(\log u +\delta)\ =\ \imath(\log a).
\]

Finally, suppose that $(A,\Delta,\imath)$ is an $\exp$-development triple, let $a \in A$, and take $\delta \in \Delta$, $r \in \R$, and $\epsilon \prec 1$ with $a = \delta+r+ \epsilon$. Then $\exp a = \exp(\delta)\exp(r)\exp(\epsilon)$, where $\exp\delta \in A$ by assumption, $\exp r \in \R\subseteq A$, and $\exp \epsilon \in A$ since $A$ is $\cF$-closed. Conversely, suppose that $A$ is closed under $\exp$, take $a \in A$, and let $\delta \in \Delta$ with $\imath(\exp a) \asymp t^\delta$. Then $\exp a \asymp \exp \delta$ by (ii), so take $u \in A^{>0}$ with $u \asymp 1$ and $\exp a = u\exp \delta$. Taking logarithms gives $a - \delta = \log u \preceq 1$. 
\end{proof} 

\begin{lemma}
\label{anexpembedding1}
Let $(A,\Delta,\imath)$ be a transserial development triple and let $y \in K\setminus A$ with $D^\imath_A(y) \not\in \imath(A)$. Then $(A,\Delta,\imath)$ can be extended to a transserial development triple $(B,\Delta,\jmath)$ with $y \in B$. 
\end{lemma}
\begin{proof}
Lemma~\ref{basicdevelopments} tells us that $D^\imath_A(y)$ realizes the same cut over $\imath(A)$ that $y$ realizes over $A$, so by o-minimality, we may extend $\imath$ to an $\cL\uf$-elementary embedding
\[
\jmath: A\langle y\rangle\to \R((t^\Delta))_{\On}
\]
which sends $y$ to $D^\imath_A(y)$. We set $B:= A\langle y \rangle$ and we claim that $(B,\Delta,\jmath)$ is a development triple. Since $\Delta$ remains unchanged, conditions (ii) and (iii), as well as the condition that $\R((t^\Delta))_{\On}$ is a transserial Hahn field, hold immediately. To see that $\jmath(B) = \imath(A)\big\langle D^\imath_A(y)\big\rangle$ is truncation closed, note that the class $\imath(A) \cup \big\{D^\imath_A(y)\big\}$ is truncation closed since all truncations of $D^\imath_A(y)$ are in $\imath(A)$ and use Lemma~\ref{tclosed}.
\end{proof}

\begin{lemma}
\label{anexpembedding2}
Let $(A,\Delta,\imath)$ be a transserial development triple and let $a \in A$ be such that $a- \delta \succ 1$ for all $\delta \in \Delta$. Then $(A,\Delta,\imath)$ can be extended to a transserial development triple $(B,\Gamma,\jmath)$ with $a-\gamma \preceq 1$ for some $\gamma \in \Gamma$. 
\end{lemma}
\begin{proof}
By assumption, the purely infinite part $\imath(a)_{>0}$ of $\imath(a)$ is not in $\imath(\Delta)$, so let $\gamma := \imath\inv\big(\imath(a)_{>0}\big)$. Then $a - \gamma \preceq 1$. We let $\Gamma$ be the $\R$-subspace $\Delta+ \R\gamma \subseteq A$ and we let $B:=A\langle \exp \gamma \rangle$. We claim that $\exp \gamma \not\asymp \exp \delta$ for any $\delta \in \Delta$. Suppose otherwise, and let $\delta \in \Delta$ and $u \in K$ with $u \asymp 1$ and $\exp \gamma = u\exp \delta$. Then $\gamma = \log u + \delta$, so $\gamma - \delta = \log u \preceq 1$. Since $a - \gamma \preceq 1$, it follows that $a - \delta \preceq 1$, a contradiction. Since $\{\exp \delta: \delta \in \Delta\}$ is a cross section for $A$, this claim tells us that $\exp \gamma \not \asymp y$ for any $y \in A$.

We now turn to defining an embedding $\jmath: B\to \R((t^\Gamma))_{\On}$. We claim that $t^\gamma$ realizes the same cut over $\imath(A)$ that $\exp \gamma$ realizes over $A$. To see this, let $y \in A^{>0}$ and take $\delta \in \Delta$ with $y \asymp \exp \delta$. Then $\imath(y) \asymp \imath(\exp \delta) = t^\delta$, so 
\[
y< \exp \gamma\ \Longleftrightarrow \ \exp \delta \prec \exp \gamma\ \Longleftrightarrow\ \delta < \gamma\ \Longleftrightarrow\ t^\delta \prec t^\gamma\ \Longleftrightarrow\ \imath(y) < t^\gamma.
\]
By o-minimality, $\imath$ extends to an $\cL\uf$-elementary embedding
\[
\jmath:B \to \R((t^\Gamma))_{\On}
\]
which sends $\exp \gamma$ to $t^{\gamma}$.

We need to show that $\big(B,\Gamma,\jmath\big)$ is a transserial development triple. Since $\imath(A) \cup \{t^\gamma\}$ is a truncation closed class, Lemma~\ref{tclosed} gives that $\jmath(B) = \imath(A)\langle t^\gamma \rangle$ is truncation closed. Given $\delta \in \Delta$ and $r \in \R$, our definition of $\jmath$ gives
\[
\jmath\big(\exp (\delta+r\gamma) \big) \ =\ \imath(\exp \delta)\jmath(\exp \gamma)^r\ =\ t^\delta (t^\gamma)^r\ =\ t^{\delta+r\gamma},
\]
so (ii) is satisfied. Since $\imath(\gamma)$ is purely infinite, (iii) is satisfied. Thus, $\R((t^\Gamma))_{\On}$ with the induced logarithm is a logarithmic Hahn field. Note that $\Gamma = \Delta^E$, where $\Delta^E\subseteq \Gamma$ is constructed as in Lemma~\ref{defofDeltaE} with $\R$ in place of $\k$. Since $\R((t^\Delta))_{\On}$ is a transserial Hahn field, the second part of Lemma~\ref{defofDeltaE} gives that $\R((t^\Gamma))_{\On}$ is also a transserial Hahn field.
\end{proof}

\begin{proposition}\label{molecularext}
Let $(A,\Delta,\imath)$ be a transserial $\exp$-development triple with $D^\imath_A(x) \in \imath(A)$ for all $x\in K$ and let $y \in K\setminus A$. Then $(A,\Delta,\imath)$ can be extended to a transserial development triple $(B,\Gamma,\jmath)$ with $y \in B$. 
\end{proposition}
\begin{proof}
By assumption, $D^\imath_A(y)$ is in $\imath(A)$, so it suffices to find a transserial development triple $(B,\Gamma,\imath)$ containing $\big| y- \imath\inv\big(D^\imath_A(y)\big)\big|$. Thus, by replacing $y$ with $\big| y- \imath\inv\big(D^\imath_A(y)\big)\big|$, we may assume that $y$ is positive and that $y \not\asymp a$ for any $a \in A$ (this uses Lemma~\ref{basicdevelopments}). By replacing $y$ with $y\inv$ if need be, we may also arrange that $y$ is infinite. We construct a sequence $(y_n)_{n\in \N}$ of infinite elements of $K$ by setting
\[
y_0\ :=\ y,\qquad y_{n+1}\ :=\ \big| \log(y_n) -\imath\inv\big(D^\imath_A(\log y_n)\big)\big|.
\]
Since $A$ is closed under exponentiation by Lemma~\ref{expdevelopmentclosed}, we see that each $y_n\in K \setminus A$. Lemma~\ref{basicdevelopments} tells us that each $y_n \not\asymp a$ for any $a \in A$. For each $n$, let $d_n := \imath\inv\big(D^\imath_A(\log y_n)\big)\in A$ and let $\epsilon_n = \pm 1$ be the sign of $\log y_n - d_n$, so $\log y_n = \epsilon_n y_{n+1}+ d_n$. Exponentiating gives 
\[
y_n \ =\ \exp(\epsilon_n y_{n+1}) \exp(d_n).
\]
If $y_{n+1}$ were not infinite, then we would have $\exp(\epsilon_n y_{n+1}) \asymp 1$ and $y_n \asymp \exp d_n\in A$, a contradiction, so each $y_n$ is infinite.

\textbf{Claim 1:} Let $n$ be fixed. We claim that $y_n \not\asymp a$ for any $a \in A\langle y_0,\ldots,y_{n-1}\rangle$. We may assume that this holds for $m<n$, in which case, Fact~\ref{wilkieineq} tells us that the class 
\[
\big\{\exp(\delta)y_0^{r_0}\cdots y_{n-1}^{r_{n-1}}: \delta \in \Delta\text{ and }r_0,\ldots,r_{n-1} \in \R\big\}
\] 
is a cross section for $A\langle y_0,\ldots,y_{n-1}\rangle$, so it is enough to show that $y_n \not\asymp a$ for any $a$ in this cross section. Suppose towards contradiction that we can find $\delta \in \Delta$, $r_0,\ldots,r_{n-1} \in \R$, and $u \in K$ with $u\asymp 1$ such that
\[
uy_n\ =\ \exp(\delta)y_0^{r_0}\cdots y_{n-1}^{r_{n-1}}.
\]
Taking logarithms gives
\[
\log u + \log y_n\ =\ \delta +r_0 \log y_0+ \cdots + r_{n-1}\log y_{n-1}.
\]
Set $a:= \delta +r_0 d_0+\cdots+r_{n-1} d_{n-1}-d_n \in A$, so 
\[
a+r_0 \epsilon_0 y_1+ \cdots + r_{n-1} \epsilon_{n-1}y_n - \epsilon_n y_{n+1}\ =\ \log u.
\]
Since $y_m \succ \log y_m \succeq y_{m+1}$ and $y_m \not\asymp a$ for each $m$, we have
\[
y_{n+1}\ \preceq \ a+r_0 \epsilon_0 y_1+ \cdots + r_{n-1} \epsilon_{n-1}y_n - \epsilon_n y_{n+1}\ =\ \log u\ \preceq\ 1,
\]
contradicting that $y_{n+1}$ is infinite. This concludes the proof of Claim 1.

Now, set $\gamma_n:= \log y_n$ for each $n$ and set 
\[
\Gamma \ :=\ \Delta + \R \gamma_0 + \R\gamma_1+\cdots.
\]
\textbf{Claim 2:} We claim that there is an $\cL\uf$-elementary embedding
\[
\jmath: A\langle y_0,y_1,\ldots\rangle\to \R((t^\Gamma))_{\On}
\]
which extends $\imath$ and which sends $y_n = \exp \gamma_n$ to $t^{\gamma_n}$ for each $n$. Suppose that we have an $\cL\uf$-elementary embedding $A\langle y_0,\ldots,y_{n-1}\rangle\to \R((t^\Gamma))_{\On}$ which extends $\imath$ and which sends $y_m$ to $ t^{\gamma_m}$ for each $m<n$. By Claim 1 and Fact~\ref{wilkieineq}, the class
\[
\big\{\exp(\delta)y_0^{r_0}\cdots y_{n-1}^{r_{n-1}}: \delta \in \Delta\text{ and }r_0,\ldots,r_{n-1} \in \R\big\}
\] 
is a cross section for $A\langle y_0,\ldots,y_{n-1}\rangle$. Given $\delta \in \Delta$ and $r_0,\ldots,r_{n-1} \in \R$, we have
\begin{align*}
\exp(\delta)y_0^{r_0}\cdots y_{n-1}^{r_{n-1}} < y_n\ &\Longleftrightarrow\ \delta+r_0\log y_0+\cdots+r_{n-1} \log y_{n-1}< \log y_n\\ &\Longleftrightarrow\ \delta+r_0\gamma_0+\cdots+r_{n-1} \gamma_{n-1}< \gamma_n\ \Longleftrightarrow\ t^{\delta+r_0\gamma_0+\cdots+r_{n-1} \gamma_{n-1}}< t^{\gamma_n}.
\end{align*}
Since $y_n \not\asymp a$ for any $a \in A\langle y_0,\ldots,y_{n-1}\rangle$ by Claim 1, this is enough to show that $y_n$ realizes the same cut over $A\langle y_0,\ldots,y_{n-1}\rangle$ that $t^{\gamma_n}$ realizes over its image. By o-minimality, we may extend our embedding $A\langle y_0,\ldots,y_{n-1}\rangle\to \R((t^\Gamma))_{\On}$ to an embedding $A\langle y_0,\ldots,y_n\rangle\to \R((t^\Gamma))_{\On}$ by sending $y_n$ to $t^{\gamma_n}$. This concludes the proof of Claim 2.

Let $B:= A\langle y_0,y_1,\ldots\rangle$. It remains to establish that $(B,\Gamma,\jmath)$ is a transserial development triple. Since the class $\imath(A) \cup \{t^{\gamma_0},t^{\gamma_1},\ldots\}$ is truncation closed and 
\[
\jmath(B)\ =\ \imath(A)\langle t^{\gamma_0},t^{\gamma_1},\ldots\rangle,
\]
the image $\jmath(B)$ is truncation closed by Lemma~\ref{tclosed}. Let $n$ be fixed. By definition, we have
\[
\jmath(\exp \gamma_n)\ =\ \jmath(y_n)\ =\ t^{\gamma_n}. 
\]
Since $\gamma_n= \log y_n = \epsilon_n y_{n+1}+ d_n$, we have
\[
\jmath(\gamma_n) \ =\ \epsilon_n \jmath(y_{n+1})+ D^\imath_A(\log y_n)\ =\ \epsilon_n t^{\gamma_{n+1}}+ D^\imath_A(\log y_n).
\]
Let $\delta \in\supp\big(D^\imath_A(\log y_n)\big)$. We claim that $\delta> \gamma_{n+1}$. By definition of $D^\imath_A(\log y_n)$, we have
\[
\log y_n - \imath\inv\big(D^\imath_A(\log y_n)\big) \ \prec\ \imath\inv(t^\delta)\ =\ \exp \delta,
\]
so $y_{n+1} \prec \exp \delta$. Taking logarithms gives $\gamma_{n+1} =\log y_{n+1} < \delta$, as desired. Since $\gamma_{n+1}> 0$, this tells us that $\jmath(\gamma_n)$ is purely infinite. It follows that $\imath(\gamma)$ is purely infinite for each $\gamma \in \Gamma$, so $(B,\Gamma,\jmath)$ is a development triple. Moreover, since $\gamma_{n+1}\not\in \Delta$, this also tells us that $D^\imath_A(\log y_n)$ is the maximal truncation of $\jmath(\gamma_n)$ contained in $\R((t^\Delta))_{\On}$, that is, 
\[
D^\imath_A(\log y_n) \ =\ \jmath(\gamma_n)_{\R((t^\Delta))_{\On}}.
\]
Thus,
\[
\big|\log t^{\gamma_n}- (\log t^{\gamma_n})_{\R((t^\Delta))_{\On}}\big|\ =\ \big|\log t^{\gamma_n}- \jmath(\gamma_n)_{\R((t^\Delta))_{\On}}\big|\ =\ \big|\log t^{\gamma_n}- D^\imath_A(\log y_n) \big|\ =\ t^{\gamma_{n+1}}
\]
for each $n$, and we see that $(t^{\gamma_n})_{n\in \N}$ is a $\Delta$-atomic $\Delta$-path. Since $\R((t^\Delta))_{\On}$ is a transserial Hahn field and $\Gamma = \Delta_{(t^{\gamma_n})}$, Lemma~\ref{lem-existsatomic2} gives that $\R((t^\Gamma))_{\On}$ is also a transserial Hahn field.
\end{proof}

\begin{corollary}
\label{anexpembedding}
$K$ admits a truncation closed, cross sectional $\cL\uf$-elementary logarithmic field embedding into a transserial Hahn field $\R((t^\Gamma))_{\On}$.
\end{corollary}
\begin{proof}
We need to find a transserial development triple $(K,\Gamma,\jmath)$, for then $\jmath:K\to \R((t^\Gamma))_{\On}$ will be such an embedding. By Lemma~\ref{expdevelopmentclosed}, any such triple is an $\exp$-development triple. Note that $(\R,\{0\}, \operatorname{id})$ is a transserial development triple, where $ \operatorname{id}$ is the identity map on $\R$, so it suffices to show that any transserial development triple $(A,\Delta,\imath)$ with $A \neq K$ can be extended to a strictly larger transserial development triple. Let $(A,\Delta,\imath)$ be given. If $D_A(y) \not\in A$ for some $y \in K\setminus A$, then we use Lemma~\ref{anexpembedding}. If $(A,\Delta,\imath)$ is not an $\exp$-development triple, then there is $a \in A$ with $a- \delta \succ 1$ for all $\delta \in \Delta$, so we may use Lemma~\ref{anexpembedding2}. If $(A,\Delta,\imath)$ is an $\exp$-development triple and $D_A(x) \in A$ for all $x \in K\setminus A$, then we take any $y \in K\setminus A$ and apply Proposition~\ref{molecularext}.
\end{proof}

\begin{proof}[Proof of Theorem~\ref{initialexponentialWeier}]
By Corollary~\ref{anexpembedding}, we may identify $K$ with a truncation closed, cross sectional $\cL\uf$-elementary logarithmic subfield of a transserial Hahn field $\R((t^\Gamma))_{\On}$. Corollary~\ref{corembedhahn} provides an initial transserial embedding $\R((t^\Gamma))_{\On}\to \No$. Since transserial embeddings preserve infinite sums and logarithms, the restriction of this embedding to $K$ is an $\cL_{\cF\df,\exp}$-embedding. The restriction of this embedding to $K$ is initial by Proposition~\ref{Prop. 2} and $\cL\ufexp$-elementary by Proposition~\ref{qeexp}.
\end{proof}

%-----------------------%
\subsection{Examples}
%-----------------------%
We collect below some consequences of Theorem~\ref{initialexponentialWeier}.
First, if $\cF = \cC^\omega_r$, then $T\uf$ defines a convergent Weierstrass system, so we have the following:

\begin{corollary}\label{allanalytic}
Let $\R\anexp$ be the expansion of $\R$ by all restricted analytic functions and the total exponential function and let $\cL\anexp$ be the corresponding language.
If $K$ is elementarily equivalent to $\R\anexp$ then $K$ admits an initial $\cL\anexp$-elementary embedding into $\No$.
\end{corollary}

By~\cite[5.4]{vdD}, the collection of all \emph{differentially algebraic} analytic functions which converge in a neighborhood of 0 form a convergent Weierstrass system. This provides another example:

\begin{corollary}
Let $\R_{\operatorname{da},\exp}$ be the expansion of $\R$ by all differentially algebraic restricted analytic functions and the total exponential function and let $\cL_{\operatorname{da},\exp}$ be the corresponding language.
If $K$ is elementarily equivalent to $\R_{\operatorname{da},\exp}$ then $K$ admits an initial $\cL_{\operatorname{da},\exp}$-elementary embedding into $\No$.
\end{corollary}

By~\cite{vdD}, if $\cF = \big\{\overline{\exp},\overline{\sin},r\in \R\big\}$, then $T\uf$ defines a convergent Weierstrass system (where, as the reader will recall, $\overline{\exp}$ and $\overline{\sin}$ are the restrictions of $\exp$ and $\sin$ to the interval $[-1,1]$). The domains of $\exp$ and $\sin$ here don't matter, so long as they are closed intervals. This gives us the following:

\begin{corollary}\label{trigexpcor}
Let $\R_{\operatorname{trig},\exp}$ be the expansion of $\R$ by $\sin\restriction_{[0,2\pi]}$, a constant for each $r \in \R$, and the total exponential function and let $\cL_{\operatorname{trig},\exp}$ be the corresponding language.
If $K$ is elementarily equivalent to $\R_{\operatorname{trig},\exp}$ then $K$ admits an initial $\cL_{\operatorname{trig},\exp}$-elementary embedding into $\No$.
\end{corollary}

The method that van den Dries uses in the case $\cF = \big\{\overline{\exp},\overline{\sin},r\in \R\big\}$ has been generalized by Sfouli, who provides sufficient conditions on $\cF$ under which $T\uf$ defines a convergent Weierstrass system~\cite{S}:

\begin{lemma}[Sfouli]
Suppose that $\cF$ satisfies the following two properties:
\begin{enumerate}[(i)]
\item If $\bar{f}:I^n\to \R$ is in $\cF$, then there is $\bar{g}:I^n\to \R$ in $\cF$ such that either $\bar{f}+i\bar{g}$ or $\bar{g}+i\bar{f}$ is holomorphic on the interior of $I^{2n}\subseteq \C^n$.
\item If $\bar{f}:I^n\to \R$ is in $\cF$, then there is $\epsilon \in (0,1)$ and $\bar{g}:I^n\to \R$ in $\cF$ such that $\bar{f}(x) = \bar{g}(\epsilon x)$ for all $x \in I^n$.
\end{enumerate}
Then $T\uf$ defines a convergent Weierstrass system.
\end{lemma}

Sfouli goes on to show that the family $\cF_{\operatorname{har}}$ of all restricted harmonic functions $\bar{f}:I^2\to \R$ satisfies these properties. While $\exp\restriction_{[-1,1]}$ is not in $\cF_{\operatorname{har}}$, it is in $\cF_{\operatorname{har}}\df$ since it can be obtained by evaluating the harmonic function $e^x\cos(y)$ at $y=0$. Thus, we have the following:

\begin{corollary}
Let $\R_{\operatorname{har},\exp}$ be the expansion of $\R$ by all restricted harmonic functions and the total exponential function and let $\cL_{\operatorname{har},\exp}$ be the corresponding language.
If $K$ is elementarily equivalent to $\R_{\operatorname{har},\exp}$ then $K$ admits an initial $\cL_{\operatorname{har},\exp}$-elementary embedding into $\No$.
\end{corollary}

%------------------------------------------------%
\section{Trigonometric fields and surcomplex exponentiation} \label{sec:TE}
%------------------------------------------------%
At the \emph{mini-workshop on surreal numbers, surreal analysis, Hahn fields and derivations} held in Oberwolfach in 2016, the following question was raised: ``Let $i=\sqrt{-1}$. Is there a good way to introduce $\sin$ and $\cos$ on $\No$ and an exponential map on $\No[i]$?''~\cite[page 3315]{BEK}. In this section we make some observations related to this question. 

Let $T_{\operatorname{trig}}$ be the theory of the real field expanded by $\sin\restriction_{[0,2\pi]}$ and $\cos\restriction_{[0,2\pi]}$. We call a model 
\[
\big(K,\sin\restriction_{[0,2\pi]},\cos\restriction_{[0,2\pi]}\big) \models T_{\operatorname{trig}}
\]
a \textbf{trigonometric ordered field}. Let $K$ be such a field. Then $K$ is real closed, so there is a discrete subring $Z \subseteq K$ such that for all $a\in K$ there is a $d \in Z$ with $d\leq a < d+1$. Following tradition, we call $Z$ an \textbf{integer part} of $K$. Using this integer part, we may extend sine and cosine to all of $K$ by setting
\[
\sin (a + 2\pi d)\ :=\ \sin a,\qquad\cos (a + 2\pi d)\ :=\ \cos a
\]
where $a \in [0, 2\pi)$ and where $d \in Z$. Since $K$ may have many integer parts, the extension of $\sin$ and $\cos$ to $K$ is not necessarily unique (indeed, if $\sin_1$ and $\sin_2$ arise from different integer parts, then they have different zero classes). However, in the case that $K$ is an initial trigonometric subfield of $\No$, there is a canonical choice of integer part, in view of the following lemma:

\begin{lemma}
Let $K$ be an initial subfield of $\No$. Then $\Oz \cap K$ is the unique integer part of $K$ which is an initial subtree of $\No$.
\end{lemma}
\begin{proof}
By~\cite[Theorem 20]{EH5}, $\Oz\cap K$ is an integer part of $K$ and an initial subtree of $\No$. Let $Z$ be an integer part of $K$ which is an initial subtree of $\No$. By~\cite[Theorem 6.5]{EK}, an initial subring of $\No$ is discrete if and only if it is a subring of $\Oz$, so $Z$ is a subring of $\Oz\cap K$. To see that $Z = \Oz\cap K$, let $a \in \Oz \cap K$ and take $d \in Z$ with $d\leq a < d+1$. The only element of $\Oz \cap K$ in the interval $(a-1,a]$ is $a$ itself, so $d$ is necessarily equal to $a$. Thus, $a \in Z$.
\end{proof}

\begin{proposition}\label{canonical1}
If $K$ is an initial trigonometric ordered subfield of $\No$, then $K$ admits unique sine and cosine functions arising from an initial integer part, namely, those arising from $\Oz \cap K$.
\end{proposition}

We refer to the sine and cosine functions in Proposition~\ref{canonical1} as the \textbf{canonical sine and cosine functions on $K$}.\footnote{The idea of employing $\Oz$ to define sine and cosine functions for $\No$ appears to originate with Martin Kruskal. The first author thanks Ovidiu Costin for bringing this to his attention.}

%-----------------------%
\subsection{Trigonometric-exponential fields}
%-----------------------%
By $T_{\operatorname{trig},\exp}$ we mean the theory of the real field expanded by $\sin\restriction_{[0,2\pi]}$, the total exponential function, and a constant symbol for each real number. We call a model 
\[
\big(K,\sin\restriction_{[0,2\pi]},\exp\big) \models T_{\operatorname{trig},\exp}
\]
a \textbf{trigonometric-exponential field}. Let $K$ be such a field. Then $\cos\restriction_{[0,2\pi]}$ is 0-definable in $K$, so $K$ may be naturally viewed as an expansion of a trigonometric ordered field.

Since $K$ is real closed, $K[i]$ is algebraically closed (where $i$ is a square root of $-1$). Let
\[
S_{ K} \ :=\ \big\{a+bi:a \in K, b\in [0,2\pi)_K\big\}\ \subseteq\ K[i].
\]
Then $S_{ K} $ admits a natural group structure given by addition of the real parts and addition modulo $2\pi$ of the imaginary parts. More precisely:
\[
(a+bi)+(c+di) \ =\ \left\{\begin{array}{ll} (a+c) + (b+d)i& \text{ if $b+d<2\pi$}\\ (a+c) + (b+d-2\pi)i& \text{ if $b+d\geq2\pi$.}\end{array}\right.
\]
The class $S_{ K} $, as well as its group structure is 0-definable in $K$, where we identify $K[i]$ with $K^2$ via the usual correspondence $a+bi\mapsto (a,b)$. The multiplication on $K[i]$ is also 0-definable in $K$, and we denote by $K[i]^\times$ the multiplicative group $K[i]\setminus \{0\}$. We define a map
\[
E: S_{ K} \to K[i]^\times,\qquad E(x+iy)\ =\ (\exp x)(\sin y+i\cos y).
\]
Then $E$ is also 0-definable in $K$, and so the $\cL_{\operatorname{trig},\exp}$-sentence ``$E$ is a group isomorphism'' is a consequence of $T_{\operatorname{trig},\exp}$, since it is true in $\R$.

We now fix an integer part $Z\subseteq K$ and extend $\sin$ and $\cos$ to all of $K$, as is done above. We define a map:
\[
a+ib\mapsto (\exp a)(\cos b+ i \sin b):K[i]\to K[i]^\times.
\]
Note that this map extends the map $\exp:K\to K^{>0}$, so we denote this map by $\exp$ as well. 
Using the fact that $E$ is a group isomorphism and that sine and cosine are periodic with period $2\pi$, we have the following:

\begin{proposition}
The map $\exp:K[i]\to K[i]^\times$ is a surjective group homomorphism with kernel $ 2\pi iZ$.
\end{proposition}

Since the extension of $\exp$ depends on the extensions of sine and cosine, it depends on the choice of the integer part $Z$. However, we have the following corollary to Proposition~\ref{canonical1}:

\begin{corollary}\label{canonical2}
If $K$ is an initial trigonometric-exponential subfield of $\No$, then $K[i]$ admits a unique exponential function arising from an initial integer part, namely, that arising from $\Oz \cap K$.
\end{corollary}

We refer to the exponential in Corollary~\ref{canonical2} as the \textbf{canonical exponential function on $K[i]$}. Questions still remain as to how ``robust'' the canonical exponential function is:

\begin{open question*}
Let $K$ be an initial trigonometric-exponential subfield of $\No$ and let $\exp$ be any exponential on $K[i]$ which extends the map $E$ on $S_K$ and which satisfies the identity $\lVert\exp(a+bi)\rVert = \exp a$, where $\lVert a+bi\rVert:= \sqrt{a^2+b^2} \in K$ for $a+bi \in K[i]$. Then the class
\[
Z\ :=\ \big\{ z \in K[i]: \exp(y) = 1\, \Rightarrow\, \exp(yz) = 1\text{ for all } y\in K[i]\big\}
\]
is a discrete subring of $K$. Thus, if we require $Z$ to be \emph{initial}, then $Z$ is a subring of $\Oz \cap K$ by~\cite[Theorem 6.5]{EK}. 
Under what conditions is $Z$ an integer part of $K$? That is, under what conditions, including the requirement that $Z$ be initial, ensure that $\exp$ coincides with the canonical exponential function on $K[i]$?
\end{open question*}

By Corollary~\ref{trigexpcor} any trigonometric-exponential field $K$ admits an initial embedding into $\No$. However, this initial embedding may not be unique, so there may not be a way to equip $K[i]$ with a canonical exponential function in general. 

\begin{open question*}
Given a trigonometric-exponential field $K$ and an initial embedding $\imath:K\to \No$, let $Z_\imath:= \imath\inv(\On\cap K)$ and let $\exp_\imath$ be the exponential on $K[i]$ arising from the integer part $Z_\imath$. If $\imath,\jmath$ are different initial embeddings $K\to \No$, then how different can the exponentials $\exp_\imath$ and $\exp_\jmath$ be? Are they isomorphic? Do the exponential fields $(K[i],\exp_\imath)$ and $(K[i],\exp_\jmath)$ have the same first order theory?
\end{open question*}

%-----------------------%
\subsection{Surcomplex exponentiation}
%-----------------------%
For each ordinal $\alpha$, let
\[
\No(\alpha)\ :=\ \big\{x \in \No : \rho_{\No}(x)<\alpha\big\}.
\]

By Proposition~\ref{DE1}, $\No$ is a trigonometric-exponential field. Moreover, by~\cite[Corollary 5.5]{vdDE}, $\No(\alpha)$ is a trigonometric-exponential subfield of $\No$ whenever $\alpha$ is an \emph{epsilon number} (that is, whenever $\omega^\alpha = \alpha$). Furthermore, $\No(\alpha)$ is initial for each $\alpha$. Thus, in virtue of Corollary~\ref{canonical2}, we have the following:

\begin{theorem}
The surcomplex numbers $\No[i]$ admits a canonical exponential function with kernel $2\pi i\Oz$. Moreover, $\No(\alpha)[i]$ admits a canonical exponential function with kernel $2\pi i\big(\Oz\cap \No(\alpha)\big)$ for each epsilon number $\alpha$.
\end{theorem}

%------------------------------------------------%
\section{Initial embeddings of some distinguished fields of transseries}\label{sec:Tran}
%------------------------------------------------%
Berarducci and Mantova~\cite{BM2} introduced the exponential ordered field $\R\langle\langle\omega\rangle \rangle$ of \emph{omega-series}. It is the smallest exponential subfield of $\No$ containing $\R$ and $\omega$ that is closed under $\exp$, $\log$ and taking infinite sums. They further isolated the exponential subfields $\R((\omega))^{LE}$ and $\R((\omega))^{EL}$ of $\No$ that are isomorphic to the exponential ordered fields of \emph{LE-series}~\cite{DMM, DMM1, Avv} and \emph{EL-series}~\cite{K, KS, KT}, respectively. The system of \emph{LE}-series in turn is isomorphic to the exponential ordered field $\T^{LE}$ of \emph{{logarithmic-exponential} transseries}. In this section, we prove that $\R\langle\langle\omega\rangle \rangle$, $\R((\omega))^{LE}$ and $\R((\omega))^{EL}$, which are models of $T (\R\anexp)$, and hence, models of $T_{\operatorname{trig},\exp}$, are initial. The methods employed for the proofs are different from the methods used in \S\ref{sec:IE} and \S\ref{sec:WS}, and only depend on material from the preliminary sections. 

For a subclass $X \subseteq \No$, we let $X^{\rc}$ be the smallest real closed subfield of $\No$ containing $X$. We say that $X$ is \textbf{closed under sums} if the sum of every summable sequences of elements in $X$ is contained in $X$. We let $X^\Sigma$ be smallest subclass of $\No$ which contains $X$ and which is closed under sums. For the rest of this section, let $K$ be an initial subfield of $\No$.

\begin{lemma}
\label{sumclosure}
$K^\Sigma$ is an initial subfield of $\No$. If $K$ is also real closed, then so is $K^\Sigma.$
\end{lemma}
\begin{proof}
Using Neumann's Lemma, e.g.\!~\cite[pages 260-261]{AL}, we see that $K^\Sigma$ is indeed an ordered field. Let $\Gamma$ be the value group of $K$, so $\Gamma$ is initial and $K$ is a truncation closed, cross sectional subfield of $\R((\omega^\Gamma))_{\On}$ by~\cite[Theorem 18]{EH5}. Then $K^\Sigma$ is also a truncation closed, cross sectional subfield of $\R((\omega^\Gamma))_{\On}$, so $K^\Sigma$ is initial as well, again by~\cite[Theorem 18]{EH5}. If, in addition, $K$ is real closed, then $\R_K$ is real closed and $\Gamma$ is divisible. Thus, $K^\Sigma = \R_K((\omega^\Gamma))_{\On}$ is also real closed.
\end{proof}

\begin{lemma}
\label{Ehrlich initial 1}
If $K$ is real closed and $X$ is a subset of $\No$ each of whose members is the simplest element of $\No$ that realizes a cut in $K$, then $(K\cup X)^{\rc}$ is initial.
\end{lemma}
\begin{proof}
This readily follows by iterating the result for the case where $X$ is a singleton established by the first author in~\cite[pages 8, 37-38, Theorem 7]{EH8}. 
\end{proof}

\begin{lemma}
\label{expstep}
If $K$ is an initial real closed subfield of $\No$ then $\big(K \cup \exp(K)\big)^{\rc}$ is initial.
\end{lemma}
\begin{proof}
We prove this by induction on the tree-rank of $x \in K$. If $\rho_{\No}(x) = 0$, then $x = 0$ and $K \cup \{\exp(0)\} = K$ is initial by assumption. Fix $x \in K$ with $\rho_{\No}(x) = \alpha>0$ and suppose that 
\[
K_\alpha\ :=\ \big(K \cup \{\exp y:y \in K,\ \rho_{\No}(y)<\alpha\}\big)^{\rc}
\]
is initial. Then since $[x-x^L]_n$, $[x-x^R]_{2n+1}$, $\frac{1}{[x^R-x]_n}$, $\frac{1}{[x^L-x]_{2n+1}}$, $\exp x^L$, and $\exp x^R$ are all in $K_\alpha$, we see that $\exp x$ is the simplest element in $\No$ realizing a cut in $K_\alpha$. By Lemma~\ref{Ehrlich initial 1},
\[
K_{\alpha+1}\ =\ \big(K_\alpha \cup \{\exp x:x \in K,\ \rho_{\No}(x) = \alpha\}\big)^{\rc}
\]
is initial. Taking the union of the $K_{\alpha+1}$ over all $\alpha \in \{\rho_{\No}(x):x \in K\}$, we see that $\big(K \cup \exp(K)\big)^{\rc}$ is initial.
\end{proof}

\begin{lemma}
\label{logstep}
Suppose that $K$ is an initial real closed subfield of $\No$, that $K$ contains $\R$, and that $K$ is closed under sums. Then $\big(K \cup \log(K^{>0})\big)^{\rc}$ is initial.
\end{lemma}
\begin{proof}
Let $\Gamma$ be the value group of $K$. Then $\Gamma$ is an initial divisible subgroup of $\No$. We first claim that $\log(K^{>0})\subseteq \big(K \cup \log(\omega^{\Gamma})\big)^{\rc}$. For $x \in K^{>0}$, we may write $x = r\omega^\gamma(1+\epsilon)$ for some $r \in \R^{>0}$, some $\gamma \in \Gamma$, and some $\epsilon \in K$ with $\epsilon \prec 1$. We have $\ln r \in \R \subseteq K$ and, since $K$ is closed under sums, we have
\[
\log \epsilon\ =\ \sum_{k=1}^{\infty } \frac{ (-1)^{k-1} \epsilon^{k}}{k}\in K.
\]
Thus, $\log x \in K+ \log(\omega^{\Gamma})\subseteq \big(K \cup \log(\omega^{\Gamma})\big)^{\rc}$.

We will now show that $ \big(K \cup \log(\omega^{\Gamma})\big)^{\rc}$ is initial by induction on the simplicity of $\gamma \in \Gamma$. If $\rho_{\No}(\gamma) = 0$, then $\omega^\gamma = 1$ and so $K \cup \{\log(1)\} = K$ is initial by assumption. Fix $\gamma \in \Gamma$ with $\rho_{\No}(\gamma) = \alpha>0$ and suppose that
\[
K_\alpha\ :=\ \Big(K\cup \big\{\log \omega^\delta :\delta \in \Gamma,\ \rho_{\No}(\delta)<\alpha\big\}\Big)^{\rc}
\]
is initial. Using that $K$ is cross sectional and that $\Gamma$ is divisible, we see that $\omega^{\gamma^L}$, $\omega^{\gamma^R}$, $\omega^\frac{\gamma^R-\gamma}{n}$, and $\omega^\frac{\gamma-\gamma^L}{n}$ are all in $K\subseteq K_\alpha$. Since $\log \omega^{\gamma^L}$, $\log \omega^{\gamma^R}$ are also in $K_\alpha$, we have $\log \omega^\gamma$ is the simplest element in $\No$ realizing a cut in $K_\alpha$. By Lemma~\ref{Ehrlich initial 1},
\[
K_{\alpha+1}\ =\ \Big(K_\alpha \cup \big\{\log \omega^\gamma:\gamma \in \Gamma,\ \rho_{\No}(\gamma)=\alpha\big\}\Big)^{\rc}
\]
is initial. By taking the union of the $K_{\alpha+1}$ over $\alpha \in \{\rho_{\No}(\gamma):\gamma \in \Gamma\}$, we deduce that $\big(K \cup \log(\omega^\Gamma)\big)^{\rc}$ is initial.
\end{proof}

The following definitions are due to Berarducci and Mantova~\cite{BM2}:

\begin{definition}\
\begin{enumerate}[(i)]
\item $\R\langle\langle\omega\rangle \rangle$ is the smallest subfield of $\No$ containing $\R(\omega)$ and closed under $\exp$, $\log$, and sums.
\item $\R((\omega))^{LE}$ is the union $\bigcup_{n} X_n$ where $X_0 = \R(\omega)$ and 
\[
X_{n+1}\ =\ \big(X_n \cup \exp (X_n) \cup \log(X_n^{>0})\big)^\Sigma.
\]
\item $\R((\omega))^{EL}$ is the union $\bigcup_{n} Y_n$ where $Y_0 = \R\big(\omega,\log(\omega),\log_2(\omega),\ldots\big)$ and 
\[
Y_{n+1}\ =\ \big(Y_n \cup \exp (Y_n) \cup \log(Y_n^{>0})\big)^\Sigma.
\] 
\end{enumerate}
\end{definition}

\begin{theorem}
The fields $\R\langle\langle\omega\rangle \rangle$, $\R((\omega))^{LE}$, and $\R((\omega))^{EL}$ are all initial.
\end{theorem}
\begin{proof}
Lemmas~\ref{sumclosure}--\ref{logstep} show that $\R\langle\langle\omega\rangle \rangle$ is initial. 
Since $\R((\omega))^{LE}$ is real closed, it is also equal to the union $\bigcup_{n} K_n$ where 
\[
K_0\ :=\ \big(\R \cup \{\omega\}\big)^{\rc,\Sigma},\qquad K_{n+1} \ :=\ \big(K_n \cup \exp(K_n) \cup \log(K_n^{>0})\big)^{\rc,\Sigma}.
\]
Lemmas~\ref{sumclosure}--\ref{logstep} likewise show that this is an initial subfield of $\No$. As for $\R((\omega))^{EL}$, we first build a subfield $L_0$ by setting
\[
L_{0,0}\ :=\ \big(\R \cup \{\omega\}\big)^{\rc,\Sigma},\qquad L_{0,m+1}\ :=\ \big(L_m \cup \log(L_m^{>0})\big)^{\rc,\Sigma},\qquad L_0\ :=\ \Big(\bigcup_m L_{0,m+1}\Big)^{\rc,\Sigma}.
\]
Note that $\R\big(\omega,\log(\omega),\log_2(\omega),\ldots\big)\subseteq L_0 \subseteq\R((\omega))^{EL}$. Now, we repeat the same process above: that is, we set
\[
L_{n+1}\ :=\ \big(L_n \cup \exp(L_n) \cup \log(L_n^{>0})\big)^{\rc,\Sigma}
\]
and observe that $\R((\omega))^{EL} = \bigcup_nL_n$.
\end{proof}

\begin{corollary}
$\R\langle\langle\omega\rangle \rangle$, $\R((\omega))^{LE}$, and $\R((\omega))^{EL}$ are all models of $T_{\operatorname{trig},\exp}$. Thus, by Corollary~\ref{canonical2}, these fields all admit a canonical exponential function on their algebraic closures.
\end{corollary}
\begin{proof}
$\R\langle\langle\omega\rangle \rangle$, $\R((\omega))^{LE}$, and $\R((\omega))^{EL}$ are all closed under exponentiation by definition. Additionally, $\R\langle\langle\omega\rangle \rangle$, $\R((\omega))^{LE}$, and $\R((\omega))^{EL}$ are all increasing unions of Hahn fields (this is by definition for $\R((\omega))^{LE}$ and $\R((\omega))^{EL}$, and this is the case for $\R\langle\langle\omega\rangle \rangle$ by~\cite[Remark 4.24 and Corollary 4.28]{BM2}). Since each restricted analytic function on $\No$ agrees with its Taylor series expansion, this gives that $\R\langle\langle\omega\rangle \rangle$, $\R((\omega))^{LE}$, and $\R((\omega))^{EL}$ are all closed under restricted analytic functions and so these three fields are elementary $\cL\anexp$-substructures of $\No$ by Proposition~\ref{qeexp}. In particular, they are all models of $T_{\operatorname{trig},\exp}$.
\end{proof}

By~\cite[Theorem 4.11]{BM2}, $\R((\omega))^{LE}$ is the image of the canonical embedding $\imath:\T^{LE} \rightarrow \No$ which sends $x$ to $\omega$ (see~\cite{Avv} for an explicit definition of $\imath$). Thus, we have the following:

\begin{corollary}
The image of the canonical embedding $\imath:\T^{LE} \rightarrow \No$ is initial.
\end{corollary}

%------------------------------------------------%	

%------------------------------------------------%

\begin{thebibliography}{99}

\bibitem{AL}
N.~Alling, \emph{Foundations of Analysis Over Surreal Number Fields}, North-Holland Publishing Company, Amsterdam (1987).

\bibitem{ADH}
M.~Aschenbrenner, L.~van den Dries, and
 J.~van der Hoeven, {\em Asymptotic Differential Algebra and Model Theory of
 Transseries}, Number 195 in Annals of Mathematics Studies. Princeton University Press, 2017.
 
\bibitem{ADH2}
\bysame \emph{Numbers, germs and transseries}, in \emph{Proceedings of the International Congress of Mathematicians, Rio De Janeiro, 2018, Volume 2}, edited by B.~Sirakov, P.~N.~de Souza and M.~Viana, World Scientific Publishing Company, NJ (2018), pp. 19-43.

\bibitem{Avv}
\bysame \emph{The surreal numbers as a universal H-field}, J. Eur. Math. Soc. 21 (2019), no. 4, 
pp. 1179-1199.

\bibitem{BH}
V.~Bagayoko and J.~van der Hoeven, \emph{Surreal Substructures}, preprint (2019), \url{https://hal.archives-ouvertes.fr/hal-02151377}.

\bibitem{BH2}
\bysame \emph{The hyperserial field of surreal numbers}, preprint (2021), \url{https://hal.archives-ouvertes.fr/hal-03232836}.

\bibitem{BHK}
V.~Bagayoko, J.~van der Hoeven, and E.~Kaplan \emph{Hyperserial fields}, preprint (2021), \url{https://hal.archives-ouvertes.fr/hal-03196388}.

\bibitem{BHM}
V.~Bagayoko and J.~van der Hoeven, and V.~Mantova \emph{Defining a surreal hyperexponential}, preprint (2020), \url{https://hal.archives-ouvertes.fr/hal-02861485}.

\bibitem{BEK}
A.~Berarducci, P.~Ehrlich and S.~Kuhlmann, \emph{Mini-workshop: surreal numbers, surreal analysis, Hahn fields and derivations. Abstracts from the mini-workshop held December 18-23, 2016}. Organized by Alessandro Berarducci, Philip Ehrlich and Salma Kuhlmann. Oberwolfach Rep. 13 (2016), no. 4, pp. 3313-3372.

\bibitem{BM}
A.~Berarducci and V.~Mantova, \emph{Surreal numbers, derivations and transseries}, J. Eur. Math. Soc. (JEMS), 20:339--390, 2018.

\bibitem{BM2}
\bysame \emph{Transseries as germs of surreal functions}, Trans. Amer. Math. Soc. 371 (2019), no. 5, pp. 3549-3592.

\bibitem{CO}
J.~H.~Conway, \emph{On Numbers and Games} (\emph{Second Edition}), A K Peters, Ltd., Natick, Massachusetts (2001).

\bibitem{CE}
O.~Costin and P.~Ehrlich, \emph{Integration on the surreals: a conjecture of Conway, Kruskal and Norton}, in~\cite{BEK}, pp. 3364-3365.

\bibitem{CEF}
O.~Costin, P.~Ehrlich and H.~Friedman, \emph{Integration on the surreals: a conjecture of Conway, Kruskal and Norton}, preprint (24 Aug 2015), arXiv.1334466.

\bibitem{D}
P.~D'Aquino, J.~Knight, S.~Kuhlmann and K.~Lange, \emph{Real closed exponential fields}, Fund. Math. 219 (2012), pp. 163-190.

\bibitem{DD}
J.~Denef and L.~van den Dries, \emph{{$p$}-adic and real subanalytic sets}, Ann. of Math. (2) \textbf{128} (1988), no.~1, 79--138.

\bibitem{DL}
J.~Denef and L.~Lipshitz, \emph{Ultraproducts and approximation in local rings. II}, Math. Ann. 253 (1980), pp. 1-28. 

\bibitem{vdD86}
L.~van den Dries, \emph{A generalization of the {T}arski-{S}eidenberg theorem, and some nondefinability results}, Bull. Amer. Math. Soc. (N.S.), 15:189--193, 1986.

\bibitem{vdD}
\bysame \emph{On the elementary theory of restricted elementary functions}, J. Symbolic Logic 53 (1988), pp. 796-808.

\bibitem{vdDE}
L.~van den Dries and P.~Ehrlich, \emph{Fields of surreal numbers and exponentiation}, Fund. Math. 167 (2001), pp. 173-188; erratum, ibid. 168 (2001), pp. 295-297. 

\bibitem{DMM}
L.~van den Dries, A.~Macintyre and D.~Marker, \emph{The elementary theory of restricted analytic fields with exponentiation}, Ann. of Math. 140 (1994), 183-205.

\bibitem{DMM1}
\bysame \emph{Logarithmic-exponential power series}, J. Lond. Math. Soc. 56 (1997), pp. 417-434.

\bibitem{EH1}
P.~Ehrlich, \emph{An alternative construction of Conway's ordered field No}, Algebra Universalis 25 (1988), pp. 7-16; errata, ibid. 25 (1988), p. 233.

\bibitem{EH2}
\bysame \emph{Absolutely saturated models}, Fund. Math. 133 (1989), pp. 39-46.

\bibitem{EH4}
\bysame \emph{All number great and small, in Real Numbers, Generalizations of the Reals, and Theories of Continua}, Philip Ehrlich (editor), Kluwer Academic Publishers, Dordrecht, Netherlands, (1994), pp. 239-258.

\bibitem{EH5}
\bysame \emph{Number systems with simplicity hierarchies: a generalization of Conway's theory of surreal numbers}, J. Symbolic Logic 66 (2001), pp. 1231-1258. \emph{Corrigendum}, J. Symbolic Logic 70 (2005), p. 1022.

\bibitem{EH6}
\bysame \emph{Surreal numbers: an alternative construction (Abstract)}, Bull. Symbolic Logic 8 (2002), p. 448.

\bibitem{EH7}
\bysame \emph{Conway names, the simplicity hierarchy and the surreal number tree}, J. Logic and Analysis 3 (2011), no. 1, pp. 1-26.

\bibitem{EH8}
\bysame \emph{The absolute arithmetic continuum and the unification of all numbers great and small}, Bull. Symbolic Logic 18 (2012), pp. 1-45.

\bibitem{EH9}
\bysame \emph{Surreal ordered exponential fields}, in~\cite{BEK}, pp. 3324-3325.

\bibitem{EK}
P.~Ehrlich and E.~Kaplan, \emph{Number systems with simplicity hierarchies: a generalization of Conway's theory of surreal numbers II}, J. Symbolic Logic 83 (2018), pp. 617-633.

\bibitem{F}
A.~Fornasiero, \emph{Integration on surreal numbers}, Ph.D. thesis, University of Edinburgh (2004).

\bibitem{F1}
\bysame \emph{Embedding Henselian fields into power series}, Journal of Algebra, 304 (2006), pp. 112-156.

\bibitem{F2}
\bysame \emph{Initial embeddings in $\No$ of models of $T_{an}(\exp)$}, (July 31, 2013).

\bibitem{Ga}
A.~Gabrielov, \emph{Projections of semianalytic sets}, Funkcional. Anal. i Prilo\v{z}en., 2(4):18--30, 1968.

\bibitem{GO}
H.~Gonshor, \emph{An Introduction to the Theory of Surreal Numbers}, Cambridge University Press, Cambridge (1986).

\bibitem{H}
H.~Hahn, \emph{\"Uber die nichtarchimedischen Gr\"ossensysteme}, Sitzungsberichte der Kaiserlichen Akademie der Wissenschaften, Wien, Mathematisch-Naturwissenschaftliche Klasse 116 (Abteilung IIa) (1907), pp. 601-655.

\bibitem{vdH}
J.~van der Hoeven, {\em Asymptotique Automatique}, PhD thesis, \'{E}cole Polytechnique, 1997.

\bibitem{KKS}
F.-V. Kuhlmann, S.~Kuhlmann, and S.~Shelah, \emph{Exponentiation in power series fields}, Proc. Amer. Math. Soc., 125(11):3177--3183, 1997.

\bibitem{K}
S.~Kuhlmann, \emph{Ordered Exponential Fields}, Fields Institute Monographs.12, Amer. Math. Soc. (2000). 

\bibitem{KS}
S.~Kuhlmann and S. Shelah, \emph{$\kappa$-bounded exponential-logarithmic power series fields}, Ann. Pure Appl. Logic
136 (2005), pp. 284-296. 

\bibitem{KT}
S.~Kuhlmann and M.~Tressl, \emph{Comparison of exponential-logarithmic and logarithmic-exponential series},
Math. Logic Q. 58 (2012), pp. 434-448.

\bibitem{ME}
E.~Mendelson, \emph{An Introduction to Mathematical Logic}, fifth ed., CRC Press, Boca Raton (2010).

\bibitem{MR}
M.-H.~Mourgues and J.-P.~Ressayre, \emph{Every real closed field has an integer part}, J. Symbolic Logic 58 (1993), pp. 641-647.

\bibitem{N}
B.~H.~Neumann, \emph{On Ordered Division Rings}, Trans. Am. Math. Soc. 66 (1949), pp. 202-252.

\bibitem{R}
J.-P.~Ressayre, \emph{Integer parts of real closed exponential fields}, in P. Clote and J. Krajicek (eds.) \emph{Arithmetic, Proof Theory, and Computational
Complexity}, Clarendon Press, Oxford, (1993), pp. 278-288.

\bibitem{SC}
M.~Schmeling, \emph{Corps de transs\'eries}, Ph.D. thesis, Universit\'e de Paris 7,
(2001).

\bibitem{S}
H.~Sfouli, \emph{On the elementary theory of restricted real and imaginary parts of holomorphic functions}, Notre Dame J. Form. Log. 53 (2012), no. 1, pp. 67-77. 

 \bibitem{SI}
A.~Siegel, \emph{Combinatorial Game Theory}, American Mathematical Society, Providence RI, 2013.

\end{thebibliography}
\end{document}